%% file: main.tex
\documentclass{amsart}

\input{sections/00-preamble}

\begin{document}

\title{The local Floer cohomology of indicator functions}
\author{Yoel Groman}
\address{ Yoel Groman,
Hebrew University of Jerusalem,
Mathematics Department\\
Email: yoel.groman@mail.huji.ac.il}

\begin{abstract}
Let $K$ be a compact domain with contact-type boundary in a symplectic manifold $M$. We associate an unweighted local Floer cohomology group to each isolated Reeb component of $\partial K$ and show that, in sufficiently small Novikov windows, these groups together with $H^*(K)$ describe the symplectic cohomology of $M$ with support on $K$. The main analytic input is an energy-barrier construction which remains uniform under finite-slope approximations of the indicator Hamiltonian of $K$.

We also study functoriality. For inclusions of contact-type domains we describe the local blocks of the restriction map without requiring the Liouville primitives on the two boundaries to agree. For a monotone isotopy carrying an isolated family of Reeb components, the corresponding block is the local Floer transport map multiplied by the Novikov weight determined by the source-minus-target action difference. These results provide the local Floer-theoretic input for companion work on relative symplectic cohomology near singularities of SYZ fibrations.
\end{abstract}

\maketitle

\input{sections/01-introduction-conventions}
\input{sections/02-floer-periodic-orbits}
\input{sections/03-infinitesimal-symplectic}
\input{sections/04-truncated-unweighted}
\input{sections/05-main-proofs}
\input{sections/07-appendices}
\bibliographystyle{plain}
\bibliography{ref}

\end{document}

%% file: sections/00-preamble.tex
\usepackage{amssymb,amsmath}
\usepackage{comment}
\usepackage{graphicx}
\usepackage{subfigure}
\usepackage{upgreek}
\usepackage{mathtools} 
\usepackage{bbm}
\usepackage{tikz}
\usepackage{pgfplots}
\usepackage{refcount}
\usepackage{enumitem}
\usepackage[all]{xy}
\usetikzlibrary{matrix}
\usetikzlibrary{calc}
\newcommand{\bC}{{\mathbb C}}

\newcommand{\bR}{{\mathbb R}}
\newcommand{\bZ}{{\mathbb Z}}
\newcommand{\bN}{{\mathbb N}}

\newcommand{\bfk}{{\mathbf k}}

\newcommand{\cG}{\mathcal G}
\newcommand{\cA}{\mathcal A}
\newcommand{\cB}{\mathcal B}

\newcommand{\cN}{\mathcal N}
\newcommand{\cM}{\mathcal M}
\newcommand{\cO}{\mathcal O}

\newcommand{\cU}{\mathcal U}
\newcommand{\cV}{\mathcal V}

\newcommand{\cL}{\mathcal L}
\newcommand{\cR}{\mathcal R}

\newcommand{\cH}{\mathcal H}

\DeclareMathOperator{\id}{id}

\DeclareMathOperator{\Per}{Per}

\newtheorem{theorem}{Theorem}
\newtheorem{tm}{Theorem}[section]

\newtheorem{cy}[tm]{Corollary}
\newtheorem{lm}[tm]{Lemma}
\newtheorem{prp}[tm]{Proposition}

\newtheorem{rem}[tm]{Remark}
\newtheorem{df}[tm]{Definition}

\newtheorem{ex}[tm]{Example}

%% file: sections/01-introduction-conventions.tex

\section{Introduction}
Let $M$ be a symplectic manifold which is either closed or geometrically finite type (see \S \ref{SecConventions}), and let $K\subset M$ be an arbitrary compact subset. Hamiltonian Floer theory associates to $K\subset M$  the cohomology group $SH^*_M(K)$, the \emph{symplectic cohomology of $M$ with support on $K$}.  This invariant plays a central role in many recent works. See \cite{GromanVarolgunes2022} and \cite{Groman2025} for applications to the reconstruction problem in mirror symmetry. 
As is typical in Floer theory, it is not usually feasible to compute relative $SH$ from its definition. However, one can hope to make progress by analyzing its underlying chain complex $SC^*_M(K)$ as the deformation of something simpler.  The idea that we pursue in the present paper is that $SC^*_M(K)$ is the deformation of \emph{local Floer cochains} \footnote{There is an unfortunate conflict of terminology, as what is here referred to as relative SH has sometimes been called, including by the author, local SH, whereas the term local Floer cohomology is used in the literature for what we call below the unweighted Floer cohomology of an orbit.}, a version of Floer cohomology which takes into account only Floer trajectories of infinitesimally small energy.  We apply this to give a zeroth order version of mirrror symmetry over the regular locus of an SYZ fibration. The first arXiv version of this paper also contained an application of this to Floer cohomology  near singular SYZ fibers in dimension $2$. This application has been separated into a separate paper.
\subsection{Statement of the main results}
From now on assume that $K\subset M$ has contact
boundary. For simplicity and definiteness we use the semipositive setting with integral coefficients; see Section~\ref{SecConventions}. We assume that its closed Reeb orbits decompose into isolated
path-connected components: each component $\gamma$ has a neighborhood
$\cU_\gamma\subset\cL M$, in the $C^\infty$ topology, containing no other closed
Reeb orbit. We also assume that $c_1(M,\omega)$ vanishes on every closed path of
Reeb orbits and fix a trivialization of the complexified canonical bundle near
each component.

Our first result associates a local invariant to each Reeb component.
\begin{theorem}\label{thmLocalReebIntro}
\begin{enumerate}
\item Every isolated component $\gamma$ has a \emph{local Floer cohomology group}
$SH^*_{uw}(\gamma)$, independent of any choices up to canonical isomorphism. It may be computed from
generators and Floer trajectories in an arbitrarily small admissible neighborhood
of $\gamma$ in the loop space; see Definition~\ref{dfAdmissibleNeighborhood}.

\item An isolated one-parameter family $\Gamma=\{\gamma^\tau\}_{\tau\in[0,1]}$ of Reeb components (called an isolating isotopy in Definition~\ref{dfReebComponentIsotopy}) induces an isomorphism
\[
\iota_{\Gamma}\colon SH^*_{uw}(\gamma^0)\xrightarrow{\sim} SH^*_{uw}(\gamma^1),
\]
which depends only on the homotopy class of the isolating isotopy relative to its endpoints. It is functorial under concatenation of isotopies and equal to the identity for the trivial isotopy. For short enough isotopies, it is similarly computed by considering continuation trajectories contained in a small isolating neighborhood. 
\end{enumerate}
\end{theorem}
The proof is given in \S~\ref{SubsecProofThmLocalReebIntro}.
\begin{rem}
If $\gamma$ is Morse--Bott, then
$SH^*_{uw}(\gamma)$ is, up to the transverse Conley--Zehnder grading shift,
the singular cohomology of $\gamma$ with coefficients in the orientation
local system of the Morse--Bott capping operators; see
Theorem~\ref{tmMorseBottCascades}. This local system need not be
trivial~\cite{BourgeoisOancea2009}. Beyond the Morse--Bott case, it is an
interesting question to give a purely topological description of
$SH^*_{uw}(\gamma)$.
\end{rem}

The complex $SC^*_M(K)$ is defined over the Novikov coefficients $\Lambda$
of~\eqref{eqNovikovField}. Write
$\Lambda_{[a,b)}:=T^{-b}\Lambda_{>0}/T^{-a}\Lambda_{>0}$.
Define
\begin{equation}\label{eqRHdef}
RH^*(K):=H^*(K)\oplus\bigoplus_{\gamma\in\cR(\partial K)}SH^*_{uw}(\gamma),
\end{equation}
where the first summand is regarded as the local Floer cohomology of the constant
orbits in $K$.  Fix a positive contact-type primitive $\theta$ of $\omega$ near $\partial K$ and set
$T_\theta(\gamma):=-\int\gamma^*\theta$, abbreviated to $T(\gamma)$ when
$\theta$ is unambiguous. Write
\[
RH^{\le a}(K;\theta)
:=
H^*(K)\oplus
\bigoplus_{\gamma:\,T_\theta(\gamma)\le a}SH^*_{uw}(\gamma).
\]

\begin{theorem}\label{mainThmA}
There exists an $\hbar>0$ such that, in every Novikov window of size at most
$\hbar$, the chain complex $SC^*_M(K)$ underlying $SH^*_M(K)$ is graded by
$H_1(\partial K;\bZ)$ and filtered by action. For every
period cutoff $a>0$ there is $\delta(a)\in(0,\hbar]$ such that
\[
SH^*_{M,[t,t+\delta(a))}(K)
\simeq
SH^{\le a}_{M,[t,t+\delta(a))}(K)
\oplus
SH^{>a}_{M,[t,t+\delta(a))}(K),
\]
and the first summand has a canonical identification
\begin{equation}\label{eqSHDeltaA}
SH^{\leq a}_{M,[t,t+\delta(a))}(K)\;\simeq\;RH^{\le a}(K;\theta)\otimes\Lambda_{[t,t+\delta(a))},
\end{equation}
for every $t\in\bR$.
\end{theorem}

We next consider an inclusion $K_1\subset K_2$ of compact domains with contact
boundary. Crucially, we do not  require their Liouville primitives to agree. See Figure \ref{FigAdmInc0}.
\begin{figure}

  \begin{tikzpicture}[scale=1.5]
    \def\edge{2}
    \draw[rounded corners, thick] (-\edge,-\edge) 
  rectangle (\edge,\edge);
    \foreach \x in {-1.8,-1.4,...,1.9} {
      \foreach \y in {-1.8,-1.4,...,1.9} {
        \draw[->] (\x,\y) -- (\x + 0.2*\x, \y + 
  0.2*\y);
      }
    }
    \node at (0,0) {$K_2$};
  
    \draw[rounded corners, thick, blue] (0.5,-
  \edge) rectangle (2,1);
    \node[blue] at (1.2,-0.4) {$K_1$};

  \end{tikzpicture}
  \caption{A non-exact inclusion with partially 
  overlapping boundaries}\label{FigAdmInc0}
  \end{figure}

A Reeb component
$\gamma\subset\partial K_2$ is:
\begin{enumerate}
\item \emph{strictly common} if
$\gamma\subset\operatorname{int}(\partial K_1\cap\partial K_2)$;
\item \emph{non-common} if the image of $\gamma$ is contained in
$M\setminus K_1$;
\item \emph{weakly common} if it is not strictly common and admits an action-preserving radial-front continuation to a Reeb component of $\partial K_1$; see Definition~\ref{dfWeaklyCommonComponent}.
\end{enumerate}
For every weakly common component, a radial-front datum is understood
to have been chosen as part of the admissibility data.
We call strictly or weakly common components \emph{common}. For every common
component $\gamma$, let $\Gamma_\gamma$ denote the chosen radial-front family in
the weakly common case and the stationary family in the strictly common case, and set
$\iota_\gamma:=\iota_{\Gamma_\gamma}\colon SH^*_{uw}(\gamma)\to
SH^*_{uw}(\Gamma_\gamma^1)$. Thus $\iota_\gamma=\id$ in the strictly common
case by Theorem~\ref{thmLocalReebIntro}. The three classes need not exhaust the
Reeb components.

\begin{theorem}\label{mainThmB}
Let $K_1\subset K_2$ be an inclusion of compact domains with contact boundary.
Fix positive contact-type primitives $\theta_i$ near $\partial K_i$ and period cutoffs $a,b>0$ such
that every common component $\gamma\subset\partial K_2$ with
$T_{\theta_2}(\gamma)\le a$ corresponds to a component
$\gamma^1\subset\partial K_1$ with $T_{\theta_1}(\gamma^1)\le b$. Then, for every sufficiently small $\delta>0$, the restriction map below, expressed using the decompositions of Theorem~\ref{mainThmA} for $(K_2,\theta_2,a)$ and $(K_1,\theta_1,b)$ and filtered by $T_{\theta_2}$,
\[
SH^{\le a}_{M,[0,\delta)}(K_2)
\longrightarrow
SH^*_{M,[0,\delta)}(K_1)
\]
has the following blocks:
\begin{enumerate}
\item it is zero on every non-common Reeb summand;
\item on a common Reeb summand it is
$\iota_\gamma\otimes\id_{\Lambda_{[0,\delta)}}$ into the corresponding target
summand;
\item on $H^*(K_2)\otimes\Lambda_{[0,\delta)}$ it is ordinary
singular-cohomology restriction.
\end{enumerate}
\end{theorem}
No claim is made for components in none of the three classes.

Passing to infinitesimal windows, Theorems~\ref{mainThmA} and~\ref{mainThmB} can be 
packaged as follows.
Define the \emph{bounded infinitesimal symplectic cohomology} of $K$ at $t$ by
\begin{equation}\label{eqSHMtplus}
SH^{*,\mathrm{bd}}_{M,t^+}(K;\theta)
:=
\varinjlim_{a\to\infty}
\varprojlim_{\epsilon\to0}
SH^{\le a}_{M,[t,t+\epsilon)}(K;\theta),
\end{equation}
where, for each fixed $a$, the inverse limit is taken over
$0<\epsilon<\delta(a)$.

For an admissible inclusion $K_1\subset K_2$, define
$\iota\colon RH^*(K_2)\to RH^*(K_1)$ by ordinary restriction on the
$H^*$-summand, by $\iota_\gamma$ on common components, and by $0$ on
non-common components. Theorem~\ref{mainThmB} similarly induces a
restriction map on bounded infinitesimal symplectic cohomology.

\begin{cy}\label{mainThmC}
There is a canonical isomorphism
\[
SH^{*,\mathrm{bd}}_{M,t^+}(K;\theta)\xrightarrow{\sim}RH^*(K).
\]
It is natural with respect to admissible inclusions.
\end{cy}
In particular, the bounded infinitesimal group is therefore independent, up to
canonical isomorphism, of the choice of positive contact-type
primitive~$\theta$.

\begin{rem}[Bounded and completed infinitesimal groups]\label{remCompletedInfinitesimal}
The unrestricted inverse limit
\[
SH^*_{M,t^+}(K):=
\varprojlim_{\epsilon\to0}SH^*_{M,[t,t+\epsilon)}(K)
\]
receives a canonical comparison map
\[
SH^{*,\mathrm{bd}}_{M,t^+}(K;\theta)\longrightarrow SH^*_{M,t^+}(K).
\]
 Its failure to be an isomorphism records a form of non-tame
behavior of Novikov torsion.
\end{rem}
 
Finally, we consider restriction maps along a monotone isotopy of domains. More precisely, let $\tau\mapsto K^\tau$, $\tau\in[0,1]$, be a nested family of compact domains with
contact boundary, monotone in the sense that $K^{\tau'}\subset K^{\tau}$ whenever
$\tau\le\tau'$. Fix a primitive $\theta$ of
$\omega$ on a neighborhood of $\bigcup_\tau\partial K^\tau$, whose restriction to every $\partial K^\tau$ is of positive contact type. Let
$\tau\mapsto\gamma^\tau$ be an isolating isotopy of Reeb components
(Definition~\ref{dfReebComponentIsotopy}), and set
\[
\Delta\cA:=T(\gamma^0)-T(\gamma^1).
\]

\begin{theorem}[Restriction along a monotone isotopy]\label{thmIntroIsotopyRestriction}
Let $a>0$ be a period cutoff, and assume $T(\gamma^\tau)\le a$ for all $\tau$.
There exists $\delta_0>0$ such that the restriction map
\[
SH^{\le a}_{M,[0,\delta_0)}(K^0)\longrightarrow SH^*_{M,[0,\delta_0)}(K^1)
\]
sends the $\gamma^0$-summand to the $\gamma^1$-summand, and on that summand is
\[
r_{\gamma^0,\gamma^1}\;=\;T^{\Delta\cA}\cdot\bigl(\iota_{\Gamma}\otimes \id_{\Lambda_{[0,\delta_0)}}\bigr),
\]
where $\iota_{\Gamma}\colon SH^*_{uw}(\gamma^0)\xrightarrow{\sim} SH^*_{uw}(\gamma^1)$
is the isotopy isomorphism of Theorem~\ref{thmLocalReebIntro}\textup{(2)}.
The same local isotopy window works for every sub-isotopy
$\tau\in[s,t]\subset[0,1]$.

\end{theorem}

\begin{rem}\label{remMainAGapBarrier}
  There are two mechanisms for obtaining the small Novikov windows used in the above theorems. One comes from geometric barriers preventing sufficiently low-energy Floer trajectories starting on an isolated Reeb component from leaving its neighborhood. Another mechanism, which is simplest in the case of Theorems~\ref{mainThmA} and~\ref{thmIntroIsotopyRestriction}, uses the uniform, period-independent threshold $\hbar$: below this threshold the relative-homology argument preserves the locally defined $H_1(\partial K)$-grading and action filtration. With this in place one can take $\delta(a)$ to be the minimum of $\hbar$ and the relevant same-class period gaps. Similar, but more involved, reasoning applies in the case of Theorem~\ref{mainThmB} when a joint primitive is available. 
  
The period-gap argument is simpler, but such gaps need not remain uniformly positive in families.  The geometric barriers require substantially more analysis, but depend on quantitative local isolation and geometric bounds and therefore admit uniform versions under corresponding uniform hypotheses.  This is the mechanism needed in the applications and in forthcoming work, where period gaps may degenerate along isotopies.
  \end{rem}
\medskip\noindent\textbf{A weak Viterbo restriction.}
For an admissible inclusion $K_1\subset K_2$ and sufficiently small
$\delta>0$, Theorem~\ref{mainThmB} gives a commutative diagram
\begin{equation}\label{eqHtoSHrestriction}
\xymatrix{
H^*(K_2)\ar[d]\ar[r]&H^*(K_1)\ar[d]\\
SH^*_{M,[0,\delta)}(K_2)\ar[r]&SH^*_{M,[0,\delta)}(K_1)
}
\end{equation}
whose vertical maps are induced by the constant-orbit summands. This is a
weak analogue of Viterbo functoriality; for a general ambient manifold $M$
there need not be a canonical map $H^*(K)\to SH^*_M(K)$.

\subsection{Application: the regular locus of a Lagrangian torus fibration}
\label{subsec:IntroSYZApplication}

Let $\pi:M\to B$ be a Maslov~$0$ Lagrangian torus fibration and let
$C\Subset B_{\mathrm{reg}}$ be an arbitrary compact subset of the regular
locus.  For a sufficiently small compact convex set $P\subset C$, write
\[
M_P:=\Gamma(P,\Lambda^*_{\mathrm{reg}})
\]
for the module of integral sections of the cotangent bundle over $P$, and set
\[
E_P^*:=H^*(\pi^{-1}(P)).
\]
We write $\vartheta_m E_P^*$ for a formal copy of $E_P^*$ indexed by
$m\in M_P$.
For $m\in M_P$ let $f_m$ be any local integral-affine function with
$df_m=m$.
For an inclusion $Q\subset P$ of sufficiently small compact convex sets in
$C$, the intrinsic displacement
\[
d_{P,Q}(m)
:=
\min_Q f_m-\min_P f_m
\]
is independent of the additive constant in $f_m$.  Precise definitions and
proofs are given in \S\ref{subsec:local_floer_regular_fibration}.

\begin{theorem}[Uniform regular-locus model]
  \label{thmRegularSymbol}
  Let $C\Subset B_{\mathrm{reg}}$.  There is a constant $\hbar_C>0$ such that,
  for every sufficiently small compact convex $P\subset C$ and every
  $t\in\mathbb R$, there is an identification
  \begin{equation}\label{eqIntroRegularWindow}
  SH^*_{M,[t,t+\hbar_C)}(K_P)
  \cong
  \bigoplus_{m\in M_P}
  \vartheta_m E_P^*\otimes\Lambda_{[t,t+\hbar_C)},
  \qquad
  K_P:=\pi^{-1}(P).
  \end{equation}
  For an inclusion $Q\subset P$ ,
  restriction is diagonal in $m$, acts by ordinary restriction on cohomology and by 
  \[
  \vartheta_m
  \longmapsto
  T^{\,d_{P,Q}(m)}\,\vartheta_m
  \]
on the formal generators. 
\end{theorem}

The theorem can be read as a zeroth order statement of mirror symmetry over the regular locus. Under this interpretation the index $m$ is the exponent of a Laurent monomial on the expected mirror. This point of view is further developed in forthcoming work. 

\subsection{Local and infinitesimal Floer cohomology of indicator functions}
Let us now outline the main ideas and the strategy of the proof. We focus on explaining Theorem~\ref{mainThmA}. The invariant $SH^*_M(K)$ is best thought of as the Floer cohomology of the indicator function
\begin{equation}\label{eqHK}
H_K(x):=\begin{cases}
0&\quad x\in K,\\
\infty&\quad x\not\in K.\notag
\end{cases}
\end{equation}
While the function $H_K$ itself cannot feature in Floer's equation, its Floer cohomology is defined by an appropriate limiting procedure. In this respect the function $H_K$ is not different from an arbitrary smooth Hamiltonian which has degenerate $1$-periodic orbits. In fact, studying the Floer cohomology of indicator Hamiltonians requires first sharpening the corresponding theory for smooth Hamiltonians.

For a possibly degenerate smooth Hamiltonian $H$, the Floer complex $CF^*(H)$ is represented by the completed mapping telescope of a monotone cofinal sequence of smooth non-degenerate Hamiltonians $H_i\to H$. The question is what remains of the picture that $CF^*(H)$ is generated by $1$-periodic orbits of $H$? If we assume that $H$ has \emph{isolated periodic components}, rigorous meaning can be given in terms of local Floer theory; see, for example, \cite{Pozniak,GinzburgGurel}.

For a well-isolated periodic component $\gamma$, sufficiently small regular perturbations define a confined local Floer complex whose trajectories remain in an exact isolating neighborhood $V$. Gromov--Floer compactness likewise applies to the corresponding confined continuation maps and homotopies. The resulting cohomology is defined over $\bZ$, by exactness of $V$, and is independent of the sufficiently small local Floer data. We denote it by $HF^*_{uw}(\gamma;H)$. It depends only on the local Floer germ near $\gamma$.

But there is more one can say. There is a positive energy cost for a Floer trajectory to enter or exit the isolating neighborhood $V$. We call this the \emph{energy barrier}. The corresponding crossing-energy estimate is a strengthening of Hein's Proposition~3.2 \cite[Proposition~3.2]{Hein}. The implication is that in every sufficiently small Novikov window there is a summand associated with $\gamma$ in truncated Floer cohomology. In Subsection~\ref{SubsecTruncUwFloer} we relate these two pictures: for a well isolated component $\gamma$ there is a natural \emph{unweighted--truncated} identification
\begin{equation}\label{eqLocInf}
HF^*_{[t,t+\delta)}(V;H)\cong HF^*_{uw}(\gamma;H)\otimes \Lambda_{[t,t+\delta)}
\end{equation}

One thing this picture buys us concretely is the behavior under isotopies. As a general rule, under a sufficiently short isotopy of smooth Hamiltonians, with no births/deaths, the continuation map is block diagonal on the isolated local summands, and the local block is the unweighted continuation map multiplied by the source-minus-target action weight $T^{A_0-A_1}$.

So far this is just a sharpening of existing theory. Let us now upgrade this to the case of indicator Hamiltonians. For this, it is technically advantageous to first consider the intermediate finite-slope Hamiltonians $H_{K,I}$. They vanish on $K$, rise with fixed slope along a symplectization collar of $\partial K$, and then bend concavely to an exterior plateau, augmented by a fixed proper
$G_\infty$ outside a compact set when $M$ is noncompact. At positive smoothing scale, every isolated Reeb component $\gamma$ gives an ordinary local Floer group; transport in the smoothing scale and in the finite-slope parameters canonically identifies these groups, giving the invariant $SH^*_{uw}(\gamma)$ of Theorem~\ref{thmLocalReebIntro}.

The main analytic difficulty is to obtain energy barriers uniformly as the smoothing scale tends to zero. As in the smooth case, the argument has two ingredients. The first is a geometric crossing estimate depending only on bounded geometry and $C^0$-bounds for $X_H$ and $\partial_sH$. The second is loop coercivity. Although the latter normally uses $C^2$-control, in a symplectization collar we reparametrize by the radial slope and reduce the estimate to quantitative isolation of the Reeb flow. This gives barriers uniform in the smoothing scale.

A further uniform relative-homology energy threshold shows that sufficiently low-energy trajectories preserve the $H_1(\partial K)$-grading and the locally defined action filtration. Consequently there is a uniform window $\hbar>0$, independent of the slope and auxiliary profile width, in which the concave generators form a subcomplex preserved by structural continuation maps. Passing to the quotient, each inner Reeb block is, in a sufficiently small window depending only on the component,
\[
SH^*_{uw}(\gamma)\otimes\Lambda_{[t,t+\delta)}.
\]
Finally, sufficiently large slope continuation kills the concave subcomplex in truncated cohomology, while the inner blocks are preserved canonically. Taking the cofinal slope limit gives Theorem~\ref{mainThmA}.

\medskip\noindent\textbf{Organization of the paper.}
Section~\ref{SecConventions} fixes our Floer-theoretic conventions and the diagrammatic framework used for generalized Hamiltonians. Section~\ref{SecInfFloer} develops the analytic framework for barrier machinery and the local theory for isolated periodic components of smooth Hamiltonians, including the relation between unweighted local Floer cohomology and Floer cohomology in small Novikov windows. Section~\ref{secInfFH} passes to finite-slope models near contact-type boundaries, constructs the local Floer cohomology and transport of isolated Reeb components, and establishes energy barriers uniform in the smoothing scale. Section~\ref{SecTruncSH} assembles these local results into global finite-slope approximations of indicator Hamiltonians and proves the uniform geometric window and finite-model package; its final subsections, beginning with \S~\ref{SecProofs}, prove the main results on relative symplectic cohomology and restriction maps. Section~\ref{subsec:local_floer_regular_fibration} then gives the regular-locus application to Lagrangian torus fibrations. The appendix records the finite-time estimates used in the barrier arguments.

\subsection{Acknowledgements}
The author is grateful to Mohammed Abouzaid for helpful discussions concerning Morse--Bott Floer theory, and to Umut Varolgunes for comments on an early version of this paper and for many years of collaboration on relative symplectic cohomology and its applications to SYZ mirror symmetry. The author also thanks the anonymous referee whose criticism lead to a thorough revision of the paper, sharpening the results, and, hopefully, improving the exposition. 

The author used ChatGPT (OpenAI) extensively during the revision of this manuscript, principally as a mathematical auditing and editorial tool: to stress-test arguments, identify possible gaps and inconsistencies, explore possible repairs, and improve the organization and wording of the paper. The author takes full responsibility for the mathematical content of the paper.

This work was supported by the Israel Science Foundation (grant nos.~2445/20 and~3605/24).

%% file: sections/02-floer-periodic-orbits.tex
\section{Conventions for Floer cohomology}\label{SecConventions}
For definiteness we work with integral coefficients and coherent orientations.
All ordinary (co)homology groups and tensor products are over $\bZ$ unless
otherwise specified. Set
\begin{equation}\label{eqNovikovField}
\Lambda:=\left\{\sum a_iT^{\lambda_i}\,\middle|\,
\lambda_i\in\mathbb{R},\,a_i\in\bZ,\,\lim_{i\to\infty}\lambda_i=\infty\right\}.
\end{equation}
We equip $\Lambda$ with the non-Archimedean norm
$|x|=e^{-\mathrm{val}(x)}$, where
$\mathrm{val}\bigl(\sum a_iT^{\lambda_i}\bigr)=\inf\{\lambda_i:a_i\neq0\}$.
Write $\Lambda_a$ for the lattice of elements of norm $<e^a$, equivalently
\[
\Lambda_a
=
\bigl\{\textstyle\sum a_iT^{\lambda_i}:
\lambda_i>-a,\,a_i\in\bZ,\,\lim_{i\to\infty}\lambda_i=\infty\bigr\}
=
T^{-a}\Lambda_{>0},
\]
and set $\Lambda_{[a,b)}:=\Lambda_b/\Lambda_a$. The introduction uses the
equivalent presentation $\Lambda_{[a,b)}=T^{-b}\Lambda_{>0}/T^{-a}\Lambda_{>0}$.

We assume the reader is familiar with the basic definitions of Floer
cohomology, and record only the conventions used below. To avoid virtual
perturbations we work throughout in the \emph{semi-positive} setting of
\cite[Chapter~6]{MS2}. For generic compatible almost complex structures,
sphere bubbling then occurs in codimension at least two in the relevant
Floer moduli spaces, so differentials and continuation maps may be defined
by ordinary transversality, with signed counts supplied by coherent orientations.
Other perturbation frameworks may be used with their compatible coefficients;
we retain the semipositive, integral conventions throughout.

For noncompact $M$ we assume $M$ is geometrically of finite type: it admits a
geometrically bounded compatible $J$ and a proper exhaustion with finitely many
critical points and $C^1$-bounded Hamiltonian vector field
\cite[\S3]{GromanVarolgunes2021}. We use
dissipative Floer data in the sense of~\cite{Groman}; the required proper
exterior models are fixed in Subsection~\ref{SubsecHKISetup}. Standard
dissipative Floer compactness and continuation results will be used without
further comment; see also~\cite{GromanVarolgunes2021}.

Let $H:S^1\times M\to\bR$ be smooth and let $J$ be a compatible,
periodically time-dependent almost complex structure. We write $X_H$ for
the Hamiltonian vector field determined by $\omega(X_H,\cdot)=dH$, and
$\Per(H)$ for the set of $1$-periodic orbits of $H$. A Floer trajectory
from $\gamma_1$ to $\gamma_2$ is a solution of
$\partial_su+J(\partial_t-X_H)=0$ with the usual asymptotics, and its
\emph{topological energy} is
\[
E_{\mathrm{top}}(u)
:=
\int u^*\omega
+
\int_{S^1}\bigl(H_t(\gamma_2(t))-H_t(\gamma_1(t))\bigr)\,dt.
\]
One has
\begin{equation}\label{eqGeoTopEn1}
E_{\mathrm{top}}(u)=\int\|\partial_su\|^2\,ds\ge0;
\end{equation}
the right-hand side is the geometric energy $E_{\mathrm{geo}}(u)$. We assume
throughout that Floer data and monotone homotopies are \emph{regular} in the
sense of \cite{HoferSalamon} (see also \cite[\S6.7]{MS2} and
\cite{Salamon1999}), so that the relevant moduli spaces are smooth manifolds
of the expected dimension and the usual compactifications by broken
trajectories apply.

The Floer complex $CF^*(H,J)$ is the free $\Lambda$-module on $\Per(H)$,
graded by a group $\mathsf G$ equipped with a map to $\bZ/2\bZ$ (the reader may take
$\mathsf G=\bZ/2\bZ$). The differential counts index-$1$ trajectories with Novikov
weight $T^{E_{\mathrm{top}}(u)}$, using orientation lines in the usual way.
We assign norm $1$ to each generator and extend by $|T^\lambda\gamma|=e^{-\lambda}$;
the differential is then norm-non-increasing.

A \emph{monotone homotopy} from $(H_1,J_1)$ to $(H_2,J_2)$ is a family
$(H^s,J^s)$ agreeing with the respective endpoints for $|s|\gg0$ and
satisfying $\partial_sH^s\ge0$. Such a homotopy exists if and only if
$H_1\le H_2$ pointwise, and the corresponding solutions still obey
$E_{\mathrm{top}}(u)\ge E_{\mathrm{geo}}(u)\ge0$. Counting them with Novikov
weights yields a chain map
$f_{H^s,J^s}:CF^*(H_1,J_1)\to CF^*(H_2,J_2)$, likewise
norm-non-increasing. For $a\in\bR$ we write $CF^*_a(H,J)$ for the subcomplex
of elements of norm $<e^a$.

\subsection{The Floer complex for general Hamiltonians}
\label{sec:floer_general_hamiltonian}

The definition extends to more general Hamiltonians as follows. We focus on
time-independent $H:M\to\bR\cup\{\infty\}$. A \emph{generalized Hamiltonian}
is a pointwise limit of a monotone sequence of smooth non-degenerate
Hamiltonians which are proper and bounded below. This class includes all
continuous functions which are proper and bounded below, as well as the
indicator functions $H_K$ of compact sets from~\eqref{eqHK}.

On noncompact $M$ we fix exterior Floer data $(G_\infty,J_\infty)$ as in
Subsection~\ref{SubsecHKISetup}. Outside a fixed compact set, Hamiltonians
have the form $G_\infty+c$, continuations have the form $G_\infty+c(s)$,
and $J=J_\infty$. Cofinality and chain-level comparisons below refer to this
fixed exterior-profile class. For an indicator $H_K$, a regular dissipative
ray which is negative on $K$ and converges pointwise to $H_K$ is an
acceleration datum. Its truncated cohomology computes the usual truncated
relative symplectic cohomology, naturally under restriction, truncation, and
comparison of acceleration data, by
\cite[Propositions~2.8--2.9]{GromanVarolgunes2021}.
This comparison does not require order-cofinality among all proper
Hamiltonians; see \cite[Remark~2.5]{GromanVarolgunes2021}.
We use a chosen fixed-profile telescope for $SC^*_M(K)$ and invoke the cited
comparison on truncated cohomology when changing exterior profiles.
For our half-open convention, put $\mathfrak m=\Lambda_{>0}$ and observe that
\[
T^b\Lambda_{[a,b)}
=\mathfrak m/T^{b-a}\mathfrak m
=\varinjlim_{\eta\downarrow0}
T^\eta\Lambda_{\ge0}/T^{\eta+b-a}\Lambda_{\ge0}.
\]
Exactness of filtered colimits transfers the cited comparison to these windows.

For a generalized Hamiltonian $H$, let $\cH_{<H}$ denote the directed
collection of regular Floer data $F=(H',J')$ with $H'\le H$ (and $H'$
strictly below $H$ unless $H$ itself is smooth and non-degenerate), together
with regular monotone homotopies between comparable pairs. On noncompact $M$
this collection and every assertion about order cofinality or chain-level
comparison are understood within the fixed exterior-profile class above.
Here and below,
$F_1\le F_2$ means that the Hamiltonians satisfy $H_1\le H_2$ pointwise.
Choosing a monotone cofinal sequence $F_i=(H_i,J_i)$ in $\cH_{<H}$ with
$H_i\to H$ pointwise, and regular monotone homotopies from $F_i$ to
$F_{i+1}$, one obtains a telescope
\[
CF^*(H_1,J_1)\longrightarrow CF^*(H_2,J_2)\longrightarrow\cdots.
\]
We define $CF^*(H)$ to be the corresponding \emph{completed mapping telescope}
in the sense of~\cite{varolgunes}: the degree-wise non-Archimedean completion
of the ordinary mapping telescope of the ray. Different cofinal sequences and
homotopies give canonically chain-homotopy equivalent complexes, so we regard
$CF^*(H)$ as defined up to canonical chain homotopy equivalence. If $H$ is
itself smooth and non-degenerate, the constant sequence recovers the ordinary
Floer complex.

\medskip\noindent\textbf{Continuation, regular homotopy, and gluing.}
\label{parContinuationHomotopyConventions}
A regular continuation datum $k=(H^s,J^s)$ from $F_0$ to $F_1$ induces a
chain map $\Phi_k:=f_{H^s,J^s}:CF^*(F_0)\to CF^*(F_1)$. A regular homotopy
between two such data induces a chain homotopy between the associated maps.
For composable continuations $k_{01}$ and $k_{12}$, long-neck gluing produces,
for every sufficiently large neck length $R$, a smooth continuation
$k_{01}\#_Rk_{12}$ with
\[
\Phi_{k_{01}\#_Rk_{12}}
\simeq
\Phi_{k_{12}}\circ\Phi_{k_{01}}.
\]
We use these standard facts, including relative regularity for compact
parameter spaces, without further comment; for a homotopy-coherent treatment
via split Floer data and cubical families see
\cite{AbouzaidGromanVarolgunes2024}.

\medskip\noindent\textbf{Floer diagrams.}\label{parFloerDiagrams}
In the applications below we will rarely commit to a single cofinal sequence
and will instead work with families of Floer data. A \emph{Floer diagram}
$\cH$ consists of
\begin{itemize}
\item a collection, also denoted $\cH$, of smooth regular Floer data
$F=(H,J)$;
\item for each pair $F_0,F_1\in\cH$ with $H_0\leq H_1$, a possibly empty set
$\cH(F_0,F_1)$ of regular monotone continuation data from $F_0$ to $F_1$;
\item for each $k^0,k^1\in\cH(F_0,F_1)$, a nonempty collection of regular
homotopies from $k^0$ to $k^1$; and
\item for every composable pair
$k_{01}\in\cH(F_0,F_1)$ and $k_{12}\in\cH(F_1,F_2)$, smooth long-neck gluings
$k_{01}\#_Rk_{12}\in\cH(F_0,F_2)$ for arbitrarily large $R$.
\end{itemize}
The data include the constant continuations and the relative regular
homotopies realizing the gluing identities. Whenever maps of telescopes or
homotopies between them are used, we choose the required finite cubical
families coherently, relative to their prescribed faces and gluing strata,
as in~\cite{AbouzaidGromanVarolgunes2024}. These families remain in the
specified Floer diagram and satisfy its geometric bounds.
The standard facts recalled above then imply that the continuation map associated
with a morphism of $\cH$ is independent of that morphism up to chain homotopy,
and that these maps compose up to chain homotopy. All functoriality statements
below are understood in this sense on chains and hence strictly on cohomology.

A Floer diagram $\cH$ is \emph{contractible} if for every pair $F_1,F_2\in\cH$ there exists $F_3\in\cH$ with $\cH(F_1,F_3)\neq\emptyset$ and $\cH(F_2,F_3)\neq\emptyset$. A contractible Floer diagram $\cH$ \emph{converges to a generalized
Hamiltonian $H$} if $\cH$ contains a monotone cofinal sequence whose
Hamiltonians converge pointwise to $H$ from below.

The canonical example used throughout the paper is the following.
\begin{df}[Canonical Floer diagram]
\label{dfCanonicalFloerDiagram}
\label{parCanonicalFloerDiagram}
Let $\cN$ be a $C^2$-open set of time-dependent Hamiltonians, relative
to the fixed exterior-profile class when $M$ is noncompact. We require that
the monotone-path spaces below admit the relative homotopies and gluing
families of a Floer diagram (as they do for suitable convex $\cN$).
Let $\cB$ be a bounded contractible open family of $\omega$-compatible almost complex structures such that the associated metrics $g_J$, $J\in\cB$, have uniform bounded-geometry constants and are uniformly quasi-isometric; in particular $\cB$ is path-connected and contains all regular paths and homotopies of almost complex structures required below. We refer to such a set $\cB$ as a bounded ball of compatible almost complex structures. Denote by $\cH(\cN,\cB)$ the Floer diagram
\begin{itemize}
\item whose Floer data are all regular pairs $(H,J)\in\cN\times\cB$;
\item whose continuation data between $(H_1,J_1)$ and $(H_2,J_2)$ are all regular monotone paths in $\cN\times\cB$ from $(H_1,J_1)$ to $(H_2,J_2)$ whose first derivative in the continuation parameter is compactly supported and bounded above by $1$; and
\item whose homotopies between two continuation data are all regular paths in the space of monotone paths satisfying the same bounds.
\end{itemize}
On a closed manifold, any fixed compact family of compatible almost complex structures is contained in a bounded ball after enlarging the uniform constants.
We require the relative homotopies and coherent gluing families in the
definition of a Floer diagram to remain in $\cN\times\cB$; a convex
$\cN$ supplies the Hamiltonian part of these families, after stretching the
$s$-coordinate to satisfy the derivative bound.
If there is a lower semi-continuous function $H$ which dominates every Hamiltonian in $\cN$ and $\cN$ contains a monotone cofinal sequence converging pointwise to $H$, and moreover every object of $\cN$ admits a monotone path in $\cN$ to a sufficiently far member of that cofinal family, then $\cH(\cN,\cB)$ is contractible and converges to $H$.
\end{df}

To a contractible Floer diagram $\cH$ we associate the chain complex $CF^*(\cH)$ by forming the
completed mapping telescope of any monotone cofinal sequence in $\cH$.
Different choices give canonically chain-homotopy equivalent complexes by
means of these coherent comparison families. If $\cH$
converges to $H$, then $CF^*(\cH)$ is canonically chain-homotopy equivalent to
$CF^*(H)$, and we speak freely of $CF^*(\cH)$ and $HF^*(\cH)$ without naming
the auxiliary sequence.

A contractible Floer diagram with no additional structure will be referred to below as an \emph{object diagram}. We now introduce continuation diagrams, which correspond to $1$-morphisms, and have associated with a source and a target.

\begin{df}[Continuation diagram]
\label{dfContinuationDiagrams}
\label{parContinuationDiagrams}
Let $\cH_1$ and $\cH_2$ be contractible Floer diagrams converging to generalized Hamiltonians $H_1$ and $H_2$ with $H_1\leq H_2$. A \emph{Floer diagram parametrizing monotone continuation maps from $\cH_1$ to $\cH_2$}, also referred to simply as a \emph{continuation diagram from $\cH_1$ to $\cH_2$}, is a Floer diagram $\cH_{1,2}$ which contains $\cH_1$ and $\cH_2$ as full sub-diagrams and has the property that for every $F_1\in\cH_1$ and $F_2\in\cH_2$ with $F_1\leq F_2$, the set $\cH_{1,2}(F_1,F_2)$ is non-empty. We refer to $\cH_1$ and $\cH_2$ as the \emph{source} and \emph{target} respectively of the continuation diagram.

To such a continuation diagram is associated a chain map
\[
CF^*(\cH_1)\longrightarrow CF^*(\cH_2),
\]
well-defined up to chain homotopy. It is obtained by choosing interleaved monotone cofinal sequences $(F_i^1)\subset\cH_1$ and $(F_i^2)\subset\cH_2$ with $F_i^1\leq F_i^2$, together with continuation data in $\cH_{1,2}(F_i^1,F_i^2)$, and assembling them into a map between the corresponding mapping telescopes; independence of the choices follows from the coherent comparison and gluing families of $\cH_{1,2}$. We will use this format throughout the paper: a Floer-theoretic construction is specified by a Floer diagram $\cH$ together, when applicable, with a continuation diagram $\cH_{1,2}$, and the associated complex and chain map are extracted without further mention of cofinal sequences.

When the source and target object diagrams converge to the same generalized
Hamiltonian $H$, we refer to a continuation diagram between them as a
\emph{comparison diagram over $H$}.
\end{df}

On cohomology the resulting continuation maps compose strictly. On chains we
use the coherent comparison families specified above.

\medskip\noindent\textbf{Functoriality and restriction.}
If $H\le K$ are generalized Hamiltonians, any continuation diagram between
contractible Floer diagrams converging to $H$ and $K$ induces a canonical
chain map $CF^*(H)\to CF^*(K)$, well-defined up to chain homotopy and
independent of the diagrams. For an inclusion of compact sets
$K_1\subset K_2$ one has $H_{K_2}\le H_{K_1}$, and the corresponding map
\[
CF^*(H_{K_2})\longrightarrow CF^*(H_{K_1})
\]
is called the \emph{restriction map}.

\medskip\noindent\textbf{Norm filtration and truncation.}
The complexes $CF^*(H_i,J_i)$ are filtered by the subcomplexes of elements of
norm $<e^a$, and monotone continuation maps are norm-non-increasing. The
mapping telescope of any Floer diagram $\cH$ therefore inherits an increasing
filtration $CF^*_a(\cH)\subset CF^*(\cH)$. When $\cH$ converges to $H$ we write
$CF^*_a(H)$ for $CF^*_a(\cH)$, the construction being independent of $\cH$ up
to filtered chain homotopy equivalence. Truncated Floer cohomology is
\[
HF^*_{[a,b)}(H)
:=
H^*\bigl(CF^*_b(H)/CF^*_a(H)\bigr).
\]
\begin{lm}\label{lm:truncation_exactness}
For any generalized Hamiltonian $H$ and any contractible Floer diagram $\cH$
converging to $H$, there is a canonical isomorphism
\begin{equation}\label{eqTruncHom}
HF^*_{[a,b)}(H)
\cong
\varinjlim_{F\in\cH}HF^*_{[a,b)}(F).
\end{equation}
\end{lm}
\begin{proof}
After choosing a monotone cofinal sequence in $\cH$, the truncated complex
$CF^*_b(H)/CF^*_a(H)$ is represented by the corresponding telescope of
truncated complexes. On homology this computes the directed colimit, because
filtered colimits of modules are exact.
\end{proof}
\begin{rem}
The notation $HF^*_{[a,b)}(H)$ differs slightly from that of \cite{Groman}:
here the filtration is by the Novikov norm, whereas there it is by the action
filtration, which is a monotone function of the norm.
\end{rem}

This colimit is the only cofinality principle used below. In particular, once
a geometrically adapted Floer diagram---for instance one isolating a
prescribed Reeb component, or adapted to a symplectization
neighborhood---is checked to be contractible and to converge to $H$, all
truncated groups and continuation maps below $H$ may be computed inside that
diagram, and splittings or filtrations constructed on the finite-slope models
pass to $HF^*_{[a,b)}(H)$ by taking the directed colimit.

\section{$C^0$ confinement and local Floer theory for smooth Hamiltonians}
\label{SecInfFloer}
\label{SecUnweightedLocal}
In this section we lay the foundations for studying local Floer theory and connecting it with Floer theory in small action windows. We do this first for  finite smooth Hamiltonian, even though our main concern is ultimately with continuous finite-slope indicator Hamiltonians, which are, by definition, not smooth. The reason is partly expositional, to separate issues already present for finite smooth Hamiltonians from those that arise only when allowing some non-smoothness. But we will also directly apply the results for finite smooth Hamiltonians later in the paper, when dealing with concave orbits. We make no claims of originality as far as local Floer theory for finite smooth Hamiltonians is concerned. These are covered by the works of \cite{Pozniak, Hein,GinzburgGurel} and others. Nevertheless, this section is indispensable for reading the rest of the paper. 

\subsection{The loopwise barrier principle and a generalization of Hein's lemma}
\label{SubsecCZeroConfinement}

This subsection develops the analytic confinement principle used throughout
the paper when $1$-periodic orbits occur in isolated families. The
engine is a shared  mechanism involving two ingredients:

\medskip\noindent\textbf{Geometric crossing.}
A solution of Floer's equation on a cylinder of length $L$ that crosses from an inner region $V_-$ to the
complement of an outer region $V_+$ admits a lower bound of the form
$E^{\mathrm{geo}}(u)+c_1L\ge c_0$, with positive constants $c_0,c_1$
depending only on bounded geometry of the almost complex metrics,
$C^0$-bounds on $X_H$ and $\partial_sH$, and spatial separation of the
regions. In fact, we prove a refined version of this in Lemma \ref{lmGraphDiameterCylinder}. This refinement is important for our applications in the non-smooth case and later on in Section~\ref{SubsecRelativeHomology}, and is a new contribution of this paper.  This input relies on monotonicity and is common to the smooth and finite-slope indicator settings.

\medskip\noindent\textbf{Loop coercivity.}
On the crossing region $V_+\setminus V_-$ one has a uniform lower bound
\[
E^2_{H_s}(u_s)\ge\delta_{\mathrm{loop}}>0
\]
on the loop defect, and therefore $E^{\mathrm{geo}}(u)\ge\delta_{\mathrm{loop}}L$.
In the smooth $C^2$-controlled setting this input is elementary. When no uniform $C^2$-bound is available,
as for the degenerating locally slope-adapted smoothings of Section~\ref{secInfFH}, establishing a uniform $\delta_{loop}$ becomes rather delicate and will take up considerable effort in the next section.

We package these two ingredients into the following loopwise barrier principle whose proof will occupy most of the remainder of the subsection.

\begin{prp}[Loopwise barrier principle]
\label{prpLoopwiseBarrier}
Fix a monotone domain-dependent Floer datum $(H,J)$ on $\bR\times S^1$,
with the uniform bounds of Lemma~\ref{lmMonEst++}, and let
$u:\bR\times S^1\to M$ be a finite-energy solution of~\eqref{eqFloer}.
Fix a finite interval $[a,b]\subset\bR$.
For $t_0\in S^1$ denote by
\[
d:=d_M\bigl(u(a,t_0),u(b,t_0)\bigr)
\]
the distance between the endpoint loops at time~$t_0$.
If for a.e.\ $s\in(a,b)$ the loop $u_s$ has defect at least
$\delta_{\mathrm{loop}}>0$, then there is a constant $c>0$, depending
only on $\delta_{\mathrm{loop}}$, the geometry of~$M$, and $C^0$-bounds
for $X_H$ and $\partial_s H$, such that
\[
E^{\mathrm{geo}}(u)
\geq
\min\bigl\{
c\cdot d,\;
c\cdot d^2
\bigr\}
\]
for every $t_0\in S^1$. Here $E^{\mathrm{geo}}(u)$ is the energy of the
whole cylinder, not only of its restriction to $[a,b]\times S^1$.
\end{prp}

The proof of the loop barrier principle involves a study of the graphs of Floer trajectories as $J$ holomorphic curves. Denote by $\Sigma$ the cylinder $\bR\times S^1$ considered as a Riemann surface with its standard complex structure $j_{\Sigma}$. 
Let $H:\Sigma\times M\to\bR$ be a function which is $s$ independent outside of a compact set. Let $(s,t)\mapsto J_{s,t}$ be a $\Sigma$-dependent family of almost complex structures on $M$ which is $s$ independent outside of a compact set.  To this datum we associate an almost complex structure $J_{H}$ on $\Sigma\times M$ defined by
\begin{equation}\label{eqGrTrick}
J_{H}:=J_M+j_{\Sigma}+X_H\otimes ds- JX_H\otimes dt.
\end{equation}
When $H$ is $s$-dependent, we assume $\partial_sH\geq 0$. 
We take the Hamiltonian vector field convention $\iota_{X_H}\omega=dH$. In this case the closed form
\begin{equation}\label{eqGrTrick2}
\omega_{H}:=\pi_1^*\omega_{\Sigma}+\pi_2^*\omega+ dH\wedge dt
\end{equation}
on $\Sigma\times M$ can be shown to be symplectic and  $J_{H}$ is compatible with it. Here $\omega_\Sigma$ is any symplectic form compatible with  $j_\Sigma$.  We shall take $\omega_\Sigma$ to coincide with the form $ds\wedge dt$ on the ends.
We denote the induced metric on $\Sigma\times M$ by $g_{J_{H}}$. We refer to the metric $g_{J_H}$ as the \emph{Gromov metric}. We stress that \emph{a Gromov metric depends on the choice of area form on $\Sigma$}. 

In the following lemma we consider in particular the case $\Sigma=\bR\times S^1$, the Hamiltonian 1-form is $H_{s,t}dt$ and $\omega_{\Sigma,\tau}:=\tau^2 ds\wedge dt$. Let 
$\pi_1,\pi_2$ be the projections to $\Sigma$ and $M$ respectively. Let $g_{\Sigma}$ be the standard metric on the cylinder $\bR\times S^1$. Given a pair of Riemannian metrics $g_1,g_2$ on a smooth manifold, we say they are \emph{$C$-equivalent} for some $C>1$ if $\frac1{C}\|v\|^2_{g_2}<\|v\|^2_{g_1}<C\|v\|^2_{g_2}$

\begin{lm}[Gromov--product metric comparison]
\label{lmProdCEquivalence}
There is a continuous function $f:\bR_+^2\to\bR_+$ converging to $1$ at $(0,0)$ such that writing $C_{\tau}:=f\left(\left|\frac{\partial_sH}{\tau^2}\right|_{\infty},\left|\frac{\|X_H\|}{\tau}\right|_{\infty}\right)$, the metric $g_{J_H}$ determined by $J_H$ and $\omega_{\Sigma,\tau}$ is $C_{\tau}$-equivalent to the product metric $\tau^2\pi_1^*g_\Sigma+\pi_2^*g_J$.  
\end{lm}
\begin{proof}
We have
\begin{equation}\label{eqGromovMetricFormula}
g_{J_H}=(\tau^2+\partial_sH)\pi_1^*g_{\Sigma}+\|X_H\|^2_{g_J}(\pi_1^*dt)^2
-\bigl(dt\otimes g_J(X_H,\cdot)+g_J(X_H,\cdot)\otimes dt\bigr)+\pi_2^*g_J.
\end{equation}
Equivalently, on a vector $a\partial_s+b\partial_t+v$ with $v\in TM$,
\[
g_{J_H}(a\partial_s+b\partial_t+v,\,a\partial_s+b\partial_t+v)=(\tau^2+\partial_sH)(a^2+b^2)+\|v-bX_H\|^2_{g_J}.
\]
To relate this metric with the product metric we split $$T(\Sigma\times M)=\left(T\Sigma\oplus\bR X_H\right)\oplus (\bR X_H)^{\perp},$$ 
where  $(\bR X_H)^{\perp}$ denotes the orthogonal complement in $TM$ with respect to $g_{J_H}$. Note this is also an orthogonal splitting with respect to $g_J$. Let $p_1, p_2$ be the orthogonal projections associated with the splitting. When $X_H\neq 0$, the Gram matrix of $g_{J_H}$ in the $g_{\mathrm{prod}}$-orthonormal frame $\bigl(\partial_s/\tau,\,\partial_t/\tau,\,X_H/\|X_H\|\bigr)$ on the first summand is
$$
\begin{pmatrix}1+\frac{\partial_sH}{\tau^2}& 0& 0\\
0& 1+\frac{\partial_sH+\|X_H\|^2}{\tau^2}& -\frac{\|X_H\|}{\tau}\\
0& -\frac{\|X_H\|}{\tau}& 1
\end{pmatrix}
$$
Thus $g_{J_H}$ is represented with respect to the product metric by the linear map $A\circ p_1+p_2$. For $x,y\geq 0$ consider the matrix $A(x,y)$
$$
\begin{pmatrix}1+x& 0& 0\\
0& 1+x+y^2& -y\\
0& -y& 1
\end{pmatrix}
$$
and let $f(x,y)=\max\left\{\lambda_{max},\frac 1{\lambda_{min}}\right\} $ where we refer to the maximal and minimal eigenvalue of the positive definite symmetric matrix $A(x,y)$. Then $f(x,y)\to 1$ as $(x,y)\to (0,0).$
Moreover, for $C_{\tau}:=f\left(\left|\frac{\partial_sH}{\tau^2}\right|_{\infty},\left|\frac{\|X_H\|}{\tau}\right|_{\infty}\right)$ the metric $g_{J_H}$ is $C_{\tau}$-equivalent to the corresponding product metric.
\end{proof}

An observation known as \emph{Gromov's trick} is that $u$ is a solution to Floer's equation
\begin{equation}\label{eqFloer}
(du-X_{H})^{0,1}=0,
\end{equation}
if and only if its graph $\tilde{u}$ satisfies the Cauchy Riemann equation
\begin{equation}
\overline{\partial}_{J_{H}}\tilde{u}=0.
\end{equation}
Thus Floer trajectories can be considered as $J_{H}$-holomorphic sections of $\Sigma\times M\to\Sigma$. 
To a Floer solution $u:\Sigma\to M$ and a subset $S\subset\Sigma$ we can now associate three different non-negative real numbers:
\begin{itemize}
\item The \emph{geometric energy} $E_{geo}(u;S):=\frac12\int_S\|(du-X_{H})\|^2$ of $u$.
\item The \emph{topological energy} $E_{top}(u;S):=\int_S u^*\omega+\tilde{u}^*(dH\wedge dt).$
\item The \emph{symplectic energy} $E(\tilde{u};S):=\int_S\tilde{u}^*\tilde{\omega}$. 
\end{itemize}
We have the relation $E(\tilde{u};S)=E_{top}(u;S)+Area(S)$. For a monotone Floer datum we have, in addition, the relation $E_{geo}(u;S)\leq E_{top}(u;S)$. 

The key to obtaining $C^0$ estimates is the monotonicity lemma \cite{Sikorav94}.  For a $J$-holomorphic map $u:\Sigma\to M$ and for a measurable subset $U\subset \Sigma$ write
\begin{equation}
E(u;U):=\int_Uu^*\omega.
\end{equation}
\begin{lm}\label{lmMonEst++}[\textbf{Monotonicity} \cite{Sikorav94} ]
There is a universal constant $c$ such that the following holds. Fix a compatible almost complex structure $J$ on $M$. Let $a$ be a constant and let $p\in M$ be a point at which $|Sec|\leq a^2$ on the ball $B_{1/a}(p)$ and such that at $p$ the injectivity radius is $\geq\frac1{a}$. 
 Let $S$ be a compact Riemann surface with boundary and let $u:S\to M$ be $J$-holomorphic such that $p$ is in the image of $u$ and such that
\[
u(\partial S)\cap B_{1/a}(p)=\emptyset.
\]
Then 
\begin{equation}
E\left(u;u^{-1}(B_{1/a}(p))\right)\geq\frac {c}{a^2}.
\end{equation}
If, instead, we only require that there exists a constant $C>1$ and a Riemannian metric $h$ satisfying the above sectional-curvature and injectivity-radius bounds and such that 
$$
\frac1{C}g_J(v,v)\leq h(v,v)\leq Cg_J(v,v),
$$
we get the inequality
\begin{equation}
E\left(u;u^{-1}(B_{1/a}(p))\right)\geq\frac{c}{C^3a^2}.
\end{equation}
 \end{lm}
\begin{rem}
The comparison clause is the form of~\cite[Theorem~4.9]{Groman}, using the
isoperimetric filling of~\cite[Lemma~4.10]{Groman}.  The power of~$C$ absorbs
the comparison of lengths and $2$-forms together with the quasi-isometric
distortion of balls.
\end{rem}

\medskip\noindent\textbf{Standing assumption.}
In the next two lemmas we fix a compatible almost complex structure~$J$
with
\[
|\operatorname{Sec}(g_J)|\leq K,
\qquad
\operatorname{inj}(g_J)\geq K^{-1/2},
\]
and write $d_M$ for the distance induced $M$ by~$g_J$.

\begin{lm}[Graph-diameter estimate on the Floer cylinder]
\label{lmGraphDiameterCylinder}
There is a nondecreasing function
\[
C:[0,\infty)\longrightarrow (0,\infty),
\]
depending only on the dimension, with the following property.

Let
\[
u:\mathbb R\times S^1\longrightarrow M
\]
solve Floer's equation for a monotone domain-dependent datum, and fix
$\tau>0$. Set
\[
q:=
\max\left\{
\frac{\|\partial_sH\|_{C^0}}{\tau^2},
\frac{\|X_H\|_{C^0}}{\tau}
\right\}.
\]
For $z_0,z_1\in\mathbb R\times S^1$, write
\[
d:=
d_M\bigl(u(z_0),u(z_1)\bigr),
\qquad
\rho:=
\min\left\{
d,\tau,K^{-1/2}
\right\}.
\]
Then
\begin{equation}\label{eqGraphDiameterCylinder}
d\rho
\leq
C(q)\left(
E^{\mathrm{geo}}(u)
+
\tau\rho\,d_\Sigma(z_0,z_1)
\right).
\end{equation}
For monotone data, $E^{\mathrm{geo}}(u)$ may be replaced by
$E^{\mathrm{top}}(u)$.
\end{lm}
\begin{proof}
Write $\widetilde g_\tau$ for the Gromov metric associated with
$\widetilde\omega_\tau:=\tau^2\pi_1^*\omega_\Sigma+\pi_2^*\omega+dH\wedge dt$,
and write $\widetilde u$ for the graph of~$u$. Let
\[
F(q):=\max_{(x,y)\in[0,q]^2}f(x,y),
\]
where $f$ is the function of
Lemma~\ref{lmProdCEquivalence}. Thus $\widetilde g_\tau$ is
$F(q)$-equivalent to the product metric
$\tau^2g_\Sigma\oplus g_J$. So the
Gromov distance
\[
\widetilde D
:=
d_{\widetilde g_\tau}\bigl(\widetilde u(z_0),\widetilde u(z_1)\bigr)
\]
satisfies $\widetilde D\ge d/\sqrt{F(q)}$.

The geometry of the product metric is $a_{prod}$-bounded for $a_{prod}= \min\bigl\{\tau,K^{-1/2}\bigr\}$. Pick any $r\leq a_{prod}$.  
Lemma~\ref{lmMonEst++}, applied to the Gromov metric using the
controlled comparison metric from
Lemma~\ref{lmProdCEquivalence}, gives
\begin{equation}\label{eqGraphBallMonotonicity}
E\left(
\widetilde u;
\widetilde u^{-1}
\bigl(B_r^{\widetilde g_\tau}(p)\bigr)
\right)
\geq c_1(q)r^2
\end{equation}
for every point $p$ on the graph, where
$c_1(q)=c_{\mathrm{univ}}/F(q)^3$ and $c_{\mathrm{univ}}$ is the universal
constant of Lemma~\ref{lmMonEst++}. 

Choose a minimizing geodesic
\[
\beta:[0,1]\longrightarrow\mathbb R\times S^1
\]
from $z_0$ to $z_1$ and consider its lift $\widetilde u\circ\beta$ to the graph. The $\tilde{g}_{\tau}$ length of this lifted path is at leas $\widetilde{D}$. Taking  $r\leq \widetilde D/2$, one can choose at least
\[
N\geq\frac{\widetilde D}{2r}
\]
points on the lifted path whose pairwise
$\widetilde g_\tau$-distances are at least $2r$. The corresponding
$r$-balls are disjoint. Summing
\eqref{eqGraphBallMonotonicity} gives
\begin{equation}\label{eqGraphPackingCylinderLower}
\sum_{i=1}^N
E\left(
\widetilde u;
\widetilde u^{-1}(B_r(p_i))
\right)
\geq \frac{c_1(q)}{2}\widetilde D\,r.
\end{equation}

Metric comparison shows that the projections to the cylinder of
these graph pieces lie in a
\[
C_1(q)\frac r\tau
\]
neighborhood of $\beta$, where $C_1(q)=\sqrt{F(q)}$. Hence their total
unscaled domain area is at most
\[
C_2(q)\left(
\frac r\tau\,\ell(\beta)
+
\frac{r^2}{\tau^2}
\right),
\]
where $C_2(q)=\max\bigl\{2C_1(q),\,\pi C_1(q)^2\bigr\}$.
The graph-energy identity and
$\|\partial_sH\|_{C^0}\leq q\tau^2$ therefore give
\[
\sum_{i=1}^N
E\left(
\widetilde u;
\widetilde u^{-1}(B_r(p_i))
\right)
\leq
E^{\mathrm{geo}}(u)
+
C_3(q)\left(
\tau r\,\ell(\beta)+r^2
\right),
\]
where $C_3(q)=(1+q)C_2(q)$.
Set
\[
r=\alpha(q)\,\rho,
\qquad
\alpha(q)
:=
\min\left\{
\frac1{2\sqrt{F(q)}},
\,
\frac{c_1(q)}{4C_3(q)\sqrt{F(q)}}
\right\}.
\]
Then $r\le a_{prod}$ and $r\le\widetilde D/2$, so the packing applies,
and $C_3(q)r^2\le\frac{c_1(q)}{4}\widetilde D\,r$. Combining with
\eqref{eqGraphPackingCylinderLower} therefore yields
\[
\widetilde D\,r
\leq
C_4(q)\left(
E^{\mathrm{geo}}(u)
+
\tau r\,\ell(\beta)
\right),
\]
where $C_4(q)=4\max\{1,C_3(q)\}/c_1(q)$. Since
$\widetilde D\ge d/\sqrt{F(q)}$ and $\ell(\beta)=d_\Sigma(z_0,z_1)$,
one has $d\rho\le\sqrt{F(q)}\,\widetilde D\,\rho$, hence
\[
d\rho
\le
\frac{\sqrt{F(q)}}{\alpha(q)}\,\widetilde D\,r
\le
\frac{\sqrt{F(q)}\,C_4(q)}{\alpha(q)}
\left(
E^{\mathrm{geo}}(u)
+
\tau\rho\,d_\Sigma(z_0,z_1)
\right).
\]
The prefactor is a nondecreasing function of $q$, which is the constant
$C(q)$ in~\eqref{eqGraphDiameterCylinder}.
\end{proof}

Sometimes the $C^0$-bounds entering~$q$ are available only on a
subdomain $S\subset\mathbb R\times S^1$. One must then prescribe a path
in~$S$ along which to pack, and set
$\rho:=\min\{d,\delta,\tau,K^{-1/2}\}$ to keep the packing away from
$\partial S$.

\begin{lm}[Interior path estimate]
\label{lmInteriorPathPacking}
With $C$ as in Lemma~\ref{lmGraphDiameterCylinder}, let
$S\subset\mathbb R\times S^1$ be a domain and let
\[
u:S\longrightarrow M
\]
solve Floer's equation. Fix $\tau>0$, and define~$q$ as above from the
$C^0$-norms of $\partial_sH$ and $X_H$ on~$S$. Let $\beta:[0,1]\to S$ be
a rectifiable path, write
\[
d:=
d_M\bigl(u(\beta(0)),u(\beta(1))\bigr),
\qquad
\delta:=
d_M\bigl(u(\beta([0,1])),u(\partial S)\bigr),
\]
and set
\[
\rho:=
\min\left\{
d,\delta,\tau,K^{-1/2}
\right\}.
\]
Then
\begin{equation}\label{eqInteriorPathEstimate}
d\rho
\leq
C(q)\left(
E^{\mathrm{geo}}(u;S)
+
\tau\rho\,\ell(\beta)
\right).
\end{equation}
\end{lm}
\begin{proof}
Apply the proof of Lemma~\ref{lmGraphDiameterCylinder} to the
prescribed path $\beta$. Metric comparison gives
\[
d_{\widetilde g_\tau}
\bigl(
\widetilde u(\beta([0,1])),
\widetilde u(\partial S)
\bigr)
\geq c(q)\delta
\]
and
\[
d_{\widetilde g_\tau}
\bigl(
\widetilde u(\beta(0)),
\widetilde u(\beta(1))
\bigr)
\geq c(q)d.
\]
Taking the packing radius to be a sufficiently small
$q$-dependent multiple of
\[
\min\{d,\delta,\tau,K^{-1/2}\}
\]
therefore ensures both that the packed balls avoid the boundary of
the graph and that the endpoint error is absorbed. The rest of the
proof is identical.
\end{proof}

\begin{rem}[Localization of $C^0$ bounds]
\label{rmLocalizedC0Packing}
If the image of the packing path lies in a set $W\subset M$, the
$C^0$-bounds entering~$q$ need only hold on a neighborhood of~$W$ whose
radius is a $q$-dependent multiple of the packing scale~$\rho$ appearing
in Lemmas~\ref{lmGraphDiameterCylinder}
and~\ref{lmInteriorPathPacking}.
\end{rem}

\begin{proof}[\medskip\noindent\textbf{Proof of Proposition~\ref{prpLoopwiseBarrier}}]
Write $A_*:=\|\partial_sH\|_{C^0}$ and $B_*:=\|X_H\|_{C^0}$. Set
\[
\tau_*:=2\max\{1,\sqrt{A_*},B_*\},
\]
so that
\[
q:=
\max\left\{
\frac{\|\partial_sH\|_{C^0}}{\tau_*^2},
\frac{\|X_H\|_{C^0}}{\tau_*}
\right\}
\le\tfrac12,
\]
and write $L:=b-a$. Let $K$ bound the geometry as in
Lemma~\ref{lmMonEst++}. Lemma~\ref{lmGraphDiameterCylinder}, applied to
the endpoints $(a,t_0)$ and $(b,t_0)$, yields
\[
d\rho_0
\leq
C(q)\left(
E^{\mathrm{geo}}(u)
+
\tau_*\min\{\tau_*,K^{-1/2}\}\,L
\right)
\]
with $\rho_0:=\min\{d,\tau_*,K^{-1/2}\}$, i.e.\
$E^{\mathrm{geo}}(u)+c_1L\ge c_0$ for
\[
c_0=\frac{d\rho_0}{C(q)},
\qquad
c_1=\tau_*\min\{\tau_*,K^{-1/2}\}.
\]
Floer's equation identifies $\|\partial_tu-X_{H_s}(u)\|$ with
$\|\partial_su\|$, so integrating the loop-defect lower bound over
$(a,b)$ yields $E^{\mathrm{geo}}(u)\ge\delta_{\mathrm{loop}}L$. The
one-variable lower bound
\[
\frac{\delta_{\mathrm{loop}}}{\delta_{\mathrm{loop}}+c_1}\,c_0
=
\frac{\delta_{\mathrm{loop}}}{\delta_{\mathrm{loop}}+c_1}\,\frac{\rho_0}{C(q)}\,d
\]
is therefore a lower bound on $E^{\mathrm{geo}}(u)$. Writing
\[
\kappa:=
\frac{\delta_{\mathrm{loop}}}{(\delta_{\mathrm{loop}}+c_1)C(q)},
\]
this is $\kappa\,d\rho_0=\kappa\min\{d^2,\,\tau_*d,\,K^{-1/2}d\}$.
The right-hand side is at least $\min\{c\cdot d,\,c\cdot d^2\}$ upon
taking $c=\kappa\min\{1,\tau_*,K^{-1/2}\}$.
\end{proof}

As an application of Proposition~\ref{prpLoopwiseBarrier} we record the
following strengthening of a Lemma by Hein \cite[Proposition 3.2]{Hein}. Given an arbitrary proper Hamiltonian $H$, truncated Floer cohomology
$HF^*_{[a,b)}(H)$ is defined via~\eqref{eqTruncHom}. The proof is written
for closed $M$; in the noncompact case we use the same local estimate for
fixed-tail dissipative data, since the argument takes place in the prescribed
precompact isolating neighborhood.

\begin{lm}[Generalized Hein lemma]\label{lmEpsilonSeparation}
Let $H\in C^\infty(S^1\times M)$ be a smooth, possibly time-dependent
Hamiltonian. Let $\gamma$ be a compact isolated family of $1$-periodic
orbits of $H$: there are precompact opens $V_-\Subset V_+$ such that every
$1$-periodic orbit of $H$ contained in $\overline V_+$ belongs to $\gamma$
and lies in $V_-$. Fix a bounded ball $\cB$ of compatible
almost complex structures in the sense of
Definition~\ref{dfCanonicalFloerDiagram}. For every sufficiently small
choice of such an outer isolating neighborhood $V=V_+$ there are
$\epsilon>0$ and a $C^2$-neighborhood $\cN$ of $H$ in
$C^\infty(S^1\times M)$ such that the following hold.
\begin{enumerate}
\item \textbf{(Energy barrier.)} Any Floer solution associated with $\cH(\cN,\cB)$ which has at least one end on a periodic orbit contained in $V$ and which is not entirely contained in $V$ has topological energy $E^{\mathrm{top}}\ge\epsilon$.
\item \textbf{(Truncated splitting.)} For any regular $(H',J)\in\cN\times\cB$ the truncated complex $CF^*_{[a,a+\epsilon)}(H',J)$ splits as the direct sum of the span of those $1$-periodic orbits of $H'$ whose entire image is contained in $V$ and the complementary span of all remaining periodic orbits, and monotone continuations in $\cN\times\cB$ with $s$-derivatives $\le 1$ preserve the split.
\end{enumerate}
We denote the corresponding summand of the ambient truncated Floer group by
$HF^*_{[a,a+\epsilon)}(V;H)$; it is obtained as a colimit over a monotone
sequence converging to~$H$.  This notation records an ambient block, not a
second definition of local Floer cohomology.
\end{lm}

\begin{rem}\label{rmHeinLemma}
If $V$ is a sub-level set of $H$, clause (1) of the lemma is an immediate consequence of \cite[Proposition 3.2]{Hein}. We wish, however, to separate a small neighborhood of a family of $1$-periodic orbits, not a sublevel set. The estimates depend only on the behavior of $(H,J)$ in $V$. In particular, for $V$ and $\epsilon$ as in the lemma, this ambient truncated block depends only on the restriction of $(H,J)$ to $V$.
\end{rem}

\begin{proof}
Use the isolation pair $V_-\Subset V_+=V$ from the statement. We first produce a $C^2$-neighborhood $\cN$ of $H$ and a constant $\delta_{\mathrm{sm}}>0$ such that every $H'\in\cN$ and every smooth loop $\ell$ with
\[
\ell(S^1)\subset V_+,
\qquad
\ell(S^1)\cap(V_+\setminus V_-)\ne\varnothing
\]
obeys
\[
\int_{S^1}
\bigl|\dot\ell-X_{H'_t}(\ell)\bigr|^2\,dt
\ge\delta_{\mathrm{sm}}.
\]
The norm is taken with respect to $g_J$. Since the metrics $g_J$, $J\in\cB$, are uniformly equivalent, the loop-defect lower bound may be chosen uniformly over $\cB$. If no such $\cN$ and $\delta_{sm}$ existed there would exist $H'_n\to H$ in $C^2$ and loops $\ell_n$ as above with loop defect tending to zero and a basepoint in the compact annulus $\overline{V_+\setminus V_-}$.  By Arzela-Ascoli, passing to a subsequence yields, in the limit, a $1$-periodic orbit of $H$ that meets $V_+\setminus V_-$, contradicting the construction of $V_+$. 

For such an $\cN$ let $u$ be any solution associated with $\cH(\cN,\cB)$ with a negative end in $V_-$ that is not entirely contained in $V_+$. Define
\[
b:=\inf\{r:u_r(S^1)\not\subset V_+\},
\qquad
a:=\sup\{r<b:u_r(S^1)\subset V_-\}.
\]
Then $a<b$, and for $s\in(a,b)$ one has $u_s(S^1)\subset V_+$ and $u_s(S^1)\cap(V_+\setminus V_-)\ne\varnothing$, so $E^2_{H_s}(u_s)\ge\delta_{\mathrm{sm}}$. The fixed separation $d(\overline{V_-},M\setminus V_+)>0$ supplies a positive lower bound on the distance $d$ of Proposition~\ref{prpLoopwiseBarrier}. That proposition therefore yields a uniform lower bound on $E^{\mathrm{geo}}(u)$ and hence on $E^{\mathrm{top}}(u)$ by monotonicity of the Floer data. The case of a positive local end is identical after defining $a,b$ by first entrance rather than last exit.

Any differential or structural continuation trajectory that connects a generator whose image lies entirely in $V$ to a generator outside that local span, and is visible in a Novikov window of width $<\epsilon$, would have topological energy $<\epsilon$, contradicting the barrier. Hence the truncated complexes split as the local span plus its complement, and continuation maps preserve the split as in~(2). Passing to the directed colimit gives the component $HF^*_{[a,a+\epsilon)}(V;H)$.
\end{proof}

\subsection{The unweighted local Floer complex}

The Floer complexes $CF^*(H)$ are defined over $\Lambda$. If $x_-$ is the input and $x_+$ is the output of a Floer trajectory $u$, then
the trajectory is weighted by $T^{E_{\mathrm{top}}(u)}$, with
$E_{\mathrm{top}}$ as in Section~\ref{SecConventions}
(cf.~\eqref{eqGeoTopEn1}).
We now consider local Floer complexes which count only trajectories contained
in an exact neighborhood $V$ of an isolated $1$-periodic component, and are
thus defined over $\bZ$.  We call this the
\emph{unweighted} local Floer complex.  In this subsection we develop and study this invariant. In the next subsection we use Hein's lemma to relate to Floer homology in small action windows. 

To study the local Floer complex we consider an 
enhanced version of a 
notion introduced by 
Pozniak \cite{Pozniak}. We use the letter $\cH$ to 
denote a Floer 
theoretic diagram as introduced in Subsection~\ref{sec:floer_general_hamiltonian}. Let $\cU\subset\cL M$ be open in the $C^\infty$ topology. We refer to  it as  a
\emph{trajectory basin}. Denote by $F_{\cH}(\cU)$ the set of Floer,
continuation, or parametrized-continuation solutions associated with~$\cH$
which map into $\cU$. Let
$S_{\cH}(\cU)$ be the image of
$\bR\times F_{\cH}(\cU)$ under evaluation $(s,u)\mapsto u(s,\cdot)$, and call it
the \emph{Floer sweep} of $\cU$ with respect to $\cH$. For a continuation problem we also allow a \emph{moving trajectory basin},
namely an open subset $\mathcal V\subset\cL M\times\bR$ which is constant
near the two ends.  A solution is contained in $\mathcal V$ when
$(u(s,\cdot),s)\in\mathcal V$ for all $s$, and its sweep is the corresponding
subset of $\cL M\times\bR$.

The \emph{footprint} of a fixed trajectory basin $\cU\subset\cL M$ is
\[
|\cU|
:=
\bigcup_{\ell\in\cU}\ell(S^1)\subset M,
\]
and for a moving basin $\mathcal V\subset\cL M\times\mathbb R$,
\[
|\mathcal V|
:=
\bigcup_{(\ell,s)\in\mathcal V}\ell(S^1)\subset M.
\]
We call this subset of~$M$ the footprint of the trajectory basin.

\begin{df}
We say that $S_{\cH}(\cU)$ is \emph{Floer-theoretically isolated} if
$\overline{S_{\cH}(\cU)}\subset\cU$. In other words, the closure of the Floer sweep of the trajectory basin $\cU$ is contained in $\cU$.
\end{df}

\begin{lm}\label{lmGromovInBasin}
Let $\cM$ be any moduli space associated with $\cH$ whose solutions
are contained in $\cU$. If $S_{\cH}(\cU)$ is isolated, then the 
Gromov--Floer limit of a sequence in $\cM$ is a possibly broken solution all
of whose cylinder components lie in $\cU$. In particular, the corresponding local
moduli space is compact modulo bubbling up to its standard broken boundary. Consequently, all standard compactified Floer boundary identities for differentials, continuations, regular homotopies, and gluing restrict to the local moduli spaces in~$\cU$.
The same conclusion holds for an isolated moving trajectory basin, with its
endpoint slices used on broken levels at the two ends.
\end{lm}
\begin{proof}
Gromov--Floer compactness holds in the ambient space $M$. A sequence of
solutions with image in $\cU$ has a convergent subsequence in the appropriate
Gromov--Floer compactification.  Every cylinder component of the
limit therefore lies in $\overline{S_{\cH}(\cU)}$, which is contained in
$\cU$ by isolation. 

By the semi-positivity convention of Section~\ref{SecConventions}, sphere
bubbling occurs in codimension at least two in the low-dimensional Floer moduli
spaces used here.  It therefore does not alter the compactified boundary
identities defining differentials, continuation maps, regular homotopies, or
their gluings.  Thus the conclusion holds.

\end{proof}

Assume now that for all $t\in S^1$ the trajectory basin $\cU$ maps under $ev_t:\cL M\to M$ into an open subset $W\subset M$ on which $\omega$ is exact.  For any $(H,J)\in\cH$ define the \emph{local Floer complex} $CF^*_{uw}(\cU;F)$ as the direct sum of the integral orientation lines of the $1$-periodic orbits of $H$ contained in~$\cU$, with differential counting index-$1$ trajectories in~$\cU$ with their orientation signs.

We then have the following immediate consequence. 

\begin{prp}[Local Floer complex]
\label{prpLocalFloerAlgebra}
\label{lmLocalContinuationPackage}
Suppose that $S_{\cH}(\cU)$ is isolated. Then:
\begin{enumerate}
\item
for each Floer datum $F=(H,J)\in\cH$, the unweighted local complex
$CF^*_{uw}(\cU;F)$ is 
well defined. That is, $d^2=0$.  We write $HF^*_{uw}(\cU;F)$ for its homology;
\item
every $k\in\cH(F_0,F_1)$ induces a chain map
\[
\Phi_k^{\cU}\colon CF^*_{uw}(\cU;F_0)\to CF^*_{uw}(\cU;F_1);
\]
\item
a regular homotopy from $k^0$ to $k^1$ induces a local chain homotopy between
$\Phi_{k^0}^{\cU}$ and $\Phi_{k^1}^{\cU}$; and
\item
for composable continuation data,
$\Phi_{k_{01}\#_Rk_{12}}^{\cU}\simeq\Phi_{k_{12}}^{\cU}\circ\Phi_{k_{01}}^{\cU}$
for all sufficiently large~$R$.
\end{enumerate}
The induced maps on unweighted local Floer cohomology are therefore independent of
continuation data and functorial.
The same assertions hold for continuation and parametrized-continuation
diagrams in an isolated moving trajectory basin; they give maps and chain
homotopies between the complexes of the two endpoint slices.
\end{prp}
\begin{proof}
This is immediate from Lemma \ref{lmGromovInBasin} and the standard Hamiltonian Floer package. 
\end{proof}

\medskip\noindent\textbf{Locality of the confined construction.}
Let $\cU$ be a trajectory basin (fixed or moving) and let $W\subset M$ be an
open neighborhood of its footprint.  The local Floer complexes, continuation
maps, and regular-homotopy maps associated with~$\cU$ depend only on the
restrictions to~$W$ of the corresponding Hamiltonian and
almost-complex-structure data.  In particular, if two choices of ambient
Floer data agree on a neighborhood of the footprint of the basin, then the
corresponding confined generators, equations, moduli spaces, and induced
local maps are identical.

In particular if $\cH$ is contractible, we obtain a Hamiltonian Floer group
\[
HF^*_{uw}(\cU;\cH)=\varinjlim_{(H,J)\in\cH}HF^*_{uw}(\cU;H,J).
\]

\medskip
\noindent
We now introduce a notion which will allow us to package the confinement arguments used throughout the paper.

\begin{df}[Energy barrier]
\label{dfDatumBarrier}\label{dfEpsilonBarrier}
Let $(H_s,J_s)$ be a monotone Floer datum and let
$\cU\subset\cL M$ be open (or, more generally, a family $\{\cU_s\}$ of opens,
constant for $|s|\gg0$). We say that $(H_s,J_s)$ has an $\epsilon$-barrier
relative to $\cU$ (resp.\ $\{\cU_s\}$) if every finite-energy solution $u$ whose
positive or negative asymptotic orbit lies in the corresponding end
neighborhood, and for which the trajectory $u\not\subset\cU$
(resp.\ the loop $u(s,\cdot)$ fails to lie in $\cU_s$ for some~$s$), satisfies
\[
E^{\mathrm{top}}(u)\ge\epsilon.
\]

For a Floer diagram $\cH$ we say that $(\cH,\cU)$ has an $\epsilon$-barrier if
every Floer, monotone-continuation, and parametrized-continuation datum
occurring in~$\cH$ has an $\epsilon$-barrier relative to~$\cU$ in the sense of
Definition~\ref{dfDatumBarrier}.
\end{df}

The barriers established in this paper are \emph{robust} in the sense that they
depend only on the Floer data on a neighborhood of the footprint of the
trajectory basin~$\cU$ and are unchanged by modifications outside that
neighborhood. 

\begin{df}\label{dfWellIsolatedComponent}
Let $H$ be a smooth Hamiltonian.  A family $\gamma$ of $1$-periodic orbits
of $H$ is a \emph{well-isolated $1$-periodic component} if
\begin{enumerate}
\item $\gamma$ is nonempty, compact, and path-connected;
\item $\gamma$ is contained in a precompact open set $V$ which meets no other
$1$-periodic orbit of $H$, such that $\omega|_V$ is exact and the complexified
canonical bundle is trivial over $V$.
\end{enumerate}
Such a $V$ is a \emph{well-isolating neighborhood} of $\gamma$.
\end{df}

\begin{lm}[Action constancy]
\label{lmActionConstancy}
Let $H$ be a smooth Hamiltonian and let $\gamma$ be a path-connected family of
$1$-periodic orbits contained in an exact open set~$W$, with primitive~$\theta$.
Then
\[
\cA_H^\theta\big|_\gamma
\]
is constant.
\end{lm}

\begin{proof}
The local action is continuous on the periodic-orbit set and its values lie
in the action spectrum, which has measure zero.  The image of a
path-connected set is connected, hence a point.
\end{proof}

From now on every fixed or moving trajectory basin is understood to be path connected and have the
closure of its footprint contained in a chosen exact open subset of~$M$.
This exact neighborhood is auxiliary data and will normally be suppressed
from the notation.

\begin{df}[$\omega$-equivalence and strong isolation]
\label{dfStrongLocalIsolation}
Let
\[
u,v:[0,1]\times S^1\longrightarrow M
\]
be compact cylinders with the same boundary loops.
Write
\[
u\sim_\omega v
\]
when the closed $2$-cycle $u\#\bar v$ lies in $\ker[\omega]$, and write
\[
[u]_\omega
\]
for the corresponding $\omega$-equivalence class.

For any pair $x_\pm$ in~$\mathcal V$, write
\[
[x_-,x_+]_{\mathrm{loc}}
\]
for the unique $\omega$-class of cylinders determined by a path in $\cV$ connecting $x_\pm$. We call $[x_-,x_+]_{\mathrm{loc}}$ the
\emph{local $\omega$-class}.

A Floer-theoretically isolated local datum $(k,\mathcal V)$ is
\emph{strongly isolated} if every finite-energy solution~$u$ of~$k$ with the
 local asymptotics $x_\pm\in\cV$ satisfies
\[
[u]_\omega=[x_-,x_+]_{\mathrm{loc}}
\Longrightarrow
u\subset\mathcal V.
\]
\end{df}

\begin{lm}[Barrier criterion for strong isolation]
\label{lmStrongIsolationBarrier}
\label{rmStrongIsolationEnergy}
Suppose the Floer-theoretically isolated local datum $(k,\mathcal V)$ has
an $\epsilon$-barrier relative to~$\mathcal V$, and the local end actions of
all local generators of all endpoint Hamiltonians appearing in the given
Floer, continuation, or compactly parametrized-continuation
package---computed using a primitive on the chosen exact neighborhood of the
basin footprint---lie in a common interval of width~$<\epsilon$.  Then
$(k,\mathcal V)$ is strongly isolated.
\end{lm}
\begin{proof}
If $[u]_\omega=[x_-,x_+]_{\mathrm{loc}}$, then $u$ is $\omega$-equivalent to a
cylinder in the chosen exact neighborhood of the footprint.  Exactness
therefore identifies the topological energy of~$u$ with the
source-minus-target local action difference, which is~$<\epsilon$, so the
barrier prevents escape.
\end{proof}

Let $V_0$ be a well-isolating neighborhood of a well-isolated component
$\gamma$ of a smooth Hamiltonian $H$.  For a well-isolating neighborhood
$V$ of the path-connected periodic component~$\gamma$, write
\[
\cU_\gamma(V)
:=
\text{the path component of }
\{\ell\in\cL M\mid \ell(S^1)\subset V\}
\text{ containing }\gamma.
\]
Equivalently,
\[
\cU_\gamma(V)
=
\{\ell:S^1\to V\mid
\ell\text{ is freely homotopic in }V
\text{ to an element of }\gamma\}.
\]
This is an open path-connected trajectory basin.

\begin{prp}[Smooth local Floer invariance]
\label{prpSmoothLocalFloerPresentations}
There exist a well-isolating neighborhood $V\Subset V_0$, a convex
$C^2$-neighborhood $\cN$ of~$H$, and a bounded ball $\cB$ of compatible almost
complex structures such that the following hold.
\begin{enumerate}
\item For every regular $(H',J)\in\cN\times\cB$, the local Floer complex
\[
CF^*_{uw}(\cU_\gamma(V);H',J)
\]
is well defined.
\item For any two such regular data, continuation in $\cU_\gamma(V)$ induces a
canonical isomorphism on cohomology, independent of the choice of
continuation data.
\item These isomorphisms satisfy the cocycle condition.
\item The resulting canonically identified group is unchanged if $V$, $\cN$,
or $\cB$ is replaced by another sufficiently small choice.
\end{enumerate}
\end{prp}

\begin{proof}
    Choose nested well-isolating neighborhoods
    \[
    V^-\Subset V\Subset V_0.
    \]
    Then
    \[
    \overline{\cU_\gamma(V^-)}\subset \cU_\gamma(V).
    \]
    Moreover, a Floer solution with one slice in $\cU_\gamma(V^-)$ whose loops
    remain spatially inside~$V^-$ stays in $\cU_\gamma(V^-)$, since
    $s\mapsto u(s,\cdot)$ is a path in loop space. Thus the spatial barrier of
    Lemma~\ref{lmEpsilonSeparation} gives an $\epsilon$-barrier for the basin
    $\cU_\gamma(V^-)$.
    
    By Lemma~\ref{lmActionConstancy}, the action is constant on~$\gamma$.
    Choose a convex $C^2$-neighborhood $\widehat{\cN}$ of~$H$ sufficiently small
    that the preceding barrier is uniform over~$\cB$, all local periodic orbits
    lie in $\cU_\gamma(V^-)$, and their local actions lie in a common interval of
    width~$<\epsilon$. Choose a convex neighborhood
    \[
    \cN\Subset\widehat{\cN},
    \]
    and take this $\cN$ in the statement of the proposition.
    
    If a Floer solution for data in $\cN\times\cB$ is contained in
    $\cU_\gamma(V)$, then both of its ends lie in $\cU_\gamma(V^-)$. Since
    $V_0$ is exact, its topological energy is the source-minus-target local action
    difference and is therefore $<\epsilon$. The barrier forces the solution to
    remain in $\cU_\gamma(V^-)$. Hence
    \[
    S_{\cH}(\cU_\gamma(V))\subset \cU_\gamma(V^-),
    \qquad
    \overline{S_{\cH}(\cU_\gamma(V))}
    \subset \overline{\cU_\gamma(V^-)}\subset \cU_\gamma(V),
    \]
    so $\cU_\gamma(V)$ is Floer-theoretically isolated.
    Proposition~\ref{prpLocalFloerAlgebra} therefore gives the local Floer package,
    and Lemma~\ref{lmStrongIsolationBarrier} shows that the resulting
    presentations, short continuations, and compact parametrized continuations may
    be chosen strongly isolated, using~$V_0$ as the chosen exact neighborhood of
    the footprint.
    
    Now let two regular endpoint data in $\cN\times\cB$ be given. Join their
    Hamiltonians by the straight path in~$\cN$ and choose a path of almost complex
    structures in~$\cB$. After subdividing if necessary,
    Remark~\ref{lmLocalMonotonicization} makes each short piece monotone by a
    spatial constant $c_j$ whose oscillation on that subinterval is small enough
    that the adjusted piece remains in $\widehat{\cN}$ and preserves the
    action-width bound; at subdivision points one uses the canonical
    identification of unweighted local complexes of Hamiltonians differing by a
    spatial constant and then resets the normalization. Spatial translations
    do not change the Hamiltonian vector field or the unweighted continuation
    equation. Thus these short continuations define a map between the endpoint
    local groups. Reversing the path and composing with the canonical
    endpoint identifications shows that this map is an isomorphism. The same
    argument for compact $2$-parameter families, with subdivision and
    short-piece bounds chosen uniformly, shows independence of the path and
    regular perturbations and gives the cocycle condition.
    
    Finally, consider two sufficiently small choices of $(V,\cN,\cB)$. Choose a
    well-isolating neighborhood
    \[
    V'\Subset V_1\cap V_2
    \]
    and shrink the Hamiltonian data so that all relevant local generators lie in
    $\cU_\gamma(V')$; enlarge the bounded ball of almost complex structures if
    necessary so that the choices are compared inside one common ball. Choose the
    presentation over $\cU_\gamma(V')$ strongly isolated, using the common exact
    neighborhood~$V_0$. Every solution counted in either
    $\cU_\gamma(V_i)$ then represents the same local $\omega$-class as the smaller
    presentation, so strong isolation forces it into $\cU_\gamma(V')$.
    Consequently the confined complexes and continuation maps agree after passage
    to this common refinement. This proves independence of all sufficiently small
    choices and, by using a common refinement for three choices, compatibility of
    the canonical identifications.
    \end{proof}
\begin{cy}\label{cyStrongLocalPresentations}
The presentations of Proposition~\ref{prpSmoothLocalFloerPresentations},
together with their short continuations and compact parametrized
continuations, may be chosen strongly isolated.
\end{cy}
\noindent Immediate from the proof of
Proposition~\ref{prpSmoothLocalFloerPresentations}.

\begin{df}\label{dfUwLocalFloer}
The canonically identified group of Proposition~\ref{prpSmoothLocalFloerPresentations}
is denoted
\[
HF^*_{uw}(\gamma;H)
\]
and is called the \emph{unweighted local Floer cohomology} of the well-isolated
periodic component $\gamma$.
\end{df}

\begin{rem}\label{rmEnlargeAlmostComplexBall}
The isolating neighborhoods $V^-\Subset V$ may be held fixed while the ball
$\cB$ of almost complex structures is enlarged, provided $\cN$ is shrunk
accordingly so that the barrier and action-width conditions in the proof of
Proposition~\ref{prpSmoothLocalFloerPresentations} still hold. This allows us to establish $J$-independence of the local Floer cohomology.  The same
applies to continuation data after shrinking the neighborhood of the
Hamiltonian path; compare Proposition~\ref{propLocalContinuation}\textup{(iv)}.
\end{rem}

\subsubsection{Smooth local continuation maps}
We now develop the rest of the Hamiltonian Floer package in the local setting. Isolation for continuation data and for homotopies between them is treated
here as an assumption. Later, in the course of constructing transport along smooth isolated families, we show that these
isolation properties hold for sufficiently small Hamiltonian isotopies and
squares.

\begin{df}[Smooth local continuation datum]
\label{dfLocalFloerContinuation}
Let $\gamma^0$ and $\gamma^1$ be well-isolated periodic components of
smooth Hamiltonians $H^0$ and $H^1$.  A \emph{smooth local
continuation datum} consists of a regular continuation
\[
k=(H_s,J_s)
\]
between sufficiently small regular endpoint data and a moving
trajectory basin
\[
\mathcal V\subset\cL M\times\mathbb R,
\]
constant near the ends, such that the closure of its footprint is contained
in an exact open set and the continuation sweep is Floer-theoretically
isolated in~$\mathcal V$.
We write $\mathcal V^-$ and $\mathcal V^+$ for the two constant end
slices.
\end{df}

\begin{rem}[Monotonicization by spatial constants]
\label{lmLocalMonotonicization}
Given a compactly parametrized family of such continuation data with
underlying Hamiltonians $H^{q,\tau}$ and a fixed precompact neighborhood~$W$
of the footprint, one may add spatial constants so as to make every member of
the family monotone.  More precisely, after subdividing the parameter
interval into sufficiently short subintervals, on each short piece one
chooses a smooth spatial constant $c_j(\tau)$, normalized independently on
that subinterval, so that the adjusted Hamiltonian is monotone there; the
oscillation of $c_j$ on that short subinterval may be made arbitrarily small.
When the endpoints are independent of an auxiliary parameter~$q$, these
choices may be made uniformly in~$q$ over any compact auxiliary parameter
family, and the subdivision and short-piece oscillation bounds may likewise
be chosen uniformly.  Adding a spatial constant leaves Hamiltonian vector
fields, periodic orbits, and the unweighted Floer and continuation equations
unchanged.  Therefore, at each subdivision point, Hamiltonians differing by a
spatial constant have canonically identical unweighted local complexes, after
which the normalization may be reset for the next subinterval.  Consequently
the piecewise adjustment preserves any fixed local action-width margin and
the isolation properties used below.  No claim is made that a single spatial
translation over the whole isotopy can be made arbitrarily small merely by
subdivision.

Strictly speaking, our definitions require monotonicity on all of~$M$. We therefore assume the data are
extended so that the same spatial-constant adjustment works globally.  On
closed~$M$ the needed bound is automatic.  In the noncompact setups of this
paper it holds because ambient extensions use one fixed admissible proper
exterior profile up to spatial constants.
\end{rem}

\begin{df}[Regular homotopy of smooth local continuation data]
\label{dfRegularHomotopyLocalContinuation}
Let $(k^0,\mathcal V)$ and $(k^1,\mathcal V)$ be smooth local
continuation data from $(\gamma^0;H^0)$ to $(\gamma^1;H^1)$ with the
same moving trajectory basin~$\mathcal V$.  A \emph{regular homotopy}
between them is a regular homotopy of continuation data from $k^0$
to~$k^1$ for which the parametrized-continuation sweep remains
Floer-theoretically isolated in this same basin.
\end{df}

If two continuation data are presented with different moving basins,
pass first to a common strongly isolated refinement via
Proposition~\ref{propLocalContinuation}\textup{(iii)} below, and then apply
ordinary regular-homotopy invariance in that common basin.

\begin{prp}[Smooth local continuation maps]
\label{propLocalContinuation}
A smooth local continuation datum induces a canonical map
\[
\Phi_{uw}:
HF^*_{uw}(\gamma^0;H^0)
\longrightarrow
HF^*_{uw}(\gamma^1;H^1),
\]
obtained by applying Proposition~\ref{prpLocalFloerAlgebra} in the
moving basin and identifying the endpoint local groups with
$HF^*_{uw}(\gamma^i;H^i)$ via
Proposition~\ref{prpSmoothLocalFloerPresentations}.

Moreover:

\begin{enumerate}[label=\textup{(\roman*)},ref=\roman*]

\item
regularly homotopic smooth local continuation data, in the sense of
Definition~\ref{dfRegularHomotopyLocalContinuation}, induce the same
map;

\item
let
\[
(k_{01},\mathcal V_{01}),
\qquad
(k_{12},\mathcal V_{12})
\]
be strongly isolated smooth local continuation data with
\[
\mathcal V_{01}^+=\mathcal V_{12}^-,
\]
and suppose that the closures of the footprints of the two continuation
basins and the intermediate endpoint basin lie in one common exact open
set~$W$, with respect to which the intermediate $F_1$ Floer presentation in
the common endpoint basin is likewise strongly isolated.
Let
\[
\mathcal V_{01}\#_R\mathcal V_{12}
\]
denote the moving basin obtained by concatenating the two basins and
inserting the common endpoint basin along a neck of length~$R$.  Then for
all sufficiently large~$R$,
\[
\bigl(
k_{01}\#_Rk_{12},
\mathcal V_{01}\#_R\mathcal V_{12}
\bigr)
\]
is a smooth local continuation datum and induces
$\Phi_{12}\circ\Phi_{01}$.
A constant continuation in a strongly isolated endpoint presentation
induces the identity.
If the endpoint basins agree only after a common strongly isolated
refinement, apply~\textup{(iii)} to both factors first and then
concatenate;

\item
if $\mathcal V'\subset\mathcal V$, the footprint of the larger basin lies in
an exact open set~$W$, this same~$W$ is used as the chosen exact
neighborhood for the smaller basin, and $(k,\mathcal V')$ is strongly
isolated, then $(k,\mathcal V')$ and $(k,\mathcal V)$ induce the same map;

\item
replacing the endpoint data by another sufficiently small choice in
the sense of Proposition~\ref{prpSmoothLocalFloerPresentations} does
not change~$\Phi_{uw}$.

\end{enumerate}
\end{prp}

\begin{proof}
The confined chain map is Proposition~\ref{prpLocalFloerAlgebra}
applied to the isolated moving basin.  Clause~\textup{(i)} is the
parametrized-continuation clause of that proposition.

For~\textup{(ii)}, if ordinary isolation of the long-neck sweep failed,
there would be $R_\nu\to\infty$ and long-neck solutions contained in the
concatenated basin with a sequence of slices converging to its boundary.
Exactness of~$W$ and the uniform local action bound give a uniform
topological-energy bound, so Gromov--Floer compactness and translation
along the neck produce a broken limit consisting of $k_{01}$ levels,
intermediate $F_1$ Floer levels, and $k_{12}$ levels.  Exactness of~$W$
identifies the local $\omega$-classes of these broken levels with those of
the corresponding pieces; strong isolation therefore confines each level to
the corresponding basin.  Ordinary isolation of the three pieces then
contradicts convergence to the boundary.  Thus the long-neck datum is
isolated for large~$R$, and the standard long-neck gluing identity yields
the composite.  The constant-path assertion is the same argument in a
strongly isolated endpoint presentation.

For~\textup{(iii)}, every solution counted in~$\mathcal V$ has image
in~$W$, hence represents the local $\omega$-class of the smaller basin
$\mathcal V'$.  Strong isolation of $(k,\mathcal V')$ therefore forces it
into~$\mathcal V'$.  The reverse inclusion is tautological.  Thus the
confined moduli spaces and maps agree.

Clause~\textup{(iv)} is
Proposition~\ref{prpSmoothLocalFloerPresentations}.
\end{proof}

\subsubsection{Transport in smooth isolated families}

We now globalize the local Floer package to arbitrary smooth isolated isotopies and squares.

\begin{df}[Isolated Hamiltonian parameter family]
\label{dfIsolatedHamiltonianFamily}
Let $P=[0,1]$ or $P=[0,1]^2$.  A \emph{smooth isolated Hamiltonian
$P$-family} consists of a smooth family $p\mapsto H^p$ of smooth Hamiltonians,
a compact subset
\[
\Gamma\subset\cL M\times P,
\]
and an open worldtube $\Gamma\subset\cU\subset\cL M\times P$ such that
\[
\Gamma^p=\Per(H^p)\cap\cU^p
\]
is a nonempty compact well-isolated periodic component for every $p\in P$.
For $P=[0,1]$ this is a \emph{smooth isolated isotopy}.  A square family which
is constant on the two vertical edges is a homotopy of smooth isolated
isotopies relative to the endpoints.
\end{df}

\begin{prp}[Transport along a smooth isolated family]
\label{propLocalFloerIsotopyTransport}
A smooth isolated isotopy $\Gamma$ induces a canonical isomorphism
\[
\iota_\Gamma:
HF^*_{uw}(\Gamma^0;H^0)
\xrightarrow{\;\sim\;}
HF^*_{uw}(\Gamma^1;H^1).
\]
It is functorial under concatenation, is the identity for a constant family,
and is inverted by reversal.  If $P=[0,1]^2$ is an isolated square and
$b,r,\ell,t$ denote its bottom, right, left, and top edges, then
\[
\iota_r\circ\iota_b=\iota_t\circ\iota_\ell.
\]
In particular, the transport depends only on the isolated homotopy class
relative to the endpoints.
\end{prp}

\begin{proof}
Cover $P$ by finitely many sufficiently small parameter boxes.  Compactness
of~$\Gamma$ and fiberwise well-isolation let
Corollary~\ref{cyStrongLocalPresentations} be applied uniformly on each box,
giving a common exact strongly isolated presentation for all short edges and
rectangles; arrange matching endpoint slices for adjacent edges.
Remark~\ref{lmLocalMonotonicization} produces globally monotone
representatives, after shortening the parameter boxes if necessary.  On
overlaps, Proposition~\ref{propLocalContinuation}\textup{(iii)--(iv)} identify
the edge maps via a common smaller strongly isolated presentation.

Transport is the composite over a finite subdivision.  Refinement
invariance follows from Proposition~\ref{propLocalContinuation}\textup{(ii)}
for the long-neck concatenation of adjacent refined edges, followed by a
regular homotopy with the direct short-edge continuation in a common strong
refinement and
Proposition~\ref{propLocalContinuation}\textup{(i)}.  Reversal and the
identity are the same argument.  For a square, the short rectangles identify
the two boundary composites, and multiplying over the grid cancels interior
edges.
\end{proof}

\subsubsection{The Morse--Bott case}
Let $H$ be a smooth Hamiltonian and let $\gamma$ be a well-isolated
Morse--Bott component of its $1$-periodic orbits. Choose a well-isolating
neighborhood and a homotopy class of trivializations of the complexified
canonical bundle there. Denote by $\iota_{CZ}(\gamma)$ the transverse
Conley--Zehnder index, normalized so that for a sufficiently small
Morse--Bott perturbation associated with a Morse function $f$ on $\gamma$,
the orbit $x_y$ corresponding to $y\in\operatorname{Crit}(f)$ satisfies
\[
\iota_{CZ}(x_y)
=
\iota_{CZ}(\gamma)+i_{\mathrm{Morse}}(y).
\]

For $x\in\gamma$, let $D_x$ be the standard Morse--Bott capping operator.
The determinant lines $\det(D_x)$ form a real line bundle over $\gamma$.
Let $\cL$ be its integral orientation local system.
See \cite{SeidelBook} for orientation lines and
\cite[Sections~9.3 and~10]{SchmaeschkeThesis} for the clean-intersection
Morse--Bott orientation package.

\begin{tm}[Morse--Bott identification]\label{tmMorseBottCascades}
There is a natural isomorphism
\[
HF^k_{uw}(\gamma;H)
\simeq
H^{k-\iota_{CZ}(\gamma)}(\gamma;\cL).
\]
Moreover, for a smooth isolated Morse--Bott family whose union is a smooth
fiber bundle over the parameter interval, this identification intertwines
local Floer transport with the Gauss--Manin transport on cohomology with
coefficients in the corresponding family of local systems $\cL$.
\end{tm}

\begin{proof}
Under the graph construction, the local Hamiltonian Floer problem near
$\gamma$ is identified with the local Lagrangian Floer problem near the clean
intersection component corresponding to $\gamma$ of the diagonal with the
graph of $\phi_H^1$.  Po\'zniak's clean-intersection theorem
\cite[Theorem~3.4.11]{Pozniak} identifies, after a sufficiently small
Morse--Bott perturbation, this local Floer complex with the Morse complex of
$\gamma$; the determinant-line refinement gives coefficients in
$\cL$.  See
\cite[Sections~9.3--10]{SchmaeschkeThesis} for the orientation and
continuation constructions.  With the normalization above, the grading of a
critical point $y$ is shifted by $\iota_{CZ}(\gamma)$, giving the displayed
cohomological grading.  The same continuation comparison gives the stated
naturality for smooth Morse--Bott families.
\end{proof}

\subsection{From truncated to unweighted Floer cohomology}\label{SubsecTruncUwFloer}
Recall $\Lambda_{[0,\epsilon)}=T^{-\epsilon}\Lambda_{>0}/\Lambda_{>0}$.
Every element has a unique representative with finite support in
$(-\epsilon,0]$, so, as an abelian group,
\[
\Lambda_{[0,\epsilon)}\cong
\bigoplus_{-\epsilon<\lambda\le0}\bZ\,T^\lambda.
\]
It is therefore free, hence flat over $\bZ$, and tensoring with it commutes
with homology.

\subsubsection*{Algebraic interlude: truncated colimits of valued complexes}
We establish two algebraic facts about directed systems of complexes equipped with a real-valued valuation which will be used throughout the paper to relate summands associated with a periodic component in the truncated Floer complex and the corresponding unweighted complex. 

\medskip\noindent\textbf{Setup.}
Let $(C_i)_{i\in I}$ be a directed system of chain complexes of abelian groups with chain maps $\kappa_{ij}\colon C_i\to C_j$ for $i\leq j$. We refer to $C_i$ as an \emph{unweighted} complex. Assume that each $C_i$ is a free abelian group, with a distinguished basis $G_i$, and that each generator $x\in G_i$ is assigned a real number $v_i(x)$. For $n_x\in\bZ$, we extend $v_i$ to nonzero chains by
\[
v_i\Bigl(\sum_x n_x x\Bigr)
:=
\min\{v_i(x):n_x\neq 0\}.
\]
We assume that the differential and structural maps are nondecreasing  with respect to these weights:
\[
v_i(dx)\geq v_i(x),
\qquad\text{respectively}\qquad v_j(\kappa_{ij}(x))\geq v_i(x).
\]

From this we form two directed systems with coefficients in $\Lambda_{[0,\epsilon)}$, analogous to the Floer-theoretic truncated complex and the unweighted complex tensored with the truncated module:
\begin{itemize}
\item \textbf{Truncated complex} $C_i^{[0,\epsilon)}$: as a module it is $C_i\otimes\Lambda_{[0,\epsilon)}$. If for $x\in G_i$ we have 
\[
dx=\sum_{y\in G_i} d_{xy}y,
\]
its weighted differential is
\[
d^w x=\sum_y d_{xy}T^{v_i(y)-v_i(x)}y.
\]
Similarly, if
\[
\kappa_{ij}(x)=\sum_y\kappa_{xy}y,
\]
the weighted structural map is
\[
\kappa_{ij}^w(x)
=
\sum_y\kappa_{xy}T^{v_j(y)-v_i(x)}y.
\]
The valuation hypotheses ensure that all exponents occurring here are non-negative.
\item \textbf{Unweighted tensored with truncated module} $C_i\otimes \Lambda_{[0,\epsilon)}$: the same generators as $C_i$, with coefficients in $\Lambda_{[0,\epsilon)}$; differential and structural maps $\eta_{ij}=\kappa_{ij}\otimes\id$.
\end{itemize}

In the Floer-theoretic applications below, the valuation is
\[
v(x):=-\cA(x).
\]
Thus, for a trajectory from an input $x_-$ to an output $x_+$,
\[
v(x_+)-v(x_-)
=
\cA(x_-)-\cA(x_+).
\]

\begin{lm}\label{lmAlgTruncUw}
In the situation above, assume that after passing to a cofinal sequence, the valuation of every generator of $C_i$ lies in $(a-\delta_i,a+\delta_i)$ with $\delta_i\to 0$. Then there exists an isomorphism
$$
\varinjlim_i H^*(C_i^{[0,\epsilon)})\;\xrightarrow{\;\sim\;}\; \bigl(\varinjlim_i H^*(C_i)\bigr)\otimes \Lambda_{[0,\epsilon)}.
$$
\end{lm}
\begin{proof}
After passing to the cofinal sequence in the hypothesis, we may assume that
the index set is $\bN$ and that every generator $x$ of $C_i$ satisfies
\[
a-\delta_i<v_i(x)<a+\delta_i,
\qquad
\delta_i\longrightarrow 0.
\]

We first discard a transient part of the system. Let $C_i^+\subset C_i$
be the submodule generated by those $x$ with $v_i(x)>a$. By valuation monotonicity $C_i^+$ is subcomplex and $\kappa_{ij}(C_i^+)\subset C_j^+$. Moreover, $\varinjlim_i C_i^+=0$. Indeed, if $v_i(x)=a+\rho$ with $\rho>0$ and $\delta_j<\rho$, then $v_j(\kappa_{ij}(x))\geq a+\rho$ while every generator of $C_j$ has valuation $<a+\rho$, so $\kappa_{ij}(x)=0$. The same holds for the truncated subcomplexes $C_i^{+,[0,\epsilon)}$. Exactness of filtered colimits lets us replace $C_i$ by $C_i/C_i^+$ without changing either side of the claimed isomorphism. We may thus assume without loss of generality that every generator satisfies
\begin{equation}\label{eqAlgOneSidedConcentration}
a-\delta_i<v_i(x)\leq a.
\end{equation}

Write $W_i:=C_i^{[0,\epsilon)}$ and $U_i:=C_i\otimes\Lambda_{[0,\epsilon)}$. Define $A_i:W_i\to U_i$ on generators by
\begin{equation}\label{eqAlgWeightingMap}
A_i(x):=T^{a-v_i(x)}x.
\end{equation}
The exponent is nonnegative by~\eqref{eqAlgOneSidedConcentration}. A direct expansion gives $A_i\circ d^w=d\circ A_i$ and $A_j\circ\kappa_{ij}^w=(\kappa_{ij}\otimes\id)\circ A_i$, so the $A_i$ induce a chain map
\[
A:\varinjlim_i W_i\longrightarrow\varinjlim_i U_i.
\]
We show that $A$ is an 
isomorphism.
\smallskip\noindent
\emph{Surjectivity.}
Let $z\in U_i$, represented with Novikov exponents in $(-\epsilon,0]$. Choose $\rho>0$ so that every such exponent satisfies $\lambda>-\epsilon+\rho$, pick $j\geq i$ with $\delta_j<\rho$, and let $z_j\in U_j$ be the image of $z$. Set
\[
B_j(T^\lambda x):=T^{\lambda+v_j(x)-a}x
\]
on generators. Then $\lambda+v_j(x)-a\in(-\epsilon,0]$, so $B_j(z_j)\in W_j$ and $A_j(B_j(z_j))=z_j$. Thus $A$ is surjective.

\smallskip\noindent
\emph{Injectivity.}
Let $w\in W_i$ with image zero under $A$ in $\varinjlim_i U_i$. After a structural map we may assume $A_j(w_j)=0$ in $U_j$ for $w_j:=\kappa_{ij}^w(w)$. Writing $w_j=\sum_{\lambda,x}c_{\lambda,x}T^\lambda x$, vanishing means $\lambda+a-v_j(x)>0$ for every nonzero term; choose $\rho>0$ with $\lambda+a-v_j(x)\geq\rho$ throughout. For $k\geq j$ with $\delta_k<\rho$ and any generator $y$ occurring in $\kappa_{jk}(x)$, the corresponding exponent in $\kappa_{jk}^w(T^\lambda x)$ is
\[
\lambda+v_k(y)-v_j(x)
=
\bigl(\lambda+a-v_j(x)\bigr)+\bigl(v_k(y)-a\bigr)
>
\rho-\delta_k
>
0,
\]
hence zero in $\Lambda_{[0,\epsilon)}$. Thus $\kappa_{jk}^w(w_j)=0$ in $W_k$, so $w=0$ in $\varinjlim_i W_i$.

We obtain a chain isomorphism $\varinjlim_i W_i\cong\varinjlim_i U_i$. Taking homology and using flatness of $\Lambda_{[0,\epsilon)}$ over $\bZ$ yields the claimed isomorphism.
\end{proof}

The next lemma is a variant of Lemma~\ref{lmAlgTruncUw} for chain maps between two such systems.

\begin{lm}\label{lmAlgTruncUwMaps}
Suppose we have the setting of Lemma~\ref{lmAlgTruncUw} for two directed systems: $(C_i^1,\kappa_{ij}^1)$ and $(C_i^2,\kappa_{ij}^2)$ with valuations $v_i^1$, $v_i^2$ on generators, concentrating, after passing to cofinal sequences, to $a\in\bR$ and $b\in\bR$ respectively. Suppose in addition we have chain maps $\phi_i\colon C_i^1\to C_i^2$ compatible with the systems ($\kappa_{ij}^2\circ \phi_i=\phi_j\circ \kappa_{ij}^1$) that respect the valuations in the following sense: for every generator
$x$ of $C_i^1$ and every generator $y$ occurring with nonzero coefficient
in $\phi_i(x)$,
\[
v_i^2(y)\geq v_i^1(x).
\]
Write $\Phi$ for the induced map
$\varinjlim_i H^*(C_i^{1,[0,\epsilon)})\to\varinjlim_i H^*(C_i^{2,[0,\epsilon)})$
and $\Phi_{uw}$ for the unweighted colimit map
$\varinjlim_i H^*(C_i^1)\to\varinjlim_i H^*(C_i^2)$.
\begin{enumerate}
\item If $b\ge a$, then under the isomorphisms of Lemma~\ref{lmAlgTruncUw},
\[
\Phi
=
T^{b-a}
\bigl(\Phi_{uw}\otimes\id_{\Lambda_{[0,\epsilon)}}\bigr).
\]
\item If $b<a$, then valuation monotonicity and concentration imply that the
maps vanish on a sufficiently late cofinal tail, hence
\[
\Phi_{uw}=0,\qquad\Phi=0.
\]
\end{enumerate}
\end{lm}
\begin{proof}
As in the proof of Lemma~\ref{lmAlgTruncUw}, discard the transient summands
$C_i^{k,+}$ generated by valuations strictly above the concentration value.
Exactness of filtered colimits identifies both sides with the corresponding
colimits after this quotient.  Compatibility of the $\phi_i$ with the
structural maps ensures that a source transient term dies in the target
colimit as well: once a source transient element~$x$ satisfies
$\kappa^1_{ij}(x)=0$ for $j$ large enough that the valuation gap exceeds the
concentration error, one has
\[
\kappa^2_{ij}\phi_i(x)
=
\phi_j\kappa^1_{ij}(x)
=0.
\]
Thus the comparison with the quotient systems is valid at the colimit level,
and we may assume both systems are one-sidedly concentrated as in
Lemma~\ref{lmAlgTruncUw}.

If $b<a$, take the source and target concentration errors sufficiently smaller
than $a-b$.  A nonzero coefficient of a generator~$y$ in the image of a
generator~$x$ would require simultaneously
\[
v_2(y)\ge v_1(x)
\]
by valuation monotonicity and
\[
v_2(y)\approx b<a\approx v_1(x)
\]
by concentration, which is impossible.  Hence $\phi_i=0$ for all large~$i$, so
$\Phi_{uw}=0$ and $\Phi=0$.

If $b\ge a$, the valuation condition makes the weighted maps
$\phi_i^{tr}\colon C_i^{1,[0,\epsilon)}\to C_i^{2,[0,\epsilon)}$ well defined.
Using the diagonal changes of basis $\alpha_i^k$ from the proof of
Lemma~\ref{lmAlgTruncUw}, one has on generators
$\alpha_i^1(x)=T^{a-v_i^1(x)}x$ and $\alpha_i^2(y)=T^{b-v_i^2(y)}y$, and the square
\[
\xymatrix{
C_i^{1,[0,\epsilon)}\ar[r]^{\alpha_i^1}\ar[d]_{\phi_i^{tr}} & C_i^1\otimes \Lambda_{[0,\epsilon)}\ar[d]^{T^{b-a}(\phi_i\otimes\id)} \\
C_i^{2,[0,\epsilon)}\ar[r]_{\alpha_i^2} & C_i^2\otimes \Lambda_{[0,\epsilon)}.
}
\]
commutes: on a generator~$x$, both composites send the coefficient of~$y$ in
$\phi_i(x)$ to that coefficient times $T^{b-v_i^1(x)}$. Passing to colimits in
homology yields $\Phi=T^{b-a}(\Phi_{uw}\otimes\id)$.
\end{proof}


\subsubsection{Truncated local blocks}

The next corollary records the relation between unweighted local Floer
cohomology and truncated Floer cohomology in small Novikov windows.  Let
$\gamma$ be a well-isolated component of a smooth Hamiltonian~$H$, with
well-isolating neighborhood~$V$, and write $\Phi_{uw}$ for the unweighted map
of a smooth local continuation datum.  After an action translation we work in
windows~$\Lambda_{[0,\delta)}$; the same statements hold for
$\Lambda_{[c,c+\delta)}$.  Here $HF^*_{[0,\delta)}(V;H)$ is the ambient
truncated block of Lemma~\ref{lmEpsilonSeparation}.

\begin{cy}[Truncated--unweighted local blocks]
\label{propLocalContinuationTruncated}
\begin{enumerate}[label=\textup{(\roman*)},ref=\roman*]
\item
For every sufficiently small $\delta>0$,
\[
HF^*_{[0,\delta)}(V;H)
\cong
HF^*_{uw}(\gamma;H)\otimes\Lambda_{[0,\delta)}.
\]
\item
For every sufficiently short confined continuation with endpoint action
centers $A_0,A_1$,
\[
\Phi_{[0,\delta)}^{\mathrm{loc}}
=
\begin{cases}
T^{A_0-A_1}
\bigl(\Phi_{uw}\otimes\id_{\Lambda_{[0,\delta)}}\bigr),
& A_0\ge A_1,\\
0,
& A_0<A_1,
\end{cases}
\]
the second case after passing to a sufficiently far cofinal tail.
\end{enumerate}
\end{cy}

\begin{proof}
For~\textup{(i)}, fix a primitive~$\theta$ on a well-isolating neighborhood.
By Lemma~\ref{lmActionConstancy}, $A_{\gamma,\theta}$ is well defined, and for
$C^2$-small regular data the local actions concentrate at this value.
Proposition~\ref{prpSmoothLocalFloerPresentations} and
Lemma~\ref{lmEpsilonSeparation} supply a confined truncated local block in
every sufficiently small window~$\Lambda_{[0,\delta)}$.
Lemma~\ref{lmAlgTruncUw}, applied with valuations $-\cA^\theta$ concentrating
at $-A_{\gamma,\theta}$, identifies that block with
$HF^*_{uw}(\gamma;H)\otimes\Lambda_{[0,\delta)}$.

For~\textup{(ii)}, start from a smooth local continuation datum as in
Definition~\ref{dfLocalFloerContinuation}.  Choose cofinal regular endpoint
presentations and compatible regular realizations; their confined unweighted
maps agree by Propositions~\ref{prpSmoothLocalFloerPresentations}
and~\ref{propLocalContinuation}, and we write~$\Phi_{uw}$ for this common map.
Assume the Novikov window is small enough that the ambient continuation
preserves the endpoint local spans and that every trajectory contributing to
the local-to-local block lies in the chosen moving basin (e.g.\ a
sufficiently short path, so that the barrier exceeds the window width), and
that the endpoint local actions concentrate at $A_0,A_1$ with respect to a
primitive on an exact neighborhood of the footprint.  Confinement identifies
the ambient local-to-local block with the confined continuation map.
Lemma~\ref{lmAlgTruncUwMaps}, applied with valuations $-\cA^\theta$
concentrating at $-A_0$ and $-A_1$, yields the displayed source-minus-target
Novikov factor.
\end{proof}

%% file: sections/03-infinitesimal-symplectic.tex

\section{Finite-slope local theory near Reeb components}\label{secInfFH}
\label{secLocalReeb}

Let $K\subset M$ be a compact set with smooth boundary of positive contact type. The remainder of the paper adapts the small-window formalism of Section~\ref{SecInfFloer} to the indicator Hamiltonian
\begin{equation*}
H_K(x):=\begin{cases}
0&\quad x\in K,\\
\infty&\quad x\not\in K.
\end{cases}
\end{equation*}
Because $H_K$ is discontinuous, we first approximate $H_K$ by finite slope S-shaped Hamiltonians introduced in Section~\ref{SecTruncSH}. In this sextion we focus on the local Floer theory of such Hamiltonians near the boundary $\partial K$. The unweighted local groups that appear are refinements of those associated by \cite{Mclean2012} to isolated Reeb orbits. We then develop the barrier machinery and the truncared-unweigheted story. 

\subsection{Reeb components and local slope-adapted Hamiltonians}\label{SubsecSHReeb}

Let $\alpha$ be a contact form on $\partial K$ such that $d\alpha=\omega|_{\partial K}$, and let $R$ denote its Reeb vector field. Throughout we consider \emph{Reeb components} as in the introduction: $C^\infty$-isolated path-connected components of the closed $-R$-orbits on which the Chern class vanishes along closed paths. 

We formally write $\gamma_0$ for the critical block
corresponding to the constant orbits in~$K$. Its actual loop-space path
components correspond to the connected components of~$K$. We denote by
$\cR(K)$ the set of all Reeb components including $\gamma_0$. Our isolation hypothesis implies that we can enumerate
the nonconstant Reeb components
$\gamma_1,\dots$.

For a contact form $\alpha$ and a Reeb component $\gamma_i$, write $T(\gamma_i,\alpha):=-\int\beta_i^*\alpha$ for $\beta_i$ any closed orbit in $\gamma_i$. The choice of $\beta_i$ doesn't matter. Indeed, the function
\[
x\longmapsto-\int_{S^1}x^*\alpha
\]
is continuous on the Reeb component~$\gamma_i$, its values belong to the Reeb
action spectrum, and the Reeb action spectrum has measure zero
\cite{HoferZhender}. Since $\gamma_i$ is path connected, its image is a
connected measure-zero subset of~$\bR$, hence a point. 

We refer to $T(\gamma_i,\alpha)$ as the \emph{period} of $\gamma_i$ with respect to $\alpha$. Ordering the finitely many components of period at most any fixed bound so that $T_i\leq T_{i+1}$, we refer to the set $\{T_i\}$ as the \emph{period spectrum of $\partial K$ with respect to the contact form $\alpha$}. If $\gamma_i$ is a critical component we take $T_i=0$.

Before proceeding we point out that the set of closed orbits of $-R$ depends only on $\omega$ and $\partial K$, not on the choice of contact form. Indeed the characteristic line field is determined by $\omega$ and $\partial K$, and we always orient it by $-R$. The choice of $\alpha$ affects the period of each periodic point.  It will be important for the discussion below to allow $\alpha$ to change. We shall assume that $\alpha$ is of \emph{positive contact type} meaning that the Liouville field $Z$ determined by $\iota_Z\omega=\alpha$ points outwards of $K$. Unless explicitly stated otherwise, every primitive used below to define Reeb periods or a Reeb-period filtration is understood to be of positive contact type on the relevant boundary component.

A choice of contact form $\alpha$ on $\partial K$ gives rise to a radial coordinate $\rho$ in a neighborhood $V$ of $\partial K$ in which the symplectic form is given by $\omega|_V=d(e^{\rho}\alpha)$ so that $\partial K$ is given by $\{\rho=0\}$. We call such a neighborhood $V$ a \emph{symplectization neighborhood}. In this neighborhood we have a primitive $\theta:=e^\rho\alpha$ of $\omega$, which we use to define the action functional. 

In this section we consider local Floer theory for Hamiltonians on~$M$ which
have a specified form near a Reeb component~$\gamma$.  More precisely, we
consider triples $I=(a,\alpha,\epsilon)$ where $a>0$ is a slope, $\alpha$ is a
positive contact form on~$\partial K$, and $\epsilon$ is a collar width.  We call a
parameter $I=(a,\alpha,\epsilon)$ \emph{$\gamma$-admissible} if
$a>2T_\alpha(\gamma)$.
\label{dfGammaAdmissibleParameter}

\begin{df}[Locally slope-adapted Hamiltonians]
\label{dfLocallySlopeAdapted}
Fix an isolated Reeb component~$\gamma$ and a $\gamma$-admissible parameter
$I=(a,\alpha,\epsilon)$.  Choose a symplectization chart
\[
\Psi:U\times(-r,r)\longrightarrow M,
\qquad
\theta=e^\rho\alpha,
\]
around~$\operatorname{ev}(\gamma)$, with $0<r\le\epsilon$.

\smallskip\noindent\textbf{Continuous finite-slope Hamiltonians.}
A continuous Hamiltonian~$H$ on~$M$ is of \emph{finite-slope form near~$\gamma$}
for~$I$ if on the chart it is radial, $H=h\circ e^\rho$, with
\[
h(e^\rho)
=
\begin{cases}
0,& \rho\le0,\\
a\bigl(e^\rho-1\bigr),& 0\le\rho<r.
\end{cases}
\]
No condition is imposed outside the chart. In particular, the outer bend and
plateau of the global finite-slope models of Section~\ref{SecTruncSH} are not
part of the present definition.

\smallskip\noindent\textbf{Smoothings.}
Let $\eta>0$ be sufficiently small (also below a positive lower bound for
nonconstant Reeb periods, and with $\eta/8<r$).  A smooth Hamiltonian~$H$
on~$M$ is an \emph{$\eta$-smoothing} of such a continuous Hamiltonian, or
\emph{locally slope-adapted near~$\gamma$} for~$(I,\eta)$, if on the same
chart one has $H=h\circ e^\rho$ and:
\begin{enumerate}
\item $-\eta/8<H<0$ on $\Psi\bigl(U\times(-r,0]\bigr)$, and
$0<h'\le\eta$ on~$[e^{-r},1]$;
\item on~$[1,e^{\eta/8}]$, $h$ is strictly convex; equivalently, $h'$ is
strictly increasing from a value in~$(0,\eta]$ to the slope~$a$;
\item $h'\equiv a$ and $h\ge0$ on~$(e^{\eta/8},e^{r})$.
\end{enumerate}
No condition is imposed outside the chart.  Conditions~(i)--(iii) force
uniform convergence on compact subsets of the chart to a continuous
finite-slope Hamiltonian for~$I$ as $\eta\to0$.  Barrier tubes below are
taken inside the constant-slope shell of~(iii).
\end{df}

From now on, for each isolated Reeb component~$\gamma$ and each
$\gamma$-admissible parameter~$I$, we fix once and for all a continuous
Hamiltonian~$H_I$ on~$M$ of finite-slope form near~$\gamma$ for~$I$, and for
each sufficiently small~$\eta$ a smooth $\eta$-smoothing~$H_{I,\eta}$.  No
condition is imposed on these Hamiltonians outside the controlled collar of
Definition~\ref{dfLocallySlopeAdapted} other than smoothness and smooth dependence on $\eta$.  
None of
the claims of this section depend on the particular smooth family chosen.
Section~\ref{SecTruncSH} further specializes the choice to obtain global finite-slope indicator
models. When the underlying domain is fixed we
omit it from the notation, writing~$H_I$ and~$H_{I,\eta}$.
Moreover, for a compact smooth parameter family $p\mapsto I^p$, the locally
slope-adapted smoothings $H_{I^p,\eta}$ may and will be chosen smoothly in
$(p,\eta)$, and relative to any parameter-boundary portion on which the data
have already been fixed.

\medskip\noindent\textbf{Action.}
For a loop~$\ell$ contained in an exact open set with primitive~$\theta$, write
\begin{equation}\label{eqDefActionHF}
    \cA_H^\theta(\ell)
    =
    -\int_{S^1}\ell^*\theta
    -
    \int_0^1H(\ell(t))\,dt.
\end{equation}
We write~$\cA_H$ when the primitive is clear.  Since $H_I$ is nonsmooth at
$\rho=0$, we do not regard the objects there as Hamiltonian periodic orbits;
rather, $\{\rho=0\}$ is the limiting Reeb/action center of period and action
$T_\alpha(\gamma)$.  For an $\eta$-smoothing~$H_{I,\eta}$, the inner orbit
occurs where $h_\eta'=T_\alpha(\gamma)$, with radial coordinate
$\rho_\eta\to0$ and
\begin{equation}\label{eqConvexAction}
\cA_{H_{I,\eta}}\to T_\alpha(\gamma).
\end{equation}

\medskip\noindent\textbf{Local Reeb neighborhoods.}

For a loop $\ell$ contained in an exact open set with primitive $\theta$, write
\[
T_\theta(\ell):=-\int_{S^1}\ell^*\theta.
\]
We call $T_\theta(\ell)$ its period coordinate. On a Reeb orbit this agrees
with its positive Reeb period.

\begin{df}\label{dfAdmissibleNeighborhood}
Let $\gamma$ be an isolated Reeb component.
Given an open set $W\subset M$ containing the evaluation image of~$\gamma$, on
which the complexified canonical bundle is trivial and $\omega$ is exact, a
primitive $\theta$ of $\omega$ on~$W$, and an open period interval
$\mathcal{P}$ containing $T_\theta(\gamma)$, write
\[
\cU_\gamma(W,\mathcal{P})
\]
for the path component containing~$\gamma$ of
\[
\{\ell\in\cL M:
\ell(S^1)\subset W,\;
T_\theta(\ell)\in\mathcal{P}\}.
\]
Any such set is called an \emph{admissible neighborhood} of~$\gamma$.
Every loop in this component is freely homotopic, hence homologous in~$W$, to
a loop in~$\gamma$.
Cylinders inside~$\gamma$ have zero $\omega$-area, so Stokes gives constancy
of $-\int\theta$ on~$\gamma$ and makes $T_\theta(\gamma)$ well defined.  If
$\theta$ is replaced by another primitive~$\theta'$, then for every loop
$\ell$ in the component,
\[
T_{\theta'}(\ell)-T_{\theta'}(\gamma)
=
T_\theta(\ell)-T_\theta(\gamma),
\]
since $\ell$ is homologous in~$W$ to a loop in~$\gamma$.
Thus the admissible neighborhood is independent of the primitive after
translating the period interval by the corresponding constant.
\end{df}

\begin{rem}\label{rmAdmissibleProperties}
An admissible neighborhood is open and path connected in the $C^\infty$ loop
topology; the defining spatial and period conditions are already $C^1$-open.
The intersection of two admissible neighborhoods contains an admissible
neighborhood: given $\cU_1=\cU_\gamma(W_1,\mathcal{P}_1)$ and
$\cU_2=\cU_\gamma(W_2,\mathcal{P}_2)$, after restricting to $W_1\cap W_2$
and expressing both period coordinates with respect to one primitive there
(translating one period interval by the corresponding constant), the
intersection of the two translated intervals contains an open interval about
$T_\theta(\gamma)$.  Taking the path component containing~$\gamma$ of the
corresponding smaller set yields an admissible neighborhood contained in
$\cU_1\cap\cU_2$: every path in that smaller set is also a path in each of the
two larger underlying sets, so this component lies in both.
\end{rem}

\subsection{The finite-slope local package}
\label{SubsecCorFloerSweep}

Throughout this subsection, a parameter $I$ attached to a Reeb component
$\gamma$ is understood to be $\gamma$-admissible.  Fix such $\gamma$ and $I$.
All  constructions in this section depend only on the restriction of the finite-slope
model to a neighborhood of~$\gamma$.

\begin{prp}[Finite-slope local package]
\label{prpFiniteSlopeLocalPackage}
Fix a nonconstant isolated Reeb component $\gamma$ of period
$T=T_{\alpha}(\gamma)$ and a $\gamma$-admissible finite-slope parameter
$I=(a,\alpha,\epsilon)$.  For every sufficiently small positive smoothing
parameter $\eta$, the fixed $\eta$-smoothing~$H_{I,\eta}$ has a unique inner
radial incarnation
\[
\gamma_{I,\eta}
\subset
\Per(H_{I,\eta})
\]
corresponding to~$\gamma$.  It is a well-isolated periodic component of
$H_{I,\eta}$.

For all sufficiently small positive $\eta$, the smooth local groups
\[
HF^*_{uw}(\gamma_{I,\eta};H_{I,\eta})
\]
are canonically identified by the transport of
Proposition~\ref{propLocalFloerIsotopyTransport}.  The identification is
independent of the positive smoothing path, of the choice of the fixed
family~$H_{I,\eta}$, and of the sufficiently small regular local Floer data.
We denote their common value by
\begin{equation}
\label{eqDefUwLocalGammaI}
HF^*_{uw}(\gamma;I).
\end{equation}

The construction depends only on the finite-slope data near~$\gamma$ and not
on the unspecified exterior of~$H_{I,\eta}$.
\end{prp}

\begin{proof}
By Definition~\ref{dfLocallySlopeAdapted}\textup{(ii)},  the radial derivative $h'$ is strictly increasing from a value
below~$T$ to the slope~$a>T$.  Hence there is a unique radius $r_{\eta}(T)$
with
\[
h'(r_{\eta}(T))=T.
\]
Define $\gamma_{I,\eta}$ to be the radial incarnation of~$\gamma$ at that
level.

We check that this component is well isolated in the spatial sense of
Definition~\ref{dfWellIsolatedComponent}.  If no sufficiently small spatial
neighborhood of $\gamma_{I,\eta}$ isolated it, there would be a sequence of
other $1$-periodic orbits whose images lie in shrinking spatial neighborhoods
of $\operatorname{ev}(\gamma_{I,\eta})$.  Convergence of these periodic orbits
to the radial incarnation forces their radial levels to converge to
$r_\eta(T)$, hence their underlying Reeb periods to converge to~$T$.
Bounded-period Reeb ODE compactness then gives a limiting closed Reeb orbit of
period~$T$ with image in $\operatorname{ev}(\gamma)$.  By uniqueness of the
Reeb flow this limit belongs to~$\gamma$, and isolation of~$\gamma$ in the loop
topology forces the nearby orbits to project to Reeb orbits lying in~$\gamma$.
Their Reeb period is therefore~$T$, and strict monotonicity of~$h'$ on the
inner bend forces their radial level to equal $r_\eta(T)$.  Hence they belong
to~$\gamma_{I,\eta}$, a contradiction.  Thus a precompact spatial
well-isolating neighborhood exists.  Exactness and the canonical-bundle
trivialization come from the chosen local Reeb neighborhood.

For any $0<\eta_-<\eta_+$ below the smoothing threshold, the components
$\gamma_{I,\eta}$, $\eta\in[\eta_-,\eta_+]$, form a compact smooth isolated
Hamiltonian family.  Proposition~\ref{propLocalFloerIsotopyTransport}
therefore identifies the smooth local groups
at its endpoints.  Reversal and concatenation make these identifications
compatible as a positive interval is subdivided.  If the positive smoothing
path, the local regular perturbations, or the choice of family~$H_{I,\eta}$ is
changed, interpolate the choices through strictly convex inner profiles after
passing to a common positive smoothing range.  The square relation then
identifies the two transports.  Thus the resulting group and positive-scale
identifications are canonical.

Finally, the confined construction depends only on the Floer data on a
neighborhood of its evaluation image, so the unspecified exterior of
$H_{I,\eta}$ does not alter the group or the transport maps.
\end{proof}

\subsection{Positive-scale realization and transport}
\label{subsecIsotopy}
\label{SubsecProofThmLocalReebIntro}

We allow the contact hypersurface, contact primitive, and isolated Reeb
component to vary with the parameter.  The following notion records exactly
the geometric isolation needed to apply smooth local Floer transport at a
fixed positive smoothing scale.

Let $P$ be either $[0,1]$ or $[0,1]^2$,
fix a relatively open set $\widetilde W\subset M\times P$, and write
\[
\widetilde W^p:=\{x\in M\mid (x,p)\in\widetilde W\}.
\]
Consider a smooth family of properly embedded cooriented contact-type hypersurfaces
$\Sigma^p\subset\widetilde W^p$, equipped with smoothly varying contact-type
primitives $\alpha^p$ compatible with the chosen coorientation, so that
\[
d\alpha^p=\omega|_{\Sigma^p}.
\]
We require each $\Sigma^p$ to bound a properly embedded
codimension $0$ domain $D^p\subset\widetilde W^p$. That is,
\[
\partial D^p\cap\widetilde W^p=\Sigma^p.
\]

\begin{df}\label{dfReebComponentIsotopy}\label{dfIsolatingIsotopy}
An \emph{isolated $P$-family of Reeb components} consists of a compact
subset
\[
\Gamma\subset\cL M\times P
\]
and an open neighborhood
\[
\Gamma\subset \cU\subset\cL M\times P
\]
such that:
\begin{enumerate}
\item the projection $\Gamma\to P$ is surjective;
\item for every $p\in P$, the fiber $\Gamma^p$ is nonempty compact and path
connected;
\item for every $p\in P$,
\[
\Gamma^p
=
\{\ell\in\cU^p\mid
\ell\text{ is a closed Reeb orbit of }\Sigma^p\};
\]
\item each $\cU^p$ is an admissible neighborhood for $\Gamma^p$.
\end{enumerate}
No local triviality of $\Gamma\to P$ or $\cU\to P$ is assumed.
In particular, the topology and dimension of $\Gamma^p$ may vary with $p$.

For $P=[0,1]$ we call this an \emph{isolating isotopy}. A homotopy relative to
the endpoints between two isolating isotopies is an isolated
$[0,1]^2$-family whose restrictions to the two horizontal edges are the given
isotopies and whose restrictions to the two vertical edges, including the
contact hypersurfaces and local insides, are constant families.
\end{df}

The period is constant on each fiber $\Gamma^p$ and varies continuously with
the parameter, hence descends to a continuous function
\[
p\longmapsto T_{\alpha^p}(\Gamma^p)
\]
on~$P$.

A smooth family
\[
p\longmapsto I^p=(a^p,\alpha^p,\epsilon^p)
\]
of finite-slope parameters is \emph{$\Gamma$-admissible} if
\[
a^p>2T_{\alpha^p}(\Gamma^p)
\]
for every~$p\in P$.

\begin{prp}[Positive-scale realization]
\label{prpIsolatedFamily}
Let $P=[0,1]$ or $P=[0,1]^2$, let $(\Gamma,\cU)$ be a compact isolated Reeb
$P$-family, and let
\[
p\longmapsto I^p=(a^p,\alpha^p,\epsilon^p)
\]
be a smooth $\Gamma$-admissible finite-slope parameter family. Then:

\begin{enumerate}[label=\textup{(\roman*)},ref=\roman*]

\item
\textbf{(Realization.)}
There exists $\eta_0>0$ such that for every $0<\eta<\eta_0$ the inner
incarnations of $\Gamma^p$ in the fixed smoothings~$H_{I^p,\eta}$ form a
smooth isolated Hamiltonian $P$-family in the sense of
Definition~\ref{dfIsolatedHamiltonianFamily}.  The construction may be made
relative to any portion of the parameter boundary on which the Reeb and
finite-slope data are already fixed.

\item
\textbf{(Independence of positive-scale data.)}
For $P=[0,1]$, the transports associated with any two sufficiently small
positive-scale realizations agree under the canonical identifications of
their endpoint groups supplied by
Proposition~\ref{prpFiniteSlopeLocalPackage}.  In particular they determine
a canonical map
\[
\Phi_{\Gamma,I}:
HF^*_{uw}(\Gamma^0;I^0)
\longrightarrow
HF^*_{uw}(\Gamma^1;I^1).
\]

For $P=[0,1]^2$, the corresponding assertion holds on each boundary edge;
after these canonical endpoint identifications, the four edge transports
satisfy the square relation of
Proposition~\ref{propLocalFloerIsotopyTransport}.

\end{enumerate}
\end{prp}

\begin{proof}
\textup{(i)}.
For every basepoint $p_0\in P$, after shrinking a parameter box~$Q$, openness
of~$\cU$ and compactness of~$\Gamma$ allow the symplectization collar and
radial projection to be chosen so that sufficiently close projected loops
remain in~$\cU^p$.  Compactness of~$P$ gives a finite cover by such boxes.

Write
\[
T(p)=T_{\alpha^p}(\Gamma^p).
\]
Continuity and positivity of~$T(p)$, together with $\Gamma$-admissibility and
strict convexity of the fixed inner profiles of~$H_{I^p,\eta}$, give a
uniform $\eta_0>0$ so that the inner bend crosses $T(p)$ exactly once for
every~$p$.

Fix $0<\eta<\eta_0$.  On each parameter box, shrink to a small radial slab
around this crossing.  A projected Hamiltonian orbit is then a Reeb orbit
in~$\cU^p$, hence belongs to~$\Gamma^p$; its period is~$T(p)$, and strict
convexity forces the unique inner radial level.  Hence the only Hamiltonian
periodic orbits in this neighborhood are the inner incarnations of~$\Gamma^p$.

These local neighborhoods give an open worldtube for the positive-scale
incarnations.  Compactness of the total positive-scale incarnation follows
from compactness of~$\Gamma$ and continuous dependence of the radial
realization on~$p$ and~$\eta$.  Thus
\[
\Gamma_\eta^p
=
\Per(H_{I^p,\eta})\cap\cU_\eta^p
\]
is a nonempty compact well-isolated component for every~$p$.

The Hamiltonians $p\mapsto H_{I^p,\eta}$ vary smoothly by the
parameter-family smoothing convention fixed after
Definition~\ref{dfLocallySlopeAdapted}.  Exactness and the canonical-bundle
trivialization are
inherited from the local Reeb neighborhoods.  Hence the inner incarnations form
a smooth isolated Hamiltonian $P$-family in the sense of
Definition~\ref{dfIsolatedHamiltonianFamily}.

All choices above may be taken fixed on a parameter-boundary portion on which
the Reeb and finite-slope data are fixed, giving the relative statement.

\textup{(ii)}.
First let $P=[0,1]$.  Two sufficiently small positive scales, or two
choices of the fixed family of strictly convex inner smoothing profiles, may,
after passing to a common positive scale range and shrinking to common local
Reeb neighborhoods, be interpolated through the same class of strictly
convex inner profiles.  The argument in~\textup{(i)} is uniform over the
resulting square, so this interpolation is a smooth isolated Hamiltonian
$[0,1]^2$-family.

Let $v_0$ and $v_1$ denote the canonical positive-scale identifications at
the two endpoints supplied by
Proposition~\ref{prpFiniteSlopeLocalPackage}.  The square relation of
Proposition~\ref{propLocalFloerIsotopyTransport} gives
\[
v_1\circ\Phi_0
=
\Phi_1\circ v_0.
\]
Thus, after the canonical endpoint identifications, the two transports
coincide and define the map $\Phi_{\Gamma,I}$.

If $P=[0,1]^2$, compare two positive-scale realizations separately on each
of the four boundary edges.  Each comparison is again a two-parameter
isolated Hamiltonian family, so the preceding argument identifies the
corresponding edge transports.  Each realization itself satisfies the square
relation by Proposition~\ref{propLocalFloerIsotopyTransport}; hence the
resulting boundary square is independent of the positive-scale realization.
The same argument is relative to any boundary portion on which the data are
fixed.
\end{proof}

\begin{prp}[Transport of unweighted local Floer groups]
\label{prpTransportLocalReeb}
\label{prpIsotopyUW}
Fix an underlying oriented characteristic component~$\underline\gamma$.
Writing~$\gamma$ for its Reeb parametrization with respect to a contact
form~$\alpha$ (which changes with~$\alpha$), the groups
$HF^*_{uw}(\gamma;I)$ for all admissible finite-slope parameters~$I$ are
canonically identified.  The identifications agree with the confined
unweighted maps induced by any sufficiently small locally slope-adapted
continuation realizing the corresponding isolated Reeb path.

An isolating isotopy~$\Gamma$ induces a canonical isomorphism
\[
\iota_\Gamma:
HF^*_{uw}(\Gamma^0;I^0)
\xrightarrow{\;\sim\;}
HF^*_{uw}(\Gamma^1;I^1),
\]
independent of the admissible finite-slope parameter family and compatible
with the canonical endpoint identifications.  It depends only on the homotopy
class relative to the endpoints and satisfies concatenation, reversal, and
identity.
\end{prp}

\begin{proof}
Write $\underline\gamma$ for the fixed underlying oriented characteristic
component of~$\gamma$, oriented by $-R$ as in~\S\ref{SubsecSHReeb}.
Contact-type primitives compatible with this orientation form a convex set,
and
\[
T_\alpha(\underline\gamma)=-\int_{\underline\gamma}\alpha
\]
is affine in~$\alpha$.  Given two admissible finite-slope parameters, first
shrink their collar widths to a common sufficiently small positive value,
then interpolate the $(a,\alpha)$-data.  The inequality
$a>2T_\alpha(\underline\gamma)$ is preserved by convex interpolation.  The
collar width can then be varied separately.  The same construction over a
square joins two such admissible paths relative to their endpoints.

For a path of primitives $\alpha_\tau$ on a fixed hypersurface, parameterize
the same oriented closed characteristics by $-R_{\alpha_\tau}$.  This gives an
isolated Reeb family with constant underlying characteristic component.  Apply
Proposition~\ref{prpIsolatedFamily} and then
Proposition~\ref{propLocalFloerIsotopyTransport}.  The square construction just
described proves independence of the admissible parameter path.  Given a
sufficiently small locally slope-adapted continuation between the same
positive-scale endpoint incarnations, after shrinking to the common
strongly isolated presentation supplied by
Corollary~\ref{cyStrongLocalPresentations}, is regularly homotopic relative
to the endpoints to the continuation used to define the canonical
positive-scale transport; Proposition~\ref{propLocalContinuation}\textup{(i)}
identifies the induced maps.

For a geometric isolating isotopy~$\Gamma$, choose any smooth admissible
finite-slope parameter family and apply Proposition~\ref{prpIsolatedFamily}.
Two such choices are compared by the corresponding two-parameter family;
writing $v_0,v_1$ for the endpoint parameter-change maps, the square relation
gives
\[
v_1\circ\Phi_{\Gamma,I}
=
\Phi_{\Gamma,I'}\circ v_0.
\]
Hence the induced endpoint map is independent of~$I$; call it
$\iota_\Gamma$.  A geometric homotopy square gives homotopy invariance, and
concatenation, reversal, and identity follow from
Proposition~\ref{propLocalFloerIsotopyTransport}.
\end{proof}

\begin{df}\label{dfSHuwGamma}
The \emph{unweighted local Floer cohomology} of an isolated Reeb component
$\gamma$ is
\[
SH^*_{uw}(\gamma):=HF^*_{uw}(\gamma;I)
\]
for any admissible finite-slope parameter~$I$, with the different choices
identified by Proposition~\ref{prpTransportLocalReeb}.
\end{df}

\begin{proof}[Proof of Theorem~\ref{thmLocalReebIntro}]
Proposition~\ref{prpFiniteSlopeLocalPackage} defines the local group at a fixed
finite-slope parameter from any sufficiently small positive smoothing scale.
Proposition~\ref{prpTransportLocalReeb} canonically identifies the groups for
all admissible parameters and gives transport under isolating Reeb isotopies,
with homotopy invariance, concatenation, reversal, and identity.  Thus
Definition~\ref{dfSHuwGamma} gives the invariant claimed in
Theorem~\ref{thmLocalReebIntro}.  The construction is local because all smooth
local complexes and transports are confined to arbitrarily small exact
neighborhoods of the Reeb component.
\end{proof}

\subsection{The local barrier at fixed Reeb component}
\label{SubsecHeinReeb}

We now turn to the analytic input needed to compare the local
invariant as the smoothing parameter goes to $0$ and the truncated Floer cohomology of a finite slope continuous model.   The first step is the Reeb analogue of the $\epsilon$-barrier statements of
Section~\ref{SecInfFloer}. The difficulty is uniformity of the barrier as $\eta\to 0$. 

Fix a Reeb component $\gamma$ of $(\partial K,\alpha)$ and a fixed exact
collar neighborhood~$W$ of the evaluation image of~$\gamma$, given by a
Liouville flow box
\[
\Psi:U\times(-\delta,\delta)\longrightarrow W,
\qquad
\theta=e^\rho\alpha,
\]
for some open neighborhood $U\subset\partial K$ of $\operatorname{ev}(\gamma)$
and some $\delta>0$. Consider a monotone path of $\gamma$-admissible parameters
$I_s=(a_s,\alpha,\epsilon_s)$ with \emph{fixed} contact form~$\alpha$.
Write $\Sigma:=\partial K$.

Write $T:=T_\alpha(\gamma)$ and define
\[
g_\alpha(\gamma)
:=
\inf\left\{
|T_\alpha(x)-T|:\;
\begin{gathered}
x\text{ is a closed $-R$-orbit},\\
x(S^1)\subset\operatorname{ev}(\gamma),\\
[x]=[\gamma]\in H_1(U),\quad T_\alpha(x)\neq T
\end{gathered}
\right\}\in(0,\infty],
\]
with $\inf\varnothing=\infty$.  One has $g_\alpha(\gamma)>0$: on the compact
evaluation image there is a positive lower bound for nonconstant Reeb periods,
while any other parametrization of a characteristic supporting~$\gamma$ has
period differing from~$T$ by an integral multiple of its primitive period.
Any
\[
0<\Delta<\min\{T,g_\alpha(\gamma)\}
\]
is called a \emph{period-isolation radius} for~$\gamma$.

\begin{df}[Isolation pairs]\label{dfNestedConstruction}
For a period-isolation radius~$\Delta$, an
\emph{isolation pair for $(\Sigma,\alpha,\gamma)$ of radius~$\Delta$} consists of
precompact open sets
\[
V_-\Subset V_+\Subset U,
\qquad
\operatorname{ev}(\gamma)\subset V_-,
\]
satisfying
\begin{equation}\label{eqPeriodLocalIsolation}
\tag{$\mathrm{Iso}_\Delta$}
\begin{aligned}
&\text{every closed $-R$-orbit $x$ with }
x(S^1)\subset\overline V_+,\ 
[x]=[\gamma]\in H_1(U),\ 
|T_\alpha(x)-T|\le\Delta\\
&\text{belongs to }\gamma.
\end{aligned}
\end{equation}
In particular $x(S^1)\subset\operatorname{ev}(\gamma)\subset V_-$.
\end{df}

\begin{lm}[Existence of isolation pairs]
\label{lmIsolationPairExistence}
Let $\gamma$ be an isolated Reeb component.  For every period-isolation
radius~$\Delta$ and every prescribed spatial neighborhood of
$\operatorname{ev}(\gamma)$, there exists an isolation pair of
radius~$\Delta$ inside that neighborhood.
\end{lm}

\begin{proof}
Fix a period-isolation radius~$\Delta$ and suppose no such isolation pair
exists inside a prescribed spatial neighborhood of
$\operatorname{ev}(\gamma)$.  Then there are nested spatial neighborhoods
$N_n$ of $\operatorname{ev}(\gamma)$ with
$N_n\downarrow\operatorname{ev}(\gamma)$, and for each~$n$ one may choose
precompact opens
\[
V_{+,n}\Subset N_n
\]
together with closed Reeb orbits $x_n\notin\gamma$ satisfying
\[
x_n(S^1)\subset\overline V_{+,n}\subset N_n,
\qquad
[x_n]=[\gamma]\in H_1(U),
\qquad
|T_\alpha(x_n)-T|\le\Delta.
\]
Choose basepoints $x_n(0)$ and pass to a subsequence.  The periods
$T_\alpha(x_n)$ remain in the compact interval $[T-\Delta,T+\Delta]$, so the
Reeb ODE gives uniform $C^\infty$ bounds on compact time intervals.  By
Arzel\`a--Ascoli and the Reeb equation, $x_n$ converges in the loop topology
to a closed Reeb orbit $x_\infty$ with
\[
x_\infty(S^1)\subset\operatorname{ev}(\gamma),
\qquad
|T_\alpha(x_\infty)-T|\le\Delta,
\qquad
[x_\infty]=[\gamma].
\]
By the definition of $g_\alpha(\gamma)$ and the inequality
$\Delta<g_\alpha(\gamma)$, the limiting period must equal~$T$.  Uniqueness of
the Reeb ODE, together with the facts that the limiting period equals
$T_\alpha(\gamma)$ and that the limiting image lies in
$\operatorname{ev}(\gamma)$, identifies the limit as an orbit of~$\gamma$.
Isolation of~$\gamma$ in the loop topology then forces $x_n\in\gamma$ for
$n$ large, a contradiction.

Consequently such an isolation pair of radius~$\Delta$ exists inside every
prescribed spatial neighborhood of $\operatorname{ev}(\gamma)$.
\end{proof}

\begin{prp}[Local analytic barrier at fixed Reeb component]
\label{prpReebEpsilonBarriers}
Let $I_s=(a_s,\alpha,\epsilon_s)$ be the compact fixed-$\alpha$ path and let
$\cB$ be a bounded ball of compatible almost complex structures in the sense of
Definition~\ref{dfCanonicalFloerDiagram}. Choose a period-isolation
radius~$\Delta$ and an isolation pair
\[
V_-\Subset V_+\Subset U
\]
of that radius for~$\gamma$, as in Lemma~\ref{lmIsolationPairExistence}.
Write $T:=T_\alpha(\gamma)$. Then there exist
\[
r>0,\qquad 4r<\Delta,
\]
together with $r_+>0$, a $C^2$-open neighborhood~$\cN$ of the smoothing family
\[
\bigl(H_{I_s,\eta}\bigr)_{s\in\mathbb{R},\,0<\eta\le\eta_*}
\]
for some $\eta_*>0$, and $\epsilon_0>0$ such that, writing
\[
\mathcal P=(T-2r,T+2r)
\qquad\text{and}\qquad
W_+=\Psi\bigl(V_+\times(-r_+,r_+)\bigr),
\]
the pair
\[
\bigl(\cH(\cN,\cB),\ \cU_\gamma(W_+,\mathcal{P})\bigr)
\]
has an $\epsilon_0$-barrier in the sense of Definition~\ref{dfEpsilonBarrier}.
\end{prp}

The proof occupies the next two subsections. Because the isolation pair and
radial collar may be chosen inside a prescribed neighborhood of~$\gamma$, the
resulting admissible neighborhood may be made arbitrarily small when that is
desired.

\subsection{Loop coercivity in symplectization collars}
\label{SubsecSymplectizationCollars}

In the smooth setting, the proof of Lemma~\ref{lmEpsilonSeparation} obtains loop coercivity from $C^1$-control of $X_H$ in a $C^2$-neighborhood of an autonomous Hamiltonian. In the present setting the radial slope of a locally slope-adapted model may develop arbitrarily large derivatives under smoothing, so that argument is unavailable. The estimates below replace it by reparametrization to the Reeb flow and a finite-time isolation modulus.

We work in a collar neighborhood of an isolated Reeb component of a contact hypersurface $(\Sigma,\alpha)$, i.e.\ the image of a Liouville flow box $\Psi:U\times(-r,r)\to M$ as in the setup of Proposition~\ref{prpReebEpsilonBarriers} with $\Sigma=\partial K$. Write $H_1=e^\rho$, so that $X_{H_1}=-R$. For a Hamiltonian $H$, a smooth
path $x:[0,T]\to W$, and $p\ge1$, set
\[
E_H^p(x;T):=\int_0^T\|x'(t)-X_H(x(t))\|^pdt,
\]
omitting $T$ when the domain is clear.  Thus $E_H^2$ is the loop energy used
below, while $E_H^1$ is convenient for reparametrization and compactness.

\begin{lm}[Reparametrization lemma]\label{lmReparam}
Let $H=h(e^\rho)$ with $h'>0$, and let $x:[0,T]\to W$ be smooth.  Define
\[
\tau(t):=\int_0^t h'(e^{\rho(x(v))})\,dv.
\]
Then $\tau$ is strictly increasing and
\[
E_H^1(x;T)
=
E_{H_1}^1\bigl(x\circ\tau^{-1};\tau(T)\bigr).
\]
In particular, the image and homology class of the path are unchanged by this
reparametrization.
\end{lm}

\begin{proof}
Write $g=\tau^{-1}$.  Since
\[
g'(s)=h'(e^{\rho(x(g(s)))})^{-1}
\]
and
\[
X_H=(h'\circ e^\rho)X_{H_1},
\]
a change of variables gives
\begin{align*}
E_H^1(x;T)
&=\int_0^{\tau(T)}
\left\|
 x'(g(s))-h'(e^{\rho(x(g(s)))})X_{H_1}(x(g(s)))
\right\|g'(s)\,ds\\
&=\int_0^{\tau(T)}
\left\|
 (x\circ g)'(s)-X_{H_1}(x\circ g(s))
\right\|ds.
\end{align*}
\end{proof}

For the remainder of the discussion write $T:=T_\alpha(\gamma)$ and fix an
isolation pair $V_-\Subset V_+$ for~$\gamma$ of radius~$\Delta$.
Given a collar depth $0<r_*<r$, set
\[
r_-:=\tfrac13 r_*,
\qquad
r_+:=\tfrac23 r_*,
\qquad
r_0:=\tfrac16 r_*,
\]
and write the corresponding \emph{collar tubes}
\[
W_-:=\Psi\bigl(V_-\times(-r_-,r_-)\bigr),
\qquad
W_+:=\Psi\bigl(V_+\times(-r_+,r_+)\bigr).
\]
The pair $(V_-,V_+)$ is held fixed throughout; only the depth~$r_*$, and hence
the tubes~$W_\pm$, are adjusted below.

Define the \emph{period-isolation modulus}
\begin{equation}\label{eqShellDefect}
\kappa_{\mathrm{iso}}(\Delta)
:=
\inf
\left\{
\int_0^\Theta\|x'(t)+R(x(t))\|dt
\ \middle|\;
\begin{array}{l}
 x:[0,\Theta]\to\overline W_+,\\
 x\text{ is smooth and closed},\\
 {}[x]=[\gamma]\text{ in }H_1(\overline W_+),\\
 |\Theta-T|\le\Delta,\\
 \pi_\Sigma(x(0))\in\overline V_+\setminus V_-
\end{array}
\right\}.
\end{equation}
The basepoint condition is equivalent, after cyclic reparametrization, to
requiring that the image of~$x$ meet the \emph{active annulus}
$\Psi\bigl((\overline V_+\setminus V_-)\times[-r_+,r_+]\bigr)$.

\begin{lm}[Positivity of the isolation modulus]\label{lmShellBounds}
One has
\[
\kappa_{\mathrm{iso}}(\Delta)>0.
\]
Moreover, for a compact smooth family of Reeb vector fields, after shrinking
the parameter neighborhood and choosing a common period-isolation
radius~$\Delta>0$ so that the transported isolation pairs
satisfy~$(\mathrm{Iso}_\Delta)$ centered at~$T(s)$, one has
\[
\inf_s\kappa_{\mathrm{iso},s}(\Delta)>0
\]
upon imposing $|\Theta-T(s)|\le\Delta$.
\end{lm}

\begin{proof}
Suppose the infimum vanishes. Then there are closed paths
$x_n:[0,\Theta_n]\to\overline W_+$ with
\[
[x_n]=[\gamma]\text{ in }H_1(\overline W_+),
\qquad
|\Theta_n-T|\le\Delta,
\qquad
\pi_\Sigma(x_n(0))\in\overline V_+\setminus V_-,
\]
and
\[
\int_0^{\Theta_n}\|x_n'(t)+R(x_n(t))\|dt\longrightarrow0.
\]
After passing to a subsequence,
\[
\Theta_n\longrightarrow\Theta_\infty
\quad\text{with}\quad
|\Theta_\infty-T|\le\Delta,
\qquad
x_n(0)\longrightarrow x_\infty
\]
with $\pi_\Sigma(x_\infty)\in\overline V_+\setminus V_-$. The finite-time ODE
estimate of Proposition~\ref{PrpDistEn0} implies that $x_n$ converges uniformly
to the $-R$ trajectory starting at~$x_\infty$. For $n$ large enough the
projections of the paths $x_n$ under the collar deformation retraction to~$U$
are uniformly close to the projection of this limiting trajectory, hence
homotopic in~$U$; the limiting closed Reeb orbit therefore represents the
same homology class as~$\gamma$ in~$U$. Closedness
gives
\[
\phi_{\Theta_\infty}(x_\infty)=x_\infty.
\]
Since $d\rho(R)=0$, this exact trajectory has constant~$\rho$ and projects to a
closed $-R$-orbit contained in $\overline V_+$ which meets
$\overline V_+\setminus V_-$, has period $\Theta_\infty$ with
$|\Theta_\infty-T|\le\Delta$, and is homologous in~$U$ to~$\gamma$,
contradicting~\eqref{eqPeriodLocalIsolation}.

For the parametric assertion, the same compactness argument, using the
assumed isolation condition~$(\mathrm{Iso}_\Delta)$ for each~$s$ with
$|\Theta-T(s)|\le\Delta$, yields
\[
\inf_s\kappa_{\mathrm{iso},s}(\Delta)>0.
\]
\end{proof}

\begin{lm}[Loop-energy estimate]\label{lmLoopEnergyEstimate}
Fix an isolation pair of radius~$\Delta$ and choose
\[
0<4r<\Delta.
\]
Let $H=h(e^\rho)$ be a locally slope-adapted radial model of slope~$a$.
After making the collar depth and smoothing parameter sufficiently small,
there is
\[
\delta_{\mathrm{loop}}>0
\]
such that every smooth loop $\ell:S^1\to W_+$ satisfying
\[
\ell(S^1)\cap(W_+\setminus W_-)\neq\varnothing,
\qquad
[\ell]=[\gamma]\in H_1(W_+),
\qquad
|\cA_H(\ell)-T|\le r
\]
obeys
\[
E_H^2(\ell)\ge\delta_{\mathrm{loop}}.
\]

The conclusion is stable under perturbations for which
\[
\|H'-H\|_{C^0(W_+)}
+
\|X_{H'}-X_H\|_{C^0(W_+)}
\]
is sufficiently small.
\end{lm}

\begin{proof}
Since $4r<\Delta<T$ and $a>2T$, the enlarged interval satisfies
\[
[T-2r,T+2r]
\Subset
(T-\Delta,T+\Delta)
\qquad\text{and}\qquad
[T-2r,T+2r]
\Subset
(0,a).
\]
Write
\[
d:=\Delta-2r>0,
\qquad
m:=\min\{T-2r,\,a-(T+2r)\}>0
\]
for its concrete positive margins from the period-isolation boundary and from
$\{0,a\}$, and set
\[
C_\theta:=\|\theta\|_{C^0(W_+)},
\qquad
C_\rho:=\|d\rho\|_{C^0(W_+)}.
\]
For the radial model $H=h(e^\rho)$ on~$W_+$, set
\[
P_H(\rho)
:=
e^\rho h'(e^\rho)
=
-\theta(X_H),
\qquad
q_H(\rho)
:=
P_H(\rho)-H(\rho).
\]
An exact time~$1$ flowline at radius~$\rho$ has action $\cA_H=q_H(\rho)$.
From~\eqref{eqDefActionHF}, writing $\dot\ell=X_H+(\dot\ell-X_H)$ and using
$-\theta(X_H)=P_H$, the pairing $-\int_{S^1}\ell^*\theta$ differs from
$\int_0^1 P_H(\rho(\ell(t)))\,dt$ by at most $C_\theta E_H^1(\ell)$. Since
$q_H=P_H-H$,
\begin{equation}
\label{eqActionMomentum}
\Bigl|
\cA_H(\ell)-\int_0^1 q_H\bigl(\rho(\ell(t))\bigr)\,dt
\Bigr|
\le
C_\theta E_H^1(\ell).
\end{equation}

We first establish the model lower bound for loops satisfying the enlarged
action condition
\[
|\cA_H(\ell)-T|\le 2r.
\]
By decreasing the collar depth~$r_*$ and then~$\eta$, arrange that:
\begin{enumerate}
\item
\[
\beta_H:=
a\sup_{W_+}|e^\rho-1|
+
\|H\|_{C^0(W_+)}
<
\frac d3;
\]
\item
$r_+$ lies in the positive constant-slope region of
Definition~\ref{dfLocallySlopeAdapted}\textup{(iii)} and $\eta/8<r_0$, so
$h'\equiv a$ on $W_+\cap\{\rho\ge r_0\}$;
\item
on the negative radial shell, Definition~\ref{dfLocallySlopeAdapted}\textup{(i)}
gives $|H|<\eta/8$ and $h'\le\eta$ on $W_+\cap\{\rho\le0\}$, hence
$|q_H|<m/2$ after shrinking~$\eta$;
\item
on the positive constant-slope shell, $q_H$ is constant and convexity of~$h$
on $[1,e^{\eta/8}]$ with $h(1)<0$ implies $q_H>a$.
\end{enumerate}

Since $d\rho(X_H)=0$, every loop $\ell$ in~$W_+$ satisfies
\begin{equation}
\operatorname{osc}_\ell\rho
\le
C_\rho E_H^1(\ell).
\tag{R}
\end{equation}

\smallskip\noindent\textbf{Case 1: substantial radial oscillation.}
If $\operatorname{osc}_\ell\rho\ge r_0$, then~(R) gives
\[
E_H^1(\ell)\ge\frac{r_0}{C_\rho}.
\]

\smallskip\noindent\textbf{Case 2: small radial oscillation, radial-shell crossing.}
Assume $\operatorname{osc}_\ell\rho<r_0$, and choose
$p\in\ell(S^1)\cap(W_+\setminus W_-)$.
If $\pi_\Sigma(p)\in V_-$, then necessarily $|\rho(p)|\ge r_-$.

If $\rho(p)\le-r_-$, the oscillation bound forces
$\rho(\ell(t))\le-r_-+r_0=-r_0$ for every~$t$, so the whole loop lies in
the negative radial shell, where $|q_H|<m/2$. Since
$|\cA_H(\ell)-T|\le 2r$ and $[T-2r,T+2r]$ is at distance at least~$m$
from~$0$,~\eqref{eqActionMomentum} yields
\[
E_H^1(\ell)\ge\frac{m}{2C_\theta}.
\]

If $\rho(p)\ge r_-$, then $\rho(\ell(t))\ge r_--r_0=r_0$ for every~$t$, so
the whole loop lies in the positive constant-slope shell, where $q_H>a$.
Since $[T-2r,T+2r]$ lies a positive distance~$m$ below~$a$,
\eqref{eqActionMomentum} again yields a uniform positive lower bound; we take
the common safe estimate
\[
E_H^1(\ell)\ge\frac{m}{2C_\theta}.
\]

\smallskip\noindent\textbf{Case 3: spatial crossing.}
Otherwise
\[
\pi_\Sigma(p)\in V_+\setminus V_-,
\]
regardless of the value of~$|\rho(p)|$.
Cyclically reparametrize $\ell$ so that $p=\ell(0)$. Set
\[
\Theta:=\int_0^1 h'\bigl(e^{\rho(\ell(t))}\bigr)\,dt.
\]
Combining~\eqref{eqActionMomentum} with
\[
\Bigl|
\Theta
-
\int_0^1 q_H\bigl(\rho(\ell(t))\bigr)\,dt
\Bigr|
\le
a\sup_{W_+}|e^\rho-1|
+
\|H\|_{C^0(W_+)}
=\beta_H
\]
gives
\begin{equation}
|\Theta-\cA_H(\ell)|
\le
\beta_H+C_\theta E_H^1(\ell).
\tag{P}
\end{equation}
If $E_H^1(\ell)\ge d/(3C_\theta)$, we are done. Otherwise~(P) and
$\beta_H<d/3$ give $|\Theta-\cA_H(\ell)|<2d/3<d$. Together with
$|\cA_H(\ell)-T|\le 2r$ and $d=\Delta-2r$, this yields
\[
|\Theta-T|<\Delta.
\]

Thus either $E_H^1(\ell)$ is already bounded below, or the reparametrized
Reeb time satisfies $|\Theta-T|<\Delta$. Apply
Lemma~\ref{lmReparam} to reparametrize $\ell$ as
$\widetilde\ell:[0,\Theta]\to W_+$. Because the basepoint was chosen at~$p$,
\[
\pi_\Sigma(\widetilde\ell(0))=\pi_\Sigma(p)\in V_+\setminus V_-.
\]
The path has the same image and homology class and
\[
E_H^1(\ell)
=
\int_0^\Theta
\|\widetilde\ell'(t)+R(\widetilde\ell(t))\|\,dt.
\]
Since $|\Theta-T|<\Delta$,~\eqref{eqShellDefect} gives
$E_H^1(\ell)\ge\kappa_{\mathrm{iso}}(\Delta)$.

\smallskip
In every case one may take
\[
E_H^1(\ell)
\ge
c_1:=
\min\left\{
\frac{r_0}{C_\rho},
\frac{m}{2C_\theta},
\frac{d}{3C_\theta},
\kappa_{\mathrm{iso}}(\Delta)
\right\}>0.
\]
Cauchy--Schwarz then yields $E_H^2(\ell)\ge c_1^2$.

\smallskip\noindent\textbf{Perturbation.}
The preceding estimate applies to~$H$ on the enlarged range
$|\cA_H(\ell)-T|\le 2r$, giving an $E_H^1$-lower bound $c_+>0$.  Choose the
$C^0$-Hamiltonian perturbation small enough that
\[
\|H'-H\|_{C^0(W_+)}
<
r.
\]
Then
\[
|\cA_{H'}(\ell)-\cA_H(\ell)|<r,
\]
so
\[
|\cA_{H'}(\ell)-T|\le r
\quad\Longrightarrow\quad
|\cA_H(\ell)-T|\le 2r,
\]
and the model estimate applies.  Moreover
\[
E_{H'}^1(\ell)
\ge
E_H^1(\ell)
-
\|X_{H'}-X_H\|_{C^0(W_+)},
\]
so a sufficiently small $C^0$-bound on $X_{H'}-X_H$ yields
$E_{H'}^1(\ell)\ge c_+/2$.
\end{proof}

\begin{rem}\label{rmLoopEnergyUniform}
The estimate of Lemma~\ref{lmLoopEnergyEstimate} is uniform for any family for
which the following quantities have uniform positive lower bounds:
$\Delta$; $\Delta-2r$; the distance of $[T-2r,T+2r]$ from~$0$ and from the
slope~$a$; the isolation moduli $\kappa_{\mathrm{iso}}(\Delta)$; and the
existing radial-shell quantities $r_0/C_\rho$.  Moreover $C_\theta$ should
have a uniform upper bound, and the radial error
\[
a\sup|e^\rho-1|
+
\|H\|_{C^0}
\]
should be uniformly smaller than a fixed fraction of~$\Delta-2r$.  No
$C^2$-precompactness of the degenerating radial model family is required.
The same case division as in the proof yields a single
$\delta_{\mathrm{loop}}>0$ under these controls.
\end{rem}

\subsection{From loop coercivity to Reeb barriers}

With the loop energy estimate in place, the barrier argument is the same as in
Section~\ref{SubsecCZeroConfinement}. Let $W_-\Subset W_+$ be fixed
tubes and set
\[
d:=\inf_{J\in\cB}
d_{g_J}(\overline W_-,M\setminus W_+)>0.
\]
If a trajectory crosses from $W_-$ to the complement of $W_+$ and every
slice on the crossing interval has loop defect at least $\delta>0$,
Proposition~\ref{prpLoopwiseBarrier} gives an energy barrier
\[
\mathfrak b(\delta,d):=c(\delta)\min\{d,d^2\}>0.
\]
Here and below the dependence on the bounded geometry of $\cB$ and the
local $C^0$ bounds for $X_H$ and $\partial_sH$ is understood. Indeed,
take the interval from the last slice contained in the inner tube to the
first slice touching the boundary of the outer tube. Every intermediate
slice meets the shell, and at some time $t_0$ the endpoint values are
separated by at least $d$. The only additional task is to keep the
action in the range where the loop-energy estimate applies.

\begin{lm}[Finite-slope generator capture]
\label{lmReebGeneratorCapture}
Fix a $\gamma$-admissible parameter $I=(a,\alpha,\epsilon)$ and an
isolation pair $V_-\Subset V_+$ of radius~$\Delta$. Choose an open interval
$\mathcal P$ containing $T:=T_\alpha(\gamma)$ and a shallow collar
$W_+=\Psi(V_+\times(-r_+,r_+))$ such that
\[
e^{[-r_+,r_+]}\overline{\mathcal P}\subset(T-\Delta,T+\Delta),
\qquad
e^{-r_+}a>\sup\mathcal P.
\]
The outer tube is chosen with closure inside the isolating region of
radius~$\Delta$. Write $\cU=\cU_\gamma(W_+,\mathcal P)$.
For every $C^1$-neighborhood $\cO$ of $\gamma$ and every $e>0$, after
decreasing the smoothing threshold there is a $C^2$-open neighborhood
$\cN_{\cO,e}$ of the model smoothings such that
\[
\begin{gathered}
x\in\Per(H)\cap\cU,\qquad H\in\cN_{\cO,e}
\\
\Longrightarrow\qquad
x\in\cO,\qquad
|T_\theta(x)-T|<e,\qquad
|\cA_H^\theta(x)-T|<e.
\end{gathered}
\]
The assertion is uniform over compact smooth families with uniform
isolation of radius~$\Delta$ and uniform strict margins in the displayed
conditions.
\end{lm}

\begin{proof}
For a radial periodic orbit,
$T_{\mathrm{Reeb}}=e^{-\rho}T_\theta$, so the collar condition and
isolation of radius~$\Delta$ identify its underlying Reeb orbit with
$\gamma$.
The negative low-slope region has $T_\theta=O(\eta)$, while the positive
constant-slope region has $T_\theta\ge e^{-r_+}a>\sup\mathcal P$.
Hence every model orbit in $\cU$ lies on the inner bend. As $\eta\to0$,
its radius tends to zero, its Hamiltonian value tends to zero, and it
approaches $\gamma$ in $C^1$. Its period coordinate and action therefore
both tend to $T$.

At each positive smoothing scale, continuous dependence of the Hamiltonian
flow makes these conclusions stable under sufficiently small $C^2$
perturbations. Taking the union of these neighborhoods gives the required
open neighborhood of the smoothing tail; their radii may depend on the
smoothing scale. Bounded-period Reeb ODE compactness gives the same
conclusions uniformly over the compact parameter families in the statement.
\end{proof}

\begin{proof}[Proof of Proposition~\ref{prpReebEpsilonBarriers}]
Write $T=T_\alpha(\gamma)$ and $a_*:=\inf_s a_s$. Choose $r>0$ with
\[
4r<\Delta.
\]
The inclusions
$[T-2r,T+2r]\Subset(T-\Delta,T+\Delta)$ and
$[T-2r,T+2r]\Subset(0,a_*)$ follow from $\Delta<T$ and $a_*>2T$.
Set $\mathcal P=(T-2r,T+2r)$. Choose the collar sufficiently shallow
that Lemma~\ref{lmReebGeneratorCapture} applies to this interval and
Lemma~\ref{lmLoopEnergyEstimate} applies to loops satisfying
$|\cA_H-T|\le r$.
By these lemmas and Remark~\ref{rmLoopEnergyUniform}, after decreasing the
smoothing threshold and taking a sufficiently small $C^2$-open
neighborhood $\cN$ of the smoothing family, there are fixed tubes
$W_-\Subset W_+$ and $\delta>0$ such that
\[
\|H\|_{C^0(W_+)}<r/4,
\]
and, writing $\cU=\cU_\gamma(W_+,\mathcal P)$,
\[
x\in\Per(H)\cap\cU
\quad\Longrightarrow\quad
x(S^1)\subset W_-,
\qquad |\cA_H(x)-T|<r/4.
\]
Moreover, every loop in $W_+$ in the class of $\gamma$ which meets
$W_+\setminus W_-$ and satisfies $|\cA_H-T|\le r$ has loop defect
at least $\delta$.

Let $u$ have negative asymptotic $x_-\in\cU$ and suppose
$E_{\mathrm{top}}(u)<r/2$. Before its first spatial exit from $W_+$,
the partial energy identity and monotonicity give
\[
|\cA_{H_s}(u_s)-T|
\le |\cA_{H_-}(x_-)-T|+E_{\mathrm{top}}(u)
<3r/4.
\]
The homology class is unchanged on this initial segment, and
\[
T_\theta(u_s)=\cA_{H_s}(u_s)+\int_0^1H_s(u_s(t))\,dt
\]
implies $|T_\theta(u_s)-T|<r$. The spatial and period estimates keep $u_s$
in the open set whose distinguished component is~$\cU$. Since
$s\mapsto u(s,\cdot)$ is a path starting at the local asymptotic, it remains
in~$\cU$. Hence any exit must be spatial. The loop defect on the crossing
shell is at least~$\delta$, so the crossing estimate above gives
$E_{\mathrm{geo}}(u)\ge\mathfrak b(\delta,d)$.

The positive-asymptotic case uses the corresponding tail energy identity
and first entrance into the outer tube; likewise, the path starting at the
local asymptotic remains in~$\cU$, so any exit must be spatial. Since
$E_{\mathrm{top}}\ge E_{\mathrm{geo}}$, one may take
\[
\epsilon_0=\min\{r/2,\mathfrak b(\delta,d)\}.
\]
The choices are uniform over all data and auxiliary parameters in
$\cH(\cN,\cB)$. By Remark~\ref{rmLocalizedC0Packing}, the geometric
bounds need only hold on a neighborhood of the crossing shell inside the
controlled collar.
\end{proof}

\begin{cy}[Quantitative dependence of the local Reeb barrier]
\label{corUniformReebBarriers}
Suppose an action margin $m>0$ ensures that every trajectory with a local
asymptotic and $E_{\mathrm{top}}<m$ stays in the permitted action and
period ranges until its first spatial exit. If $\delta>0$ is the
loop-coercivity bound on the crossing shell and $d>0$ its spatial width,
then
\[
\epsilon_0=\min\{m,c(\delta)\min(d,d^2)\}
\]
is an energy barrier. It is uniform whenever $m,\delta,d$ have uniform
positive lower bounds and the background geometric and local $C^0$
bounds are uniform. Remark~\ref{rmLoopEnergyUniform} gives the geometric
conditions that supply a uniform $\delta$.
\end{cy}

\subsection{Barriers along isolating isotopies}
\label{SubsecIsotopyBarriers}

For a nearby Reeb model, the same crossing argument applies. We use one
primitive on the common exact neighborhood and measure action relative
to the local component. The only extra estimate controls the variation
of the normalized Hamiltonian.

\begin{df}[Reeb-adapted local Floer data]
\label{dfReebAdaptedLocalDatum}
A family of local Floer data is \emph{Reeb-adapted near $\Gamma$} if it
has smoothly varying symplectization collars on a common exact
neighborhood $W$ of the evaluation worldtube, a smooth $\Gamma$-admissible
parameter family $I_p=(a_p,\alpha_p,\epsilon_p)$, and each Hamiltonian
slice, after subtraction of a spatial constant, lies in a sufficiently
small $C^2$-neighborhood of a corresponding locally slope-adapted model
smoothing.

When the term is applied to a cofinal family of continuations
$H_{\sigma,\eta}$ along a fixed parameter path, the normalizations are
chosen once, and we also require uniform control of their variation:
for every $\nu>0$ there is a finite subdivision, independent of
sufficiently small $\eta$, such that on each piece
\[
\int\|\partial_\sigma
(H_{\sigma,\eta}-c_\eta(\sigma))\|_{C^0(W)}\,d\sigma<\nu.
\]
The normalizing functions $c_\eta$ are smooth and constant near the ends.
\end{df}

For a fiberwise $C^2$-open family $\cN=\{\cN_p\}_{p\in Q}$, call a
monotone datum $(H_\sigma,J_\sigma)$ over a path $p(\sigma)\in Q$
\emph{$(\cN,\nu)$-adapted} if $J_\sigma\in\cB$, the datum satisfies
the derivative bounds of Definition~\ref{dfCanonicalFloerDiagram}, and
there is a smooth spatial constant $c(\sigma)$, constant near the ends,
such that
\[
K_\sigma:=H_\sigma-c(\sigma)\in\cN_{p(\sigma)},
\qquad
\int_{\bR}\|\partial_\sigma K_\sigma\|_{C^0(W)}\,d\sigma<\nu.
\]
No monotonicity of $K_\sigma$ or bound on the constants $c(\sigma)$ is
required.

\begin{prp}[Local Reeb barrier for nearby models]
\label{prpBarrierShortIsotopy}
Let $(\Gamma,\cU)$ be an isolated family with Reeb-adapted data on a
common exact neighborhood $W$, and let $\cB$ be a bounded ball as in
Definition~\ref{dfCanonicalFloerDiagram}. Near every basepoint $p_0$,
after decreasing the smoothing threshold, there exist a parameter
neighborhood $Q$, an adapted isolating family
$\Gamma|_Q\subset\mathcal V\subset\cU|_Q$, a fiberwise $C^2$-open
neighborhood $\cN$ of the normalized model smoothings, and
$\nu_0,\epsilon_0>0$ such that every $(\cN,\nu_0)$-adapted monotone
datum has an $\epsilon_0$-barrier relative to $\mathcal V$.
These constants may be chosen uniformly over auxiliary families with
uniform action margins, loop-coercivity bounds, spatial separation, and
background geometric bounds.
\end{prp}

\begin{proof}
Fix a primitive $\theta$ on $W$ and write
$A(p):=T_\theta(\Gamma^p)$. This function is continuous. Choose a
period-isolation radius and an isolation pair of that radius at~$p_0$.
After shrinking~$Q$, and decreasing the radius if necessary, one common
$\Delta>0$ satisfies~$(\mathrm{Iso}_\Delta)$ for every nearby fiber, centered
at $T_{\alpha^p}(\Gamma^p)$; this follows from bounded-period Reeb ODE
compactness and isolation of the Reeb family.

Choose the moving collar tubes $W_-^p\Subset W_+^p$ with room between
their closures. After shrinking $Q$, choose fixed tubes satisfying
\[
\bigcup_{p\in Q}\overline{W_-^p}
\subset W_-\Subset W_+
\subset\bigcap_{p\in Q}W_+^p.
\]
By compactness of $\Gamma$ and openness of $W_+$, after shrinking $Q$
all components $\Gamma^p$ represent one homology class $g$ in $W_+$.
Write $\theta_p=e^{\rho_p}\alpha_p$ for the radial primitive. Since
$\theta-\theta_p$ is closed, on the class~$g$ one has
\[
T_\theta(\ell)-A(p)
=
T_{\theta_p}(\ell)-T_{\alpha_p}(\Gamma^p)
\]
for every loop~$\ell$ in class~$g$, and likewise
\[
\cA_K^\theta(\ell)-A(p)
=
\cA_K^{\theta_p}(\ell)-T_{\alpha_p}(\Gamma^p).
\]
Choose $r>0$ with $4r<\Delta$. By the period-coordinate identity and
$4r<\Delta$, the condition $|T_\theta(\ell)-A(p)|<2r$ places the
corresponding radial period coordinate inside the radius-$\Delta$ isolation
range. Define the open subset
\[
\mathcal S
:=
\Bigl\{
(\ell,p)\in\cL M\times Q:
\ell(S^1)\subset W_+,\;
[\ell]=g,\;
|T_\theta(\ell)-A(p)|<2r
\Bigr\}.
\]
After choosing $W_+$ and~$r$, shrink the connected parameter box~$Q$ so that
the path component $\mathcal V$ of $\mathcal S$ containing $\Gamma|_Q$
satisfies
\[
\Gamma|_Q\subset\mathcal V\subset\cU|_Q.
\]
This is arranged by the common-refinement property of admissible
neighborhoods, openness of~$\cU$, and compactness of~$\Gamma$.  Moreover
$\Gamma|_Q$ lies in a single path component of~$\mathcal S$: $Q$ is
connected and $\Gamma|_Q\to Q$ is a compact surjection with connected fibers,
hence $\Gamma|_Q$ is connected; since $\mathcal S$ is open in the locally
path-connected space $\cL M\times Q$, this connected set lies in a single
path component of~$\mathcal S$.  That component is open, so
$\mathcal V$ is an open path-connected moving trajectory basin.

Thus Lemmas~\ref{lmReebGeneratorCapture} and~\ref{lmLoopEnergyEstimate},
with Remark~\ref{rmLoopEnergyUniform}, give $\delta>0$ and a fiberwise
neighborhood $\cN$ such that local generators lie in $W_-$, their
centered action errors are less than $e$, and
$\|K\|_{C^0(W_+)}<r/4$. Every loop in $W_+$ in class $g$ meeting
$W_+\setminus W_-$ with
$|\cA_K^\theta(\ell)-A(p)|\le r$ has loop defect at least $\delta$.
Indeed, it lies in the model outer tube and meets its crossing shell.
The endpoint tolerance $e>0$ may be made arbitrarily small.

Shrink $Q$ and choose $e,\nu_0$ so that
\[
e+\operatorname{osc}_Q A+\nu_0<\frac r2.
\]
For an adapted datum put $K_\sigma=H_\sigma-c(\sigma)$. Normalization
does not change the Floer equation or geometric energy, and
\[
\frac{d}{d\sigma}\cA_{K_\sigma}^\theta(u_\sigma)
=-\int_0^1|\partial_\sigma u|^2\,dt
-\int_0^1\partial_\sigma K_\sigma(u_\sigma(t))\,dt.
\]
Integrating from a local asymptotic, before the first spatial exit, and using
$E_{\mathrm{geo}}\le E_{\mathrm{top}}$, gives
\[
|\cA_{K_\sigma}^\theta(u_\sigma)-A(p(\sigma))|
\le e+\operatorname{osc}_Q A+\nu_0+E_{\mathrm{geo}}(u)
<r
\]
whenever $E_{\mathrm{top}}(u)<r/2$. The same bound follows from a positive
local asymptotic by integrating over the terminal segment. Since
$\|K_\sigma\|_{C^0(W_+)}<r/4$,
\[
\begin{aligned}
\bigl|T_\theta(u_\sigma)-A(p(\sigma))\bigr|
&\le
\bigl|\cA_{K_\sigma}^\theta(u_\sigma)-A(p(\sigma))\bigr|
+
\|K_\sigma\|_{C^0(W_+)}
\\
&<
r+\frac r4
=
\frac54 r
<
2r.
\end{aligned}
\]
Thus the moving slice satisfies
\[
u_\sigma(S^1)\subset W_+,\qquad
[u_\sigma]=g,\qquad
|T_\theta(u_\sigma)-A(p(\sigma))|<2r,
\]
so remains in~$\mathcal S$; since it starts in~$\mathcal V$, continuity keeps
it in that path component.  Therefore any exit from~$\mathcal V$ must
cross the spatial shell, where the loop defect is at least~$\delta$. Writing
\[
d=\inf_{J\in\cB}d_{g_J}(\overline W_-,M\setminus W_+),
\]
one may take
\[
\epsilon_0
=
\min\left\{
\frac r2,\,
\mathfrak b(\delta,d)
\right\}.
\]
This also proves the asserted uniformity.
\end{proof}

For cofinal families, the variation condition in
Definition~\ref{dfReebAdaptedLocalDatum} permits a common subdivision into
the short pieces used above. A sufficient condition is an integrable
bound, independent of the smoothing scale,
\[
\|\partial_\sigma K_{\sigma,\eta}\|_{C^0(W)}\le g(\sigma),
\qquad g\in L^1(\bR).
\]
For smoothly parameterized model profiles, such a bound follows from
uniform $C^0$ bounds on their parameter derivatives and radial slopes,
together with control of the collar motion. It involves no bound on
the second radial derivatives of the smoothings.
\subsection{Interface with the global truncated theory}
\label{SubsecReebGlobalInterface}

The local invariant and its smooth transport were constructed without the
absolute Reeb barriers.  We now combine the two independent ingredients: the
barriers prevent low-energy trajectories from leaving the local block, while
generator capture identifies that block with the positive-scale Reeb invariant.

\begin{cy}[Truncated transport along a Reeb-adapted continuation]
\label{cyTruncatedTransportSubdivision}
Let $(\Gamma,\cU)$ be a compact isolating isotopy, and let $(H_s,J_s)$ be an
actual monotone Reeb-adapted continuation, in the sense of
Definition~\ref{dfReebAdaptedLocalDatum}, between locally slope-adapted
finite-slope endpoint systems associated with a finite-slope family~$I$.
Assume that a single primitive $\theta$ of~$\omega$ is fixed on a neighborhood
of the evaluation worldtube, and let $A_0,A_1$ be the limiting endpoint actions
computed with~$\theta$.

Then there is $\epsilon_\Gamma>0$ such that, for every
$0<\epsilon<\epsilon_\Gamma$, the induced truncated continuation map is block
diagonal with respect to the endpoint local/complement decompositions, and its
local block is
\[
\Phi^{\mathrm{loc}}_{[0,\epsilon)}
=
T^{A_0-A_1}
\bigl(
\iota_\Gamma\otimes\id_{\Lambda_{[0,\epsilon)}}
\bigr).
\]
The same window works after passing to sufficiently far cofinal smoothing
tails at the endpoints.
\end{cy}

\begin{proof}
Cover the compact isotopy by finitely many parameter neighborhoods supplied by
Proposition~\ref{prpBarrierShortIsotopy} and subdivide accordingly.  Let
$\epsilon_\Gamma>0$ be smaller than all of the corresponding absolute Reeb
barriers and the endpoint action-concentration margins.  On sufficiently far
cofinal smoothing tails, any trajectory which crosses from the chosen local
basin to its complement has topological energy at least the relevant barrier,
so it contributes zero in every window of width
$\epsilon<\epsilon_\Gamma$.  Thus the actual global continuation preserves the
local blocks on every short piece.

At a fixed positive smoothing stage, the confined part of each short
continuation is a smooth local continuation of the corresponding inner
incarnations.  By the regular-homotopy comparison in
Proposition~\ref{prpTransportLocalReeb}, its unweighted map is the canonical
smooth transport along that short isolated path.  Passing farther down the
cofinal tails does not change this map.  Lemma~\ref{lmReebGeneratorCapture}
gives endpoint action concentration at $A_0,A_1$, so
Corollary~\ref{propLocalContinuationTruncated} gives the short-piece weighted
block as source-minus-target action shift times the unweighted transport.  Composing the
short pieces telescopes both the unweighted transports and the action shifts,
giving the displayed formula.
\end{proof}

%% file: sections/04-truncated-unweighted.tex
\section{Global Floer theory of indicator Hamiltonians}\label{SecTruncSH}
In Section~\ref{secLocalReeb} we constructed the unweighted local Floer theory and
barriers for Hamiltonians which are locally slope-adapted near isolated Reeb components.
We now pass from this local theory to global finite-slope approximations of the
indicator Hamiltonian.

\subsection{Indicator approximations and their continuation maps}
\label{SubsecHKISetup}

Fix once and for all a contact-type ambient collar $V_*$ of
$\partial K$, with Liouville coordinate~$\rho$ and primitive
$\theta=e^\rho\alpha$ as in Section~\ref{secInfFH}. The auxiliary
profile width~$\epsilon$ stays strictly inside~$V_*$. On closed $M$ set
$G_\infty=0$. If $M$ is geometrically finite type, fix a nonnegative
proper Hamiltonian $G_\infty$ vanishing on a compact plateau containing
$K\cup V_*$ and all compact regions used below. Outside a slightly larger
compact set it agrees, up to a constant, with a sufficiently small multiple
of an admissible finite-type exhaustion
\cite[Lemma~3.1 and \S5]{GromanVarolgunes2021}. Choose the transition
$C^2$-small and without critical points outside the plateau, so that
$G_\infty$ has no nonconstant $1$-periodic orbits off the plateau,
$\|X_{G_\infty}\|_{C^0}$ is as small as required by
Proposition~\ref{prpRelEn}, and $(G_\infty,J_\infty)$ is dissipative
\cite{Groman}. Use the same fixed exterior datum for all slopes, widths,
regularizations, and continuations, and choose it commonly for finitely
many domains or a compact family.

\begin{df}[Finite-slope indicator Hamiltonian]\label{dfSShaped}
For each auxiliary width~$\epsilon$, fix $0<\zeta<\epsilon$ and a continuous
profile $f_{\epsilon,\zeta}\colon\mathbb R\to\mathbb R$ as follows. These
profile choices are fixed independently of the slope~$a$ and are suppressed from
the notation~$I$. Require $f_{\epsilon,\zeta}$ to be $0$ on
$(-\infty,0]$, to equal $e^t-1$ on $[0,\epsilon-\zeta]$, and to be constant
equal to some
\[
e_0\in(e^{\epsilon-\zeta}-1,e^\epsilon-1)
\]
on $[\epsilon,\infty)$, and require it to be nondecreasing and smooth on
$(0,\infty)$. Put
\[
P_{\epsilon,\zeta}(r):=f_{\epsilon,\zeta}(\log r).
\]
Require, in addition, that $P_{\epsilon,\zeta}$ be strictly concave on
$(e^{\epsilon-\zeta},e^\epsilon)$. Thus
$P_{\epsilon,\zeta}(r)=r-1$ on $[1,e^{\epsilon-\zeta}]$. For
$I=(a,\alpha,\epsilon)$ define
\[
H_{K,I}(x):=
\begin{cases}
0,&x\in K,\\
a f_{\epsilon,\zeta}(\rho(x)),&x\in V\setminus K,\\
a e_0+G_\infty(x),&x\in M\setminus(K\cup V),
\end{cases}
\]
where $V\subset V_*$ is the corresponding symplectization neighborhood.
Thus $H_{K,I}-G_\infty=ae_0$ on the genuine exterior.
\end{df}

For fixed $(\alpha,\epsilon)$ the Hamiltonians
$H_{K,(a,\alpha,\epsilon)}$ increase pointwise to $H_K$ as $a\to\infty$.
We use the partial order $I\le I'$ when $H_{K,I}\le H_{K,I'}$. When discussing
the Floer cohomology of a finite-slope Hamiltonian, we assume that $a$ lies
outside the Reeb period spectrum of~$\alpha$; structural parameter
continuations may cross spectral values. For each such~$I$, choose locally
slope-adapted model smoothings $H_{K,I,\eta}$ which increase to $H_{K,I}$ as
$\eta\downarrow0$. Outside a smoothing region shrinking with~$\eta$, they agree
with $H_{K,I}$ up to additive spatial constants tending uniformly to zero.

There are two continuation directions below: the approximation direction
$j\to\infty$ for a fixed~$I$, and the parameter direction $I\le I'$.

We now fix the chain-level presentations used throughout this section. For each
$I=(a,\alpha,\epsilon)$, choose a sequence $\eta_{I,j}\downarrow0$ and regular
Floer data
\[
F_{I,j}=(H_{I,j},J_{I,j})
\]
such that
\[
\|H_{I,j}-H_{K,I,\eta_{I,j}}\|_{C^2}\longrightarrow0,
\qquad
H_{I,j}\le H_{I,j+1}\le H_{K,I},
\qquad
H_{I,j}\longrightarrow H_{K,I}
\]
pointwise. We take all $J_{I,j}$ of contact type on~$V_*$ and in the fixed
bounded ball~$\cB$ of Definition~\ref{dfCanonicalFloerDiagram}, with
$J_{I,j}=J_\infty$ outside a fixed compact set when $M$ is noncompact.

We make one additional convention away from the collar. Set
\[
M_{\rm in}:=K\setminus V_*,
\qquad
M_{\rm out}:=M\setminus(K\cup V_*).
\]
For each~$j$, choose $C^2$-small autonomous Morse regularizations
\[
m_j^{\rm in},\qquad m_j^{\rm out},
\]
independent of~$I$, tending to zero with~$j$, and sufficiently small to have no
nonconstant one-periodic orbits. In the finite-type case the exterior
regularization is supported in a slightly larger compact neighborhood of
the zero plateau of $G_\infty$, with an outward-pointing gradient at the
plateau boundary and no periodic orbits in the transition or proper tail;
all exterior critical generators remain in the zero plateau.
The regularizations are chosen so that
\[
H_{I,j}|_{M_{\rm in}}-m_j^{\rm in},
\qquad
H_{I,j}|_{M_{\rm out}}-G_\infty-m_j^{\rm out}
\]
are spatially constant. These constants converge to $0$ on $M_{\rm in}$ and to
$ae_0$ on $M_{\rm out}$. Thus, away from the collar, changing~$I$ changes only
the spatial constant, while the normalized Morse datum is fixed.

Choose regular monotone continuation data
\[
k_{I,j}\colon F_{I,j}\longrightarrow F_{I,j+1},
\]
with continuation almost complex structures in~$\cB$.
We use
\begin{equation}\label{eqFiniteSlopeTelescope}
CF^*(H_{K,I})
:=
\widehat{\operatorname{tel}}
\Bigl(
CF^*(F_{I,1})\xrightarrow{\Phi_{k_{I,1}}}
CF^*(F_{I,2})\xrightarrow{\Phi_{k_{I,2}}}\cdots
\Bigr)
\end{equation}
as our fixed chain-level representative of the Floer complex of \(H_{K,I}\).
By Section~\ref{sec:floer_general_hamiltonian} it is canonically
chain-homotopy equivalent to any other cofinal presentation.  The notation
\(CF^*_{[0,\delta)}(H_{K,I})\) refers to the corresponding truncated
telescope.

We choose the presentations compatibly in the following sense. For every
fixed-$\alpha$ ordered pair $I\le I'$, after passing to cofinal tails and
reindexing,
\[
H_{I,j}\le H_{I',j}
\]
for every~$j$. Choose stagewise regular monotone continuations
\[
F_{I,j}\longrightarrow F_{I',j}
\]
with continuation almost complex structures in~$\cB$, sufficiently close to
compatible smoothings of
\[
H_s=(1-\beta(s))H_{K,I}+\beta(s)H_{K,I'}.
\]
On $M_\circ$, $\circ\in\{{\rm in},{\rm out}\}$, keep $m_j^\circ$ fixed;
equivalently, $H_{j,s}|_{M_\circ}-G_\infty-m_j^\circ$ is spatially constant
for every~$s$ (with $G_\infty=0$ on $M_{\rm in}$).
On a sufficiently shallow inner collar, choose the continuation Reeb-adapted in
the sense of Definition~\ref{dfReebAdaptedLocalDatum}.

Choose these stagewise maps as a continuation diagram in the sense of
Section~\ref{sec:floer_general_hamiltonian}. They induce a telescope map
\[
\Phi_{I,I'}:
CF^*(H_{K,I})
\longrightarrow
CF^*(H_{K,I'}),
\]
which we call the \emph{structural parameter continuation}. These continuation
diagrams are chosen coherently in the sense of
Section~\ref{sec:floer_general_hamiltonian}; in particular, the resulting maps
have the canonical composition property.

Elementary comparison near $r=1$ gives
\[
I=(a,\alpha,\epsilon)\le I'=(a',\alpha,\epsilon')
\quad\Longrightarrow\quad
a\le a',
\]
since both profiles equal $a(r-1)$ and $a'(r-1)$ arbitrarily close to
$r=1$. Consequently, if the source is $\gamma$-admissible, the entire
structural interpolation is $\gamma$-admissible.

The fixed-$\alpha$ parameter set is directed, and every fixed-width slope
ray is cofinal. Indeed, if $P$ and $P'$ are two admissible profiles, then
\[
P'\le C P
\]
for some $C>0$: the two profiles have the same positive linear germ $r-1$ near
$r=1$, and are positive and constant outside their respective collars. Hence a
sufficiently large slope on any fixed profile dominates any prescribed
finite-slope model; the slope may moreover be chosen outside the Reeb period
spectrum.

\medskip\noindent\textbf{Periodic orbits, actions, and filtrations.}
After passing to a cofinal tail, which does not change the completed telescope,
the generators at every stage consist of the critical/interior block
$\gamma_0$, an inner incarnation $\gamma^i$ and an outer incarnation
$\gamma^o$ of each Reeb component $\gamma$ with $T_\alpha(\gamma)<a$, and
the exterior/plateau generators. Inner incarnations together with $\gamma_0$
are called \emph{convex}; outer incarnations together with the
exterior/plateau generators are called \emph{concave}. Reeb incarnations use
the local action defined by~$\theta$, while constant exterior generators use
constant cappings. Along the chosen presentation,
\[
\cA_{H_{I,j}}(\gamma^i)=T_\alpha(\gamma)+o_j(1),
\]
and an outer incarnation satisfies
\begin{equation}\label{eqConcaveAction}
\cA_{H_{I,j}}(\gamma^o)
\le
 e^\epsilon T_\alpha(\gamma)
-a\bigl(e^{\epsilon-\zeta}-e^{\eta_{I,j}/8}\bigr)
+o_j(1).
\end{equation}
A constant exterior generator~$p$ satisfies
\[
\cA_{H_{I,j}}(p)=-ae_0+o_j(1).
\]
Here and below $o_j(1)\to0$ as $j\to\infty$ along the chosen presentation;
when a compact finite-slope parameter family is under consideration, these
errors are taken uniformly on that family.

Adding an $s$-dependent spatial constant does not change the Hamiltonian
vector field or Floer equation. In particular, in geometric estimates we may
subtract the exterior plateau scalar. The corresponding action shift is always
retained; the displayed exterior action above is the unnormalized one.

We use the displayed local actions as action labels on the generators.
Proposition~\ref{prphbarfilt} shows that, in the uniform $\hbar$-window, the
differentials and structural continuation data are nonincreasing for these
labels; the induced telescope filtration is the action filtration used below.

We also use without further comment the freedom to replace a presentation by
any cofinal regular presentation, in particular by a cofinal tail. We call a
subcomplex of a telescope \emph{cofinally defined} if it is represented by the
completed sub-telescope of stagewise subcomplexes on some cofinal
presentation. Any two such choices admit a common cofinal refinement.

\begin{prp}\label{prphbarfilt}
Fix the contact form~$\alpha$, the ambient collar~$V_*$, the primitive
$\theta=e^\rho\alpha$ from Subsection~\ref{SubsecHKISetup}, and the bounded
ball~$\cB$. There is an $\hbar>0$, depending only on the geometry of
$(M,\omega)$, the ambient collar~$V_*$, the exact-pair data, and~$\cB$,
with the following properties.

\begin{enumerate}
\item
For every $I=(a,\alpha,\epsilon)$, the truncated telescope
\[
CF^*_{[0,\hbar)}(H_{K,I})
\]
is graded by $H_1(\partial K;\bZ)$, is action-filtered, and has a cofinally
defined subcomplex spanned by the concave generators.

\item
For every fixed-$\alpha$ ordered pair $I\le I'$, the structural parameter
continuation
\[
\Phi_{I,I'}:
CF^*_{[0,\hbar)}(H_{K,I})
\longrightarrow
CF^*_{[0,\hbar)}(H_{K,I'})
\]
preserves the $H_1(\partial K;\bZ)$-grading and the action filtration. On
compatible cofinal regular presentations, it sends the stagewise concave
subcomplexes of the source into those of the target, and hence induces maps on
the corresponding cofinally defined concave subcomplexes and inner quotients.
\end{enumerate}
\end{prp}
\subsection{Relative homology and the uniform geometric threshold}
\label{SubsecRelativeHomology}

Choose precompact open sets
\[
U_0\Subset U_1\subset M
\]
such that $U_0$ has smooth boundary, $U_1$ deformation retracts onto
$U_0$, and
\[
\omega|_{U_1}=d\theta
\]
for a fixed primitive~$\theta$. In the applications below we take $U_0$ to be
the fixed ambient collar~$V_*$ of Subsection~\ref{SubsecHKISetup}
(or a slight thickening thereof) and $U_1$ an exact neighborhood of its
closure. All subsequent geometric constants are therefore allowed to depend on
this fixed ambient collar, but not on the auxiliary profile width~$\epsilon$.

For a compact oriented surface $S$ and a map
\[
v:(S,\partial S)\longrightarrow(M,U_1),
\]
write
\[
[v]\in H_2(M,U_1;\mathbb Z)
\]
for its relative homology class and define
\[
E^{rel}(v)
:=
\left\langle[\omega,\theta],[v]\right\rangle
=
\int_Sv^*\omega-\int_{\partial S}v^*\theta.
\]
In particular,
\[
[v]=0
\quad\Longrightarrow\quad
E^{rel}(v)=0.
\]

When a connected component is cut from
$u^{-1}(M\setminus U_0)$ along a Floer or monotone-continuation cylinder, the
resulting surface may have cylindrical ends. By construction, in both the
closed and finite-type cases every periodic orbit in $M\setminus U_0$ is a
compact constant/Morse orbit; in the finite-type case the transition and
proper tail have no periodic orbits. Hence every cylindrical end of a
complement component may be capped by a constant disc exactly as in the
closed case. The remaining boundary lies in $U_0\subset U_1$, so the capped
surface defines a class in $H_2(M,U_1;\mathbb Z)$.

\begin{lm}[Small slices imply relative triviality]
\label{lmSmallSlicesRelativeTrivial}
Let $S$ be the compactification of a connected component cut from a Floer or
monotone-continuation cylinder, and let
\[
v:(S,\partial S)\longrightarrow(M,U_1)
\]
be as above. The $s$-slices below refer to the original cut component, before
the constant caps are attached. There exists $\epsilon_*>0$, depending only on
$U_0\Subset U_1$ and the chosen Riemannian geometry, with the following
property.

If $[v]\neq0$, then some regular $s$-slice contains an arc $\gamma$ such that
\[
\operatorname{diam}v(\gamma)\ge\epsilon_*,
\qquad
d_M\bigl(v(\gamma),U_0\bigr)\ge\epsilon_*/2.
\]
\end{lm}
\begin{proof}
Choose $\epsilon_*>0$ so small that
\[
N_{2\epsilon_*}(U_0)\subset U_1
\]
and every ball of radius $\epsilon_*$ in the relevant bounded-geometry region
is strongly convex. Suppose, contrapositively, that no regular horizontal
slice contains an arc satisfying the two displayed inequalities.

Let $C$ be a connected component of a regular horizontal slice of the original
cut surface. If $v(C)$ meets $N_{\epsilon_*/2}(U_0)$, then
\[
v(C)\subset N_{3\epsilon_*/2}(U_0)\subset U_1.
\]
Indeed, otherwise a subarc between the last crossing of the
$\epsilon_*/2$-level and the first crossing of the $3\epsilon_*/2$-level would
stay at distance at least $\epsilon_*/2$ from $U_0$, while its endpoints are at
distance at least $\epsilon_*$ because the distance-to-$U_0$ function is
$1$-Lipschitz. This would be a forbidden arc. In particular every arc
component of a regular slice maps into $U_1$, since its endpoints lie on the
cutting boundary and hence map to $U_0$.

It remains to consider full-circle slice components disjoint from
$N_{\epsilon_*/2}(U_0)$. Each has image diameter $<\epsilon_*$, for otherwise
a diameter-realizing subarc would again be forbidden. On every band of such
full-circle slices, contract each loop geodesically to the value at a fixed
$t_0\in S^1$. Strong convexity makes these contractions continuous in $s$.
At a finite end of such a band the slice approaches the cutting boundary,
so nearby slices meet $N_{\epsilon_*/2}(U_0)$ and lie in $U_1$ by the preceding
paragraph; taper the contraction to the identity there. At a cylindrical end
the contraction extends over the constant cap. Thus this gives a homotopy of
pairs after which the part of $v$ outside $U_1$ is constant in the
$S^1$-direction. Modulo $U_1$ the map
therefore factors through the $s$-coordinate, hence through a one-dimensional
space. Its relative $2$-class is zero. Thus $[v]=0$, proving the
contrapositive.
\end{proof}

For a Floer or monotone-continuation Hamiltonian $H$ and a function $f=f(s)$ of
the domain coordinate alone, write
\[
\widehat H:=H-f.
\]

\begin{prp}[Relative-homology energy threshold]
\label{prpRelEn}
Fix the pair $U_0\Subset U_1$, the primitive $\theta$, a bounded ball $\cB$ of
compatible almost complex structures in the sense of
Definition~\ref{dfCanonicalFloerDiagram}, and a constant $q>0$. There exist
\[
\tau_{\mathrm{rel}}>0,
\qquad
\hbar_{\mathrm{rel}}>0,
\]
depending on this data, with the following property.

Let $S$ be the compactification of a connected component cut from
$u^{-1}(M\setminus U_0)$ along a Floer or monotone-continuation cylinder, and
let
\[
v:(S,\partial S)\longrightarrow(M,U_1)
\]
be the induced map. Let $W$ be the component of $M\setminus U_0$ containing
$v(\operatorname{int}S)$. Suppose that for some function $f=f(s)$,
\[
\widehat H:=H-f
\]
satisfies
\[
\partial_s\widehat H\ge0,
\qquad
\|X_{\widehat H}\|_{C^0(W)}
\le q\tau_{\mathrm{rel}},
\qquad
\|\partial_s\widehat H\|_{C^0(W)}
\le q\tau_{\mathrm{rel}}^2.
\]
If $[v]\neq0$ in $H_2(M,U_1;\mathbb Z)$, then
\[
E^{geo}(v)\ge\hbar_{\mathrm{rel}}.
\]
\end{prp}
Equivalently, $E^{geo}(v)<\hbar_{\mathrm{rel}}$ forces
$[v]=0$, and hence $E^{rel}(v)=0$.

\begin{proof}
Assume $[v]\neq0$ in $H_2(M,U_1;\mathbb Z)$. Let $\kappa>0$ be a common
bounded-geometry constant for the metrics $g_J$, $J\in\cB$, chosen so that the
hypotheses of Lemma~\ref{lmInteriorPathPacking} hold uniformly, and let
$\epsilon_*>0$ be the constant of
Lemma~\ref{lmSmallSlicesRelativeTrivial}, decreased so that
\[
\epsilon_*\le\min\{1,\kappa^{-1/2}\}.
\]

That lemma yields a regular slice arc $\gamma$ in the original Floer or
continuation component with
\[
d_M\bigl(v(\gamma(0)),v(\gamma(1))\bigr)\ge\epsilon_*,
\qquad
d_M\bigl(v(\gamma),v(\partial S)\bigr)\ge\frac{\epsilon_*}{2}.
\]
Passing to a diameter-realizing subarc, we may moreover assume
$\ell(\gamma)\le L_0$, where $L_0$ is the circumference of the normalized
circle.

Since $X_{\widehat H}=X_H$, the map $v$ satisfies the same Floer equation for
$\widehat H$. Apply Lemma~\ref{lmInteriorPathPacking} to the path~$\gamma$ with
\[
d=\epsilon_*,
\qquad
\delta=\frac{\epsilon_*}{2}.
\]
By Remark~\ref{rmLocalizedC0Packing}, the required $C^0$ bounds need only
hold on a $q$-dependent neighborhood of the image of~$\gamma$.

Choose $\tau>0$ sufficiently small, depending only on
$q,\epsilon_*,\kappa$ and $L_0$, so that
\[
\tau\le\min\Bigl\{\kappa^{-1/2},\,\frac{\epsilon_*}{2}\Bigr\},
\]
the $q$-dependent localization neighborhood corresponding to packing scale
$\tau$ stays in~$W$, and
\[
\tau\le\frac{\epsilon_*}{2C(q)L_0}.
\]
If
\[
\|X_{\widehat H}\|_{C^0(W)}\le q\tau,
\qquad
\|\partial_s\widehat H\|_{C^0(W)}\le q\tau^2,
\]
then the packing scale in Lemma~\ref{lmInteriorPathPacking} is~$\tau$, and
hence
\[
\epsilon_*\tau
\le
C(q)\left(
E^{\mathrm{geo}}(v;S)+\tau^2L_0
\right).
\]
Therefore
\[
E^{\mathrm{geo}}(v;S)
\ge
\frac{\epsilon_*\tau}{2C(q)},
\]
and the same lower bound holds for $E^{\mathrm{geo}}(v)$. Set
\[
\tau_{\mathrm{rel}}:=\tau,
\qquad
\hbar_{\mathrm{rel}}
:=
\frac{\epsilon_*\tau_{\mathrm{rel}}}{2C(q)}.
\]
\end{proof}
\subsection{The concave--convex energy barrier}
\label{SubsecConcaveConvexBarrier}

From now on apply Proposition~\ref{prpRelEn} with $q=1$, and write
$\tau_{\mathrm{rel}},\hbar_{\mathrm{rel}}$ for the corresponding constants.
Its hypotheses can be arranged componentwise for all indicator Floer data
used below. On each component $W$ of $M\setminus U_0$, subtract an
$s$-dependent spatial constant. For the model data the normalized Hamiltonian
is zero on the deep interior and, on the genuine exterior, is zero in the
closed case and the fixed~$G_\infty$ in the finite-type case. For regular
stages it also contains the $C^2$-small Morse perturbations
$m_j^{\rm in}$ or $m_j^{\rm out}$ from Subsection~\ref{SubsecHKISetup}.
Structural parameter continuations keep these normalized off-collar data
fixed. For a presentation structure map write on $W$
\[
H_s=c(s)+G_\infty+m_s,
\]
where $G_\infty=0$ on the interior component.
After stretching in~$s$, choose a smooth function $A(s)$ with
$A(s)\ge\|\partial_sm_s\|_{C^0(W)}$ and $A$ as small as required, and take
$f'(s)=c'(s)-A(s)$. Then
\[
\partial_s(H_s-f)=\partial_sm_s+A(s)\ge0
\]
and its $C^0$-norm is arbitrarily small. Thus, after passing to suitable
cofinal regular presentations,
\[
\partial_s\widehat H\ge0,
\qquad
\|X_{\widehat H}\|_{C^0(W)}\le\tau_{\mathrm{rel}},
\qquad
\|\partial_s\widehat H\|_{C^0(W)}
\le\tau_{\mathrm{rel}}^2
\]
on every component $W$ met by the trajectory.
Fix a contact-type neighborhood $U\subset V_*$ of $\partial K$ containing the
radial bands under consideration, and take the almost complex structures to be
of contact type there. For each model datum or structural continuation under
consideration, choose a separating radius $r_c>1$ in the common linear branch
of the relevant endpoint profiles and sufficiently far outside the inner
smoothing region, and set $\Sigma_c:=\{e^\rho=r_c\}\subset U$. Thus every
convex orbit lies on the inner side of $\Sigma_c$, while every outer Reeb
incarnation and every exterior/plateau generator lies on the outer side.

For a model Floer datum or model monotone continuation which is radial near
$\Sigma_c$, write $H_s=h_s(r)$ there, with the nonradial perturbation vanishing
near the cut, and put
\[
B_s:=h_s(r_c)-r_ch_s'(r_c).
\]
Write $M_c^+$ for the closed component of $M\setminus\Sigma_c$ pointing
outward from~$K$.

\begin{lm}[Outer no-escape criterion]
\label{lmIntegratedOuterNoEscape}
Let $u$ be a finite-energy model Floer or monotone-continuation solution such
that every compactified component of $u^{-1}(M\setminus U_0)$ has vanishing
relative energy and
\[
\partial_sH_s(x)\ge B_s'
\qquad (x\in M_c^+).
\]
Then:
\begin{enumerate}
\item if both asymptotic orbits lie on the inner side of $\Sigma_c$, then $u$
does not cross $\Sigma_c$;
\item if the negative asymptotic $p_-$ lies in the exterior/plateau block and
the positive asymptotic lies on the inner side, then
\[
H_-(p_-)\le B_-;
\]
\item if the negative asymptotic $x_-$ is a radial Reeb orbit in
$U\cap M_c^+$ and the positive asymptotic lies on the inner side, then
\[
\cA_{H_-}^{\theta}(x_-)+B_-\ge0.
\]
\end{enumerate}
\end{lm}

\begin{proof}
It suffices to prove the assertion when $r_c$ is a regular value of
$r\circ u$: regular values are dense, and the resulting inequalities persist
under passage to the limit. Let $S_{\mathrm{out}}$ be a connected component
of $u^{-1}(M_c^+)$ and set
\[
\Gamma:=\partial S_{\mathrm{out}}\cap u^{-1}(\Sigma_c),
\]
with its boundary orientation.  Vanishing of the relative energy on the
complement pieces, followed by Stokes through the intervening collar pieces,
reduces the energy identity on $S_{\mathrm{out}}$ to the usual exact-domain
one.  Near the cut, if $\tau$ is the positively oriented tangent to~$\Gamma$,
Floer's equation gives
\[
u^*\theta(\tau)+H_s\,dt(\tau)
=
r_c\,d\rho\bigl(du(-j\tau)\bigr)+B_s\,dt(\tau),
\]
and the first term on the right is nonpositive because $-j\tau$ is the
outward conformal normal to $S_{\mathrm{out}}$.

If $S_{\mathrm{out}}$ has no puncture, Stokes applied to $B_s$ gives
\[
E^{\mathrm{geo}}(u|_{S_{\mathrm{out}}})
\le
\int_{S_{\mathrm{out}}}
\bigl(B_s'-\partial_sH_s(u)\bigr)\,ds\,dt
\le0.
\]
If $S_{\mathrm{out}}$ met the interior of $M_c^+$, equality would force
$\partial_su=0$ there, which is impossible for a relatively compact component
with boundary on $\Sigma_c$.  This proves~(1).  If $S_{\mathrm{out}}$ contains
the negative puncture, apply the same calculation to truncations and let the
truncation length tend to infinity.  For a plateau orbit one obtains
\[
E^{\mathrm{geo}}(u|_{S_{\mathrm{out}}})
\le
B_--H_-(p_-)
+
\int_{S_{\mathrm{out}}}
\bigl(B_s'-\partial_sH_s(u)\bigr)\,ds\,dt,
\]
and for a radial Reeb orbit $x_-\subset U$ one obtains
\[
E^{\mathrm{geo}}(u|_{S_{\mathrm{out}}})
\le
B_-+\cA_{H_-}^{\theta}(x_-)
+
\int_{S_{\mathrm{out}}}
\bigl(B_s'-\partial_sH_s(u)\bigr)\,ds\,dt.
\]
Since geometric energy is nonnegative and
$B_s'-\partial_sH_s(u)\le0$, the first inequality gives
$H_-(p_-)\le B_-$ and the second gives
$\cA_{H_-}^{\theta}(x_-)+B_-\ge0$. This proves~(2) and~(3).
\end{proof}

The criterion is invariant under adding an $s$-dependent spatial constant to
$H_s$: the same derivative is added to $\partial_sH_s$ and $B_s'$, while
$\cA_{H_-}^{\theta}(x_-)+B_-$ is unchanged.  We therefore normalize such
constants away.

The same proof applies to a cut on an inward collar, $r_c<1$, provided the
Hamiltonians are radial near that cut and the complementary pieces have
vanishing relative energy. In particular, if
$h_s(r)=\mu(r-1)-c$ there, independently of $s$, then
$B_s=-\mu-c$ and monotonicity gives $\partial_sH_s\ge B_s'=0$ outside the cut.

Set
\[
\hbar:=\tfrac12\hbar_{\mathrm{rel}}.
\]
As in Subsection~\ref{SubsecHKISetup}, a concave orbit is either an outer
incarnation of a Reeb component or a generator of the exterior/plateau block.

\begin{lm}[Concave--convex energy threshold]
\label{lmBorgOanc}
The constant $\hbar$ defined above depends only on the fixed ambient collar,
the exact-pair data $U_0\Subset U_1$, and the bounded ball~$\cB$ of
Subsection~\ref{SubsecHKISetup}, and has the following properties.
\begin{enumerate}
\item For every $I=(a,\alpha,\epsilon)$ there is a cofinal regular
presentation such that every Floer trajectory in a stage and every stagewise
trajectory defining a presentation structure map from a concave input to a
convex output satisfies
\[
E^{geo}(u)\ge\hbar.
\]
\item For every fixed-$\alpha$ ordered pair $I_-\le I_+$, after choosing
compatible cofinal regular presentations, every stagewise trajectory defining
the structural continuation from a concave input to a convex output satisfies
the same inequality.
\item For any finite collection of such presentations and structural
continuations, the regular homotopies used to compare choices and to implement
continuation composition may be chosen so that every parametrized-continuation
solution from a concave input to a convex output satisfies the same inequality.
\end{enumerate}
In particular, $\hbar$ is independent of the slope and the auxiliary profile
width.
\end{lm}

\begin{proof}
Consider first a model structural parameter continuation for an ordered
endpoint pair
\[
I_-=(a_-,\alpha,\epsilon_-),
\qquad
I_+=(a_+,\alpha,\epsilon_+),
\]
realized by the monotone model interpolation
\[
H_s=(1-\beta(s))H_-+\beta(s)H_+.
\]
Choose the separating radius $r_c>1$ inside the common linear branch of the
two endpoint radial profiles. Thus near~$r_c$,
\[
H_\pm=a_\pm(r-1),
\]
and hence
\[
H_s=a_s(r-1),
\qquad
a_s=(1-\beta)a_-+\beta a_+,
\qquad
B_s=-a_s.
\]
Because $H_-\le H_+$, comparison arbitrarily close to $r=1$ implies
$a_-\le a_+$, so $B_s'=-a_s'\le0$. On the outer side,
\[
\partial_sH_s=\beta'(H_+-H_-)\ge0,
\]
and therefore $\partial_sH_s\ge B_s'$. For a fixed model Floer datum both
sides vanish. For a model presentation structure continuation, once the
separating cut lies outside the inner smoothing region the endpoint
Hamiltonians differ there only by spatial constants; after normalizing those
constants away one again has $\partial_sH_s=B_s'=0$ on $M_c^+$.

Choose the presentations far enough down that the separating cut lies outside
the inner smoothing region. Let $u$ be a model Floer trajectory, presentation
structure trajectory, or structural parameter continuation trajectory from a
concave input to a convex output, and suppose
\[
E^{\mathrm{geo}}(u)<\hbar_{\mathrm{rel}}.
\]
By Proposition~\ref{prpRelEn} and the normalization fixed above,
every compactified component of $u^{-1}(M\setminus U_0)$ has vanishing relative
energy. For an exterior critical/Morse source~$p_-$ one has
\[
H_-(p_-)-B_-=a_-(e_0^-+1)>0,
\]
contradicting Lemma~\ref{lmIntegratedOuterNoEscape}\textup{(2)}.
If the input is an outer Reeb orbit at radius~$r$ for the source
profile~$P_-=P_{\epsilon_-,\zeta}$, then
\[
T=a_-P_-'(r),
\qquad
\cA_{H_-}^{\theta}(x_-)=a_-\bigl(rP_-'(r)-P_-(r)\bigr),
\]
and hence
\[
\cA_{H_-}^{\theta}(x_-)+B_-
=
a_-\bigl(rP_-'(r)-P_-(r)-1\bigr)<0,
\]
since for $F(r):=rP_-'(r)-P_-(r)$ one has $F(e^{\epsilon_--\zeta})=1$ and
$F'(r)=rP_-''(r)\le0$ on the outer bend, so strict concavity gives
$F(r)<1$ for every $r>e^{\epsilon_--\zeta}$. This contradicts
Lemma~\ref{lmIntegratedOuterNoEscape}\textup{(3)}.
Thus every model concave-to-convex trajectory has energy at least
$\hbar_{\mathrm{rel}}$.

For each fixed positive smoothing scale and each compact model family of Floer
or continuation data above, standard Gromov--Floer compactness gives a
sufficiently small regular-perturbation neighborhood in which every
concave-to-convex solution has energy at least $\hbar$: otherwise a sequence
of regular perturbations, together with convergent auxiliary parameters, would
produce a broken model trajectory of total energy at most $\hbar$, and some
nonconstant level would contradict the model estimate. Choose the regular data
stagewise inside these neighborhoods. Doing this along a cofinal smoothing
sequence gives the presentation in~(1), and doing it simultaneously for a
fixed ordered pair gives the compatible continuation diagram in~(2).

The same argument applies to the compact one-parameter families of
continuation data defining regular homotopies. Indeed,
$\partial_sH_s-B_s'$ depends linearly on the Hamiltonian, so the model
no-escape inequality is preserved under convex interpolation of model
continuations; it is also preserved under the long-neck concatenations used
for composition. Relative regularity for compact parameter spaces therefore
allows the comparison and composition homotopies to be chosen in sufficiently
small regular neighborhoods of these model families, and the preceding
compactness argument gives the same $\hbar$-barrier for their parametrized
solutions. This proves~(3). Only the threshold~$\hbar$ is uniform; the
perturbation neighborhoods need not be.
\end{proof}
\subsection{The uniform $\hbar$-window}
\label{SubsecUniformHbar}
Let $\hbar$ be as in Lemma~\ref{lmBorgOanc}.

\begin{proof}[Proof of Proposition~\ref{prphbarfilt}]
Take $U_0$ and $U_1$ as in Subsection~\ref{SubsecRelativeHomology}, and work
on the cofinal regular presentations and compatible continuation diagrams
supplied by Lemma~\ref{lmBorgOanc}. This does not change the completed
telescopes or the structural maps. Throughout, the continuation cylinders
under consideration are Floer trajectories, presentation structure maps, and
fixed-$\alpha$ structural parameter continuations of
\S\ref{SubsecHKISetup}. The relative-homology argument is uniform in the
auxiliary width~$\epsilon$.

\smallskip\noindent\textbf{Grading and action filtration.}
Let $u$ be such a trajectory between $1$-periodic orbits $\gamma^\pm$
contributing in a Novikov window of width strictly less than $\hbar$. Then
\[
E_{\mathrm{geo}}(u)
\le
E_{\mathrm{top}}(u)
<\hbar
<\hbar_{\mathrm{rel}}.
\]
Hence every compactified complement component $v$ of
$u^{-1}(M\setminus U_0)$ satisfies
$[v]=0$ in $H_2(M,U_1;\mathbb Z)$ by Proposition~\ref{prpRelEn}.
Replace each such component by a relative chain in $U_1$. The entire
cylinder is thereby homologous relative to its ends to a chain contained in
$U_1$. Hence
\[
[\gamma^-]=[\gamma^+]
\quad\text{in }H_1(U_1;\mathbb Z)
\simeq H_1(\partial K;\mathbb Z).
\]
Summing the vanishing relative-energy identities over the complement
components yields
\[
E_{\mathrm{top}}(u)
=
\cA_{H^-}(\gamma^-)
-
\cA_{H^+}(\gamma^+).
\]
For a fixed Floer equation one has $E_{\mathrm{geo}}(u)=E_{\mathrm{top}}(u)$.
For a monotone continuation one uses only
$0\le E_{\mathrm{geo}}(u)\le E_{\mathrm{top}}(u)$. In either case
\[
\cA_{H^-}(\gamma^-)
\ge
\cA_{H^+}(\gamma^+).
\]
Hence the differential is action-decreasing and the structural parameter
continuations (as well as the presentation structure maps) decrease action
in the appropriate sense. The truncated complex is therefore
action-filtered and the filtration is preserved by these maps.

\smallskip\noindent\textbf{Concave subcomplex.}
A Floer trajectory, presentation structure map, or structural parameter
continuation contributing in the window satisfies
\[
E_{\mathrm{geo}}(u)
\le
E_{\mathrm{top}}(u)
<\hbar.
\]
On the cofinal regular presentations supplied by
Lemma~\ref{lmBorgOanc}, this rules out a concave input and a convex output,
including for arbitrary fixed-$\alpha$ structural parameter continuations.
Thus the concave generators define the asserted cofinal subcomplex, preserved
by the structural maps on compatible cofinal presentations. The parametric
clause of Lemma~\ref{lmBorgOanc} shows that the comparison homotopies and the
homotopies implementing composition preserve the same subcomplexes. Hence the
resulting maps on the cofinally defined concave subcomplexes and inner
quotients are independent of the chosen representatives and functorial up to
chain homotopy, and therefore strictly functorial on cohomology.
\end{proof}

\begin{rem}[Uniform indicator window]
\label{corUniformIndicatorWindow}
The same $\hbar$ works for any family of domains in which the exact-pair,
collar, and almost-complex-geometry bounds entering
Proposition~\ref{prpRelEn} and Lemma~\ref{lmBorgOanc} are uniform: the
constants depend only on those inputs, by inspection of the proofs.
\end{rem}

For $0<\delta\le\hbar$, let
$COut^*_{[0,\delta)}(H_{K,I})$ denote the cofinally defined completed
sub-telescope spanned by the concave generators, and let
\[
CIn^*_{[0,\delta)}(H_{K,I})
:=
CF^*_{[0,\delta)}(H_{K,I})/COut^*_{[0,\delta)}(H_{K,I})
\]
be the corresponding inner quotient. Write
\[
Out^*_{[0,\delta)}(H_{K,I})
:=H^*(COut^*_{[0,\delta)}(H_{K,I})),
\qquad
In^*_{[0,\delta)}(H_{K,I})
:=H^*(CIn^*_{[0,\delta)}(H_{K,I})).
\]
By Proposition~\ref{prphbarfilt}, these constructions are natural under all
fixed-$\alpha$ structural parameter continuations. The resulting short exact
sequence of completed telescopes induces the long exact sequence
\begin{equation}\label{eqInOutLES}
\dots\to Out^*_{[0,\delta)}(H_{K,I})\to HF^*_{[0,\delta)}(H_{K,I})\to In^*_{[0,\delta)}(H_{K,I})\to\dots,
\end{equation}
natural under all fixed-$\alpha$ structural parameter continuations
$I\le I'$.

\subsection{The finite-model package}
\label{SubsecFiniteModelPackage}
\label{SubsecInnerOuter}

Adjoin the critical component $\gamma_0$, with $T(\gamma_0)=0$ and $[\gamma_0]=0$ in $H_1(\partial K;\bZ)$, to the Reeb-component spectrum, and set
\[
SH^*_{uw}(\gamma_0):=H^*(K).
\]
For every component $\gamma$, define
\[
g(\gamma)
=
\inf\left\{
|T_\alpha(\gamma)-T_\alpha(\gamma')|
\ \middle|\
[\gamma']=[\gamma],\
T_\alpha(\gamma')\neq T_\alpha(\gamma)
\right\},
\]
where $\gamma'$ ranges over the augmented collection of Reeb components
together with $\gamma_0$, and $\inf\varnothing=+\infty$. Set
\[
\delta(\gamma)
:=
\min\bigl\{\hbar,\,\tfrac12 g(\gamma)\bigr\}.
\]
As in \S\ref{SubsecHeinReeb}, a finite-slope parameter
$I=(a,\alpha,\epsilon)$ is $\gamma$-admissible if $a>2T_\alpha(\gamma)$; in
particular every such $I$ is $\gamma_0$-admissible.

While the period spectrum depends on the choice of primitive inducing the
contact form $\alpha$, these gaps for a fixed homology class of $\partial K$ are
independent of this choice. By the grading and action-filtration
conclusions of Proposition~\ref{prphbarfilt}, in windows of size
$\leq\hbar$ the Floer cohomology splits into components associated with
distinct classes of $H_1(\partial K;\bZ)$.

\begin{prp}[Finite-model package]
\label{prpFiniteModelPackage}

\begin{enumerate}
\item\label{itFMPkgInnerSplit}
\label{prpInInfHKI}
\label{itFMPkgGrade}
\label{itFMPkgOuter}
For every component $\gamma$ such that $I$ is $\gamma$-admissible, and every
$0<\delta\le\delta(\gamma)$, the $\gamma$-generators define a cofinally
defined direct summand
\[
CIn^*_{[0,\delta)}(\gamma;H_{K,I})
\subset
CIn^*_{[0,\delta)}(H_{K,I}),
\]
and the resulting splitting is preserved by every fixed-$\alpha$ structural
parameter continuation $I\le I'$. Write
\[
In^*_{[0,\delta)}(\gamma;H_{K,I})
:=
H^*\bigl(CIn^*_{[0,\delta)}(\gamma;H_{K,I})\bigr).
\]
\item\label{itFMPkgIdentify}
\label{prpInInfHKIiso}
For each $\gamma$-admissible~$I$, there is a canonical identification
\begin{equation}\label{eqInTruncUwIso}
In^*_{[0,\delta(\gamma))}(\gamma;H_{K,I})
\cong
SH^*_{uw}(\gamma)\otimes\Lambda_{[0,\delta(\gamma))}.
\end{equation}
It is natural under the entire fixed-$\alpha$ structural parameter system of
\S\ref{SubsecHKISetup}.
\item\label{itFMPkgOuterVanish}
\label{prpOutInfHKI}
For every fixed $I=(a,\alpha,\epsilon)$ and every $0<\delta\le\hbar$, the
fixed-width slope continuation to
\[
I'=(b,\alpha,\epsilon)
\]
---a special case of a structural parameter continuation---is zero for
every sufficiently large slope increase $a\to b$:
\[
Out^*_{[0,\delta)}(H_{K,I})
\longrightarrow
Out^*_{[0,\delta)}(H_{K,I'}).
\]
\end{enumerate}
\end{prp}
In particular, outer vanishing holds up to the full geometric
threshold~$\hbar$.

The remainder of the subsection is devoted to the proof.

\begin{proof}[Proof of Proposition~\ref{prpFiniteModelPackage}\textup{(\ref{itFMPkgInnerSplit})}]
We first treat $\gamma=\gamma_0$.  Generators associated with the
critical block have actions converging uniformly to~$0$.  A
nonconstant inner generator in a different homology class is separated
from the critical block by the $H_1$-grading of
Proposition~\ref{prphbarfilt}. A
nonconstant inner generator in the zero homology class has action
converging to a period at least $g(\gamma_0)$. Pass to a sufficiently far
cofinal regular presentation so that the critical actions have absolute value
less than $g(\gamma_0)/8$ and all nonconstant zero-class inner actions differ
from their limiting periods by less than the same quantity.  By the action filtration of Proposition~\ref{prphbarfilt},
a trajectory from a
critical generator to a nonconstant inner generator would increase
action and therefore cannot occur.  A trajectory in the opposite
direction has action drop greater than
$\delta(\gamma_0)$, so its coefficient vanishes over
$\Lambda_{[0,\delta(\gamma_0))}$.  Hence the critical generators form
a direct summand in this truncated window. Representing a fixed-$\alpha$
structural parameter continuation on compatible cofinal regular presentations,
the same action-gap argument in both directions shows that it preserves the
critical summand and its complementary summand.

Now suppose that $\gamma$ is nonconstant and $I$ is $\gamma$-admissible.
We use the setup of \S\ref{SubsecHeinReeb}: choose an adapted admissible
neighborhood $\cU$ of $\gamma$ supplied by
Proposition~\ref{prpReebEpsilonBarriers} and a Floer diagram
$\cH=\cH(\cN,\cB)$ of Reeb-adapted data near a locally slope-adapted model, so
that by Definition~\ref{dfEpsilonBarrier} the pair $(\cH,\cU)$ has an
$\epsilon_0$-barrier. After passing to a cofinal regular presentation, take
the relevant stagewise Floer data in~$\cH$. Let $V\subset M$ be the underlying
open set of~$\cU$.

Consider a Floer trajectory $u$ contributing in the window $[0,\delta(\gamma))$ and connecting a generator associated with $\gamma$ to a generator associated with another component $\gamma'\neq\gamma$ (Reeb or critical). Then $E_{\mathrm{top}}(u)<\delta(\gamma)\le\hbar$. Since $E_{\mathrm{geo}}(u)=E_{\mathrm{top}}(u)<\hbar\le\hbar_{\mathrm{rel}}$, Proposition~\ref{prpRelEn} makes every exterior component relatively null-homologous, so the grading and action-filtration conclusions of Proposition~\ref{prphbarfilt} give $[\gamma]=[\gamma']$ in $H_1(\partial K;\bZ)$ and
\[
E^{geo}(u)=\cA(\gamma^-)-\cA(\gamma^+).
\]
Pass farther in this cofinal presentation, if necessary, so that every
relevant generator has individual action error strictly less than
\[
\min\bigl\{\tfrac18 g(\gamma),\,\tfrac14\epsilon_0\bigr\}.
\]

\smallskip\noindent\textbf{Case 1: $T(\gamma) \neq T(\gamma')$.} Since $[\gamma]=[\gamma']$ in $H_1(\partial K;\bZ)$ and $\gamma'$ ranges over the augmented collection \textup{(}Reeb components together with $\gamma_0$\textup{)}, the definition of $g(\gamma)$ gives $|T(\gamma)-T(\gamma')|\ge g(\gamma)$. The pairwise action error is then strictly less than $g(\gamma)/4$, so the action difference remains greater than $3g(\gamma)/4\ge\delta(\gamma)$. Thus such a trajectory cannot contribute in the window $[0,\delta(\gamma))$.

\smallskip\noindent\textbf{Case 2: $T(\gamma) = T(\gamma')$.} The two generators lie in distinct blocks, and their pairwise action difference is strictly less than $\epsilon_0/2<\epsilon_0$. A connecting trajectory therefore has energy less than $\epsilon_0$. The $\epsilon_0$-barrier relative to $\cU$ then prevents it from connecting the $\gamma$-block to its complement.

\smallskip
In both cases, no trajectory contributing in the chosen Novikov window connects the $\gamma$-block to its complementary block, so the splitting holds.

Now let $I\le I'$ and suppose $I$ is $\gamma$-admissible. Represent the
structural continuation on compatible cofinal regular presentations. By the
observation in Subsection~\ref{SubsecHKISetup}, the structural interpolation
is $\gamma$-admissible throughout and is Reeb-adapted near the inner
incarnation of~$\gamma$. For a continuation trajectory connecting the $\gamma$-block with another
inner block in either direction, if the limiting periods differ use exactly
the same homology/action-gap argument as for the differential; if the limiting
periods agree but the blocks are distinct, subdivide the Reeb-adapted
continuation if necessary and apply the moving local barrier/transport package
of \S\ref{secInfFH}. Thus no term in the stated window connects the
$\gamma$-block with its complementary inner block in either direction. Hence
the induced structural parameter map is block diagonal with respect to the
cofinally defined $\gamma$-summand and its complement.
\end{proof}
\begin{rem}\label{remFMPkgBarrierWindow}
The same $\gamma$-block splitting can be obtained from a local energy barrier
instead of from the period gap~$g(\gamma)$. If the finite-slope data for a
nonconstant~$\gamma$ have an $\epsilon_0(\gamma)$-barrier relative to an
isolating neighborhood, then any trajectory leaving the $\gamma$-block has
energy at least~$\epsilon_0(\gamma)$, so the splitting (and its preservation
under fixed-$\alpha$ structural parameter continuations) holds in every window
$0<\delta\le\epsilon_0(\gamma)$. Throughout the paper either mechanism may be
used to produce the truncated windows.
\end{rem}

\begin{proof}[Proof of Proposition~\ref{prpFiniteModelPackage}\textup{(\ref{itFMPkgIdentify})}]
For a nonconstant component~$\gamma$, work with the $\gamma$-summand of the
inner quotient supplied by Part~\ref{itFMPkgInnerSplit}, whose preservation
and local confinement under structural parameter continuations were just
established.  The coefficients
between two $\gamma$-generators in that quotient are the corresponding full
Floer coefficients with outer-output terms discarded.

Any trajectory contributing nontrivially in the window $[0,\delta(\gamma))$
has
\[
E_{\mathrm{top}}<\delta(\gamma)\le\hbar.
\]
Proposition~\ref{prpRelEn} therefore forces the relevant exterior relative
class to vanish. Consequently, for a trajectory whose two endpoints belong
to the same $\gamma$-block,
\[
E_{\mathrm{top}}
=
\cA_{\mathrm{source}}
-
\cA_{\mathrm{target}}.
\]
On the cofinal regular presentation chosen above,
Lemma~\ref{lmReebGeneratorCapture} gives uniform convergence of both endpoint
actions to $T(\gamma)$, so this energy tends uniformly to zero. Pass farther
if necessary so that the energy lies below the fixed local barrier of
Proposition~\ref{prpReebEpsilonBarriers}. Such trajectories are then confined
to the local basin.

Hence the $\gamma$-summand of the inner quotient is computed by the confined
local directed system. Lemma~\ref{lmAlgTruncUw} identifies its truncated
colimit with
\[
HF^*_{uw}(\gamma;I)\otimes\Lambda_{[0,\delta(\gamma))},
\]
and Definition~\ref{dfSHuwGamma} together with
Proposition~\ref{prpTransportLocalReeb} identify this with
$SH^*_{uw}(\gamma)\otimes\Lambda_{[0,\delta(\gamma))}$.
This is~\eqref{eqInTruncUwIso}.

Now let $I\le I'$ be a fixed-$\alpha$ ordered pair, with arbitrary widths,
and represent the structural continuation on compatible cofinal regular
presentations. By the Reeb-adapted choice of structural parameter
continuations in Subsection~\ref{SubsecHKISetup} and the
preservation/confinement just established in Part~\ref{itFMPkgInnerSplit},
the continuation acts on the confined $\gamma$-summand. On a sufficiently shallow inner collar both
endpoint profiles are in their common linear branch, so the corresponding
inner Reeb incarnation of~$\gamma$ has limiting action center
$T_\alpha(\gamma)$ at both endpoints, independently of~$a$ and~$\epsilon$.
Apply Corollary~\ref{propLocalContinuationTruncated} and
Proposition~\ref{prpTransportLocalReeb}. The Novikov factor is
\[
T^{T_\alpha(\gamma)-T_\alpha(\gamma)}=1,
\]
and the unweighted transport is the canonical parameter identification.
Hence under the canonical $SH^*_{uw}(\gamma)$-identifications the structural
parameter map is the identity.

For the critical block, fix an inward collar cut $r_c<1$ and let $L\subset K$
be the domain it bounds; thus $K$ deformation retracts onto $L$.
Choose a Morse function $f$ on $L$, equal to $r-1$ near the cut, with outward
gradient there, and use $\mu_j f-c_j$ on $L$, where $\mu_j,c_j\downarrow0$.
Choose these data independently of slope and width. Extend them radially
across the cut to the finite-slope profile, without additional critical
points. For trajectories with two critical ends, the relative-energy
threshold and the inward-cut version of
Lemma~\ref{lmIntegratedOuterNoEscape} give confinement to $L$; the same holds
for the structural continuations and their comparison homotopies.
On this fixed compact core, the sufficiently small autonomous
Floer--Morse comparison identifies the differential with the Morse
cochain differential of $f$; see \cite[\S7]{HoferSalamon} and
\cite[\S3.5]{Salamon1999}. The outward gradient gives ordinary
$H^*(L)=H^*(K)$. This comparison is applied after confinement, so no
global smallness of the finite-slope Hamiltonian is asserted.
The critical actions converge uniformly to~$0$.
Lemma~\ref{lmAlgTruncUw} therefore gives
\[
In^*_{[0,\delta)}(\gamma_0;H_{K,I})
\cong
H^*(K)\otimes\Lambda_{[0,\delta)}
\]
for $0<\delta\le\delta(\gamma_0)$. Choose structural parameter continuations
constant on this Morse core. Their confined unweighted maps are the identity;
the zero limiting action displacement gives no Novikov factor. Under
$H^*_{\mathrm{crit}}\cong H^*(K)$, every fixed-$\alpha$ structural
parameter map is therefore
\[
\id_{H^*(K)}\otimes\id_{\Lambda_{[0,\delta)}}.
\]
\end{proof}
\begin{proof}[Proof of Proposition~\ref{prpFiniteModelPackage}\textup{(\ref{itFMPkgOuterVanish})}]
As a special case of a structural parameter continuation, it suffices to
consider fixed-width slope continuations for $s\in[a,b]$. Choose the
stagewise data and continuation diagrams on compatible cofinal regular
presentations, uniformly over this compact slope interval; the resulting
$o_j(1)$ errors are then uniform in~$s$.

Let $A_s(T)$ be the outer action center for a Reeb period $T>0$. If $r_s(T)$ is
its radial coordinate, then
\[
A_s(T)=r_s(T)T-sP(r_s(T)),
\qquad
sP'(r_s(T))=T,
\]
so
\[
\partial_TA_s(T)=r_s(T),
\qquad
-\partial_sA_s(T)=P(r_s(T)).
\]
Writing
\[
r_-:=e^{\epsilon-\zeta},
\qquad
m:=P(r_-)=e^{\epsilon-\zeta}-1>0,
\qquad
A_s(0):=-se_0,
\]
one has
\begin{align*}
A_s(T')-A_s(T)&\ge r_-(T'-T)\quad(T'>T>0),\\
m(s_1-s_0)&\le A_{s_0}(T)-A_{s_1}(T)\le e_0(s_1-s_0).
\end{align*}
Exterior generators have centers
\[
\cA_s(p)=-se_0+o_j(1).
\]

Fix $0<\delta\le\hbar$. Only finitely many outer Reeb periods occur at the
source slope~$a$, all strictly smaller than~$a$, so there is $d_I>0$ such that
whenever $T$ is one of these source periods (including $T=0$) and $T'>T$ is
the period of a component in the same $H_1(\partial K;\bZ)$-class, then
$T'-T\ge d_I$. Subdivide $[a,b]$ so finely, and pass far enough down the chosen
cofinal presentation, that on every short step
\[
e_0\Delta s+o_j(1)<\tfrac12 r_-d_I.
\]
Action monotonicity then forces the target period not to exceed the source
period. This includes a zero-period exterior source, because for $T'>0$
\[
A_s(T')-A_s(0)\ge r_-T'.
\]

Consequently every composite outer term has action drop at least
\[
\min\{m,e_0\}(b-a)-o_j(1).
\]
Positive-to-positive steps contribute at least
\[
m\Delta s-o_j(1),
\]
a positive-to-zero step contributes at least
\[
e_0\Delta s-o_j(1),
\]
and every subsequent zero-to-zero step contributes
\[
e_0\Delta s-o_j(1).
\]
Choose $b$ sufficiently large and then pass sufficiently far down the
chosen cofinal regular presentation. The composite outer continuation vanishes over
$\Lambda_{[0,\delta)}$.
\end{proof}

%% file: sections/05-main-proofs.tex
\subsection{Proof of Theorems \ref{mainThmA} and \ref{thmIntroIsotopyRestriction}}
\label{SecProofs}
\label{SubsecProofThmA}
\label{SubsecIndicatorAssembly}
\label{SubsecProofIsotopy}

\begin{proof}[Proof of Theorem~\ref{mainThmA}]
By Novikov translation it suffices to treat a window starting at~$0$. Fix once
and for all an auxiliary profile width~$\epsilon$ inside the ambient collar of
Subsection~\ref{SubsecHKISetup}, together with the contact form~$\alpha$ of
that subsection, and an increasing cofinal sequence $a_j\to\infty$ with each $a_j$ outside
the Reeb period spectrum of~$\alpha$. The corresponding finite-slope
Hamiltonians satisfy
\[
H_{K,(a_j,\alpha,\epsilon)}\nearrow H_K
\]
pointwise. A regular diagonal of their presentations is an acceleration
datum, so the comparison in Section~\ref{SecConventions} shows that this ray computes
\[
SH^*_{M,[0,\delta)}(K).
\]
All colimits in this proof are taken over this fixed-$(\alpha,\epsilon)$ slope
ray. Fix an arbitrary $0<\delta\le\hbar$, with $\hbar$ as in
Proposition~\ref{prphbarfilt}. The mapping telescope over the slope ray, using
the finite-slope telescope representatives of Subsection~\ref{SubsecHKISetup},
is a chain model for $SC^*_{M,[0,\delta)}(K)$. Proposition~\ref{prphbarfilt}
makes this chain model $H_1(\partial K;\bZ)$-graded and action-filtered, since
the structural maps preserve both structures. On cohomology, taking the
filtered colimit gives
\[
SH^*_{M,[0,\delta)}(K)
=
\varinjlim_j HF^*_{[0,\delta)}\bigl(H_{K,(a_j,\alpha,\epsilon)}\bigr).
\]
For every source slope~$a_j$,
Proposition~\ref{prpFiniteModelPackage}\textup{(\ref{itFMPkgOuterVanish})}
eventually kills $Out^*_{[0,\delta)}(H_{K,(a_j,\alpha,\epsilon)})$ after a
sufficiently large slope increase, so
$\varinjlim_j Out^*_{[0,\delta)}(H_{K,(a_j,\alpha,\epsilon)})=0$. With the
long exact sequence~\eqref{eqInOutLES} one therefore obtains
\[
SH^*_{M,[0,\delta)}(K)
\cong
\varinjlim_j In^*_{[0,\delta)}\bigl(H_{K,(a_j,\alpha,\epsilon)}\bigr).
\]
Under this identification, action concentration of the inner Reeb blocks
induces the exhaustive Reeb-period filtration on truncated symplectic
cohomology. This proves the first assertion of
Theorem~\ref{mainThmA} for every window of size at most~$\hbar$.

For fixed~$\alpha$, all finite-slope parameters form the directed structural
system of Subsection~\ref{SubsecHKISetup}, and every fixed-width slope ray
is cofinal. Proposition~\ref{prpFiniteModelPackage} identifies every local
inner block naturally under all structural parameter continuations.
Therefore the resulting Reeb-period summands, the critical summand, and
their identifications with
\[
SH^*_{uw}(\gamma)\otimes\Lambda_{[0,\delta)}
\]
are independent of the chosen cofinal ray and hence of the auxiliary width.

Now fix a period cutoff $a>0$. Choose
\[
\delta(a)
=
\min\bigl\{\delta(\gamma_0),\,\delta(\gamma):T(\gamma)\le a\bigr\},
\]
or alternatively the barrier window
$\delta(a)=\min\bigl\{\delta(\gamma_0),\,\epsilon_{\mathrm{loc}}\bigr\}$,
where $\delta(\gamma_0)$ is the critical separation of
Subsection~\ref{SubsecFiniteModelPackage} and $\epsilon_{\mathrm{loc}}$ is a
common $\epsilon_0$-barrier from Proposition~\ref{prpReebEpsilonBarriers} for
all Reeb components of period at most~$a$ (with $\gamma$-admissible finite-slope
data); in this alternative the $\gamma$-block splittings are supplied by
Remark~\ref{remFMPkgBarrierWindow} rather than by the period-gap window of
Proposition~\ref{prpFiniteModelPackage}\textup{(\ref{itFMPkgInnerSplit})}.
At this smaller window, the block-diagonal structural maps of
Proposition~\ref{prpFiniteModelPackage}\textup{(\ref{itFMPkgInnerSplit})}
\textup{(}or of Remark~\ref{remFMPkgBarrierWindow}, in the barrier
alternative\textup{)} make the sum of the critical component $\gamma_0$ and
all Reeb components of period $\le a$, together with its complement, a
decomposition of directed systems. Passing to the colimit yields
\[
SH^{\le a}_{M,[0,\delta(a))}(K)
\oplus
SH^{>a}_{M,[0,\delta(a))}(K).
\]
Proposition~\ref{prpFiniteModelPackage}\textup{(\ref{itFMPkgIdentify})},
together with the structural parameter naturality of that proposition,
identifies the first summand with
$RH^{\le a}(K;\theta)\otimes\Lambda_{[0,\delta(a))}$, which
is~\eqref{eqSHDeltaA}.

For noncompact $M$, changing the exterior profile does not change these
local identifications in a sufficiently small cutoff window. To see this,
choose comparison acceleration data which agree near $K$ and the finitely
many local basins under consideration. A common lower proper exterior
profile gives comparison maps to both choices. On the Reeb basins the local
barriers of Proposition~\ref{prpReebEpsilonBarriers} apply, since their
proof uses only the data near the basin and the total geometric energy.
For the critical basin, use a fixed subperiod source slope and an isolating
annulus in its linear collar; Proposition~\ref{prpLoopwiseBarrier} supplies
the analogous local barrier. After decreasing the window below their finite
minimum, the comparisons are confined and constant on each basin, hence
induce its identity identification. The natural comparison of acceleration
data then identifies these blocks in the usual truncated invariant. This
argument does not apply a global relative-energy estimate to a homotopy
between different unbounded exterior profiles.
\end{proof}

\begin{lm}[Restriction on the critical summand]\label{lmCriticalRestriction}
For $K_1\subset K_2$, there is $\delta_{\mathrm{crit}}>0$ such that, for
$0<\delta<\delta_{\mathrm{crit}}$, indicator restriction sends the critical
summand of Theorem~\ref{mainThmA} into the critical summand and induces
\[
H^*(K_2)\otimes\Lambda_{[0,\delta)}
\longrightarrow H^*(K_1)\otimes\Lambda_{[0,\delta)},
\qquad \operatorname{res}\otimes\id.
\]
\end{lm}

  \begin{proof}
  First suppose $K_1\Subset\operatorname{int}K_2$. Choose an exact pair
  $U_0\Subset U_1$ around $\partial K_1$, disjoint from $\partial K_2$,
  with $U_0$ containing the target profile collar. Choose a fixed
  source slope $a_2>0$ below the first Reeb period of $\partial K_2$
  and sufficiently small that the source Hamiltonian vector field is
  small on $M\setminus U_0$. On noncompact $M$ use a common fixed
  exterior profile $G_\infty$.
  
  For arbitrarily large target slopes $a_1$, choose ordered endpoint
  models and a monotone continuation. Outside $U_0$, the target
  Hamiltonian depends on $a_1$ only through spatial constants, while
  the spatially nonconstant part of the source is uniformly small.
  After subtracting componentwise spatial constants, stretching the
  continuation parameter, and taking sufficiently small regular
  perturbations, the hypotheses of Proposition~\ref{prpRelEn} hold
  uniformly in $a_1$. All periodic ends outside $U_0$ are constant.
  Thus there is $\hbar_{12}>0$, independent of $a_1$, such that
  trajectories of topological energy below $\hbar_{12}$ have vanishing
  exterior relative energy. Stokes consequently gives
  \begin{equation}\label{eqCriticalRestrictionAction}
  E_{\mathrm{top}}(u)
  =
  \mathcal A_{\mathrm{source}}(x_-)
  -
  \mathcal A_{\mathrm{target}}(x_+),
  \end{equation}
  using constant cappings and the target primitive $\theta_1$
  on its Reeb collar.
  
  Choose a Morse function $f$ on $K_2$ with
  $K_1=\{f\le0\}$, equal to $r_1-1$ near $\partial K_1$,
  and with outward gradient near $\partial K_2$. Use
  $\mu_jf-c_j$, with $\mu_j,c_j\to0$, as the source Morse
  regularization, and the same function on $K_1$ for the target.
  Keep the continuation constant on this common Morse core.
  
  Since $a_2$ is subperiod, the source has only critical and exterior
  generators, with respective action centers $0$ and
  $-C_2$, where $C_2=a_2e_{0,2}>0$. Choose
  $\delta_{\mathrm{crit}}$ smaller than $\hbar_{12}$, $C_2/4$,
  and the two critical splitting windows. By
  \eqref{eqCriticalRestrictionAction}, the action gap excludes
  all mixed critical--exterior coefficients in the source complex
  within this window. Hence its critical block has representatives
  consisting entirely of critical generators in the full complex.
  Proposition~\ref{prpFiniteModelPackage} identifies their structural
  images with the critical summand of indicator cohomology.
  
  The same action identity excludes critical-to-target-inner-Reeb
  coefficients, since the source critical actions tend to zero and
  the target Reeb actions tend to strictly positive periods.
  Moreover, a source critical point $x_-$ outside $K_1$ cannot map
  to a target critical point $x_+$ inside $K_1$, because
  \[
  E_{\mathrm{top}}(u)
  =
  \mu_j\bigl(f(x_+)-f(x_-)\bigr)<0.
  \]
  
  It remains to identify the map between critical generators in
  $K_1$. Choose an inward collar cut $r_1=r_c<1$ enclosing their
  critical points. Near this cut the continuation is constant, with
  \[
  H_s=\mu_j(r_1-1)-c_j,\qquad
  B_s=H_s(r_c)-r_cH_s'(r_c)=-\mu_j-c_j,
  \qquad B_s'=0.
  \]
  The relative-energy vanishing above and the inward-cut version
  of Lemma~\ref{lmIntegratedOuterNoEscape} confine every
  low-energy continuation trajectory with both ends in the
  common Morse core to that core. The confined continuation
  is constant and therefore induces the identity. By the
  standard confined Floer--Morse comparison, the critical map
  is consequently the Morse cochain projection
  \[
  CM^*(f;K_2)\longrightarrow CM^*(f|_{K_1};K_1),
  \]
  which induces ordinary cohomological restriction. Both limiting
  critical action centers are zero, so Lemma~\ref{lmAlgTruncUwMaps}
  gives Novikov factor $1$. Target outer terms vanish in the
  indicator colimit by
  Proposition~\ref{prpFiniteModelPackage}
  \textup{(\ref{itFMPkgOuterVanish})}.
  This proves the claim for strict inclusions.
  
  For a general inclusion choose a small outward contact thickening
  $K_2^+$ of $K_2$. Both inclusions $K_i\subset\operatorname{int}K_2^+$
  are strict, so the preceding case applies. Since
  $H^*(K_2^+)\to H^*(K_2)$ is an isomorphism, functoriality
  of indicator restriction gives the assertion after shrinking
  the window to the minimum of the two strict-inclusion windows.
  \end{proof}

\begin{proof}[Proof of Theorem~\ref{thmIntroIsotopyRestriction}]
Apply Theorem~\ref{mainThmA} at the endpoints. We first check that restriction may be represented, near the isolating family \(\Gamma\), by an actual monotone Reeb-adapted continuation.

Let \(Z\) be the Liouville vector field of the fixed primitive \(\theta\). Since \(\theta|_{\partial K^\tau}\) is of positive contact type, \(Z\) is positively transverse to every \(\partial K^\tau\). After subdividing \([0,1]\) sufficiently finely, on each subinterval a neighborhood of the corresponding part of the evaluation worldtube is contained in a fixed \(Z\)-flowbox
\[
\Psi(y,t),
\]
in which
\[
\partial K^\tau=\{t=s_\tau(y)\}.
\]
Since \(K^{\tau'}\subset K^\tau\) for \(\tau\le\tau'\), we have
\[
s_{\tau'}\le s_\tau .
\]
Thus the Liouville coordinate relative to \(\partial K^\tau\),
\[
\rho_\tau=t-s_\tau(y),
\]
is pointwise nondecreasing in \(\tau\).

Choose compatible ordered finite-slope endpoint systems, increasing the target slope if necessary, and interpolate their slopes monotonically. Taking
\(\alpha_\tau=\theta|_{\partial K^\tau}\), the preceding inequality allows the corresponding locally slope-adapted finite-slope models and their inner smoothings to be chosen pointwise nondecreasing in \(\tau\). They therefore give a monotone Reeb-adapted local continuation in the sense of Definition~\ref{dfReebAdaptedLocalDatum}. Extending this to a global monotone continuation is achieved by fixed-cutoff patching with any global monotone continuation between the chosen ordered endpoint systems; since the cutoff is independent of the continuation parameter, monotonicity is preserved. Hence each short piece is realized by an actual monotone Reeb-adapted continuation representing restriction.

Choose such a finite subdivision
\[
0=t_0<\cdots<t_N=1
\]
fine enough that every short piece lies in a parameter neighborhood of
Proposition~\ref{prpBarrierShortIsotopy}, and let
\(\epsilon_{\mathrm{iso}}>0\) be the minimum of the corresponding short-piece barrier windows. Choose
\[
0<\delta
<
\min\{\delta_{\mathrm{src}}(a),\delta_{\mathrm{tgt}}(a),\epsilon_{\mathrm{iso}}\},
\]
where \(\delta_{\mathrm{src}}(a)\) and \(\delta_{\mathrm{tgt}}(a)\) are the
endpoint Theorem~\ref{mainThmA} cutoff windows at~$K^0$ and~$K^1$.

Choose one ordered pair of sufficiently far finite-slope endpoint models for
which the $\gamma^0$- and $\gamma^1$-blocks are defined in this window and the
above Reeb-adapted continuation exists. By
Proposition~\ref{prpFiniteModelPackage}\textup{(\ref{itFMPkgIdentify})}, under
the canonical $SH^*_{uw}$-identifications every further fixed-$\alpha$
structural continuation acts as the identity on these blocks. Hence each
chosen finite-slope block maps isomorphically onto the corresponding indicator
summand in the inner colimit, and the indicator restriction block is determined
by the continuation between these finite-slope representatives. This is just
the naturality of the continuation diagram with the endpoint structural maps.

Corollary~\ref{cyTruncatedTransportSubdivision} now identifies this finite-slope local block of restriction with
\[
T^{A_0-A_1}
\bigl(\iota_\Gamma\otimes\id_{\Lambda_{[0,\delta)}}\bigr).
\]
Since the finite-slope collar models vanish on the boundary core,
Lemma~\ref{lmReebGeneratorCapture} identifies the limiting endpoint action
centers with the periods
\[
A_i=T_\theta(\gamma^i).
\]
Hence \(A_0-A_1=\Delta A\), and the local block is
\[
T^{\Delta A}
\bigl(\iota_\Gamma\otimes\id_{\Lambda_{[0,\delta)}}\bigr).
\]
By the finite-determination paragraph above, the same formula is therefore the
restriction map on the transported indicator $\Gamma$-summand.

The same \(\epsilon_{\mathrm{iso}}\) works for every sub-isotopy after refining the same finite subdivision. A common full window for all sub-isotopies therefore exists whenever the Theorem~\ref{mainThmA} cutoff windows are uniformly bounded below along the family.
\end{proof}

\subsection{Restriction under inclusions of domains}
\label{SubsecRestrictionInclusions}
We now turn to continuation maps involving an inclusion of domains $K_1\subset
K_2$ which are not necessarily related by an isotopy. The two domains are
allowed independent positive contact-type primitives.

\begin{ex}\label{exATF}
A basic example is when $M$ is the affine variety $\bC^2\setminus\{xy-1\}$ equipped with the symplectic form $\mathrm{Im}\bigl(\frac{dx\wedge dy}{xy-1}\bigr)$. Then $M$ carries a Lagrangian torus fibration $\pi\colon M\to\bR^2$ with a nodal singular fiber over the origin. Let $B\simeq\bR^2$ be the integral affine manifold with singularities coming from Arnold--Liouville. Let $P_2\subset B$ be a convex polygon containing the singular value, and let $P_1\subset P_2$ be a convex polygon not intersecting the monodromy invariant line (Figure~\ref{FigAdmInc}). Take $K_i:=\pi^{-1}(P_i)$. Then each $K_i$ is a Liouville domain, but $K_1$ is not a Liouville subdomain of $K_2$ with respect to any contact form on $K_1$.
\end{ex}

This is the model situation motivating Theorem~\ref{mainThmB}: although the
two domains are individually of contact type, the regular-locus reduction
to a common-primitive isotopy is unavailable.

\begin{figure}
\begin{center}
\begin{tikzpicture}
\draw (0,-1) -- (0,2) -- (3,2) -- (4,-1) -- (0,-2) -- cycle;
\node at (2,0) {\large $\star$};
\draw[dashed] (0,0) -- (3.7,0);
\fill[gray] (0,1) -- (3.35,1) -- (3,2) -- (0,2) -- cycle;
\node at (2,-0.5) {$P_2$};
\node at (2,1.5) {$P_1$};
\end{tikzpicture}
\end{center}
\caption{An inclusion of compact domains with partially overlapping boundaries.}\label{FigAdmInc}
\end{figure}

\begin{df}\label{dfCommonNonCommon}
Let $K_1\subset K_2$ be an inclusion of compact domains in $M$. A Reeb component $\gamma$ of $\partial K_2$ is \emph{strictly common for $K_1\subset K_2$} if $\gamma\subset \mathrm{int}(\partial K_1\cap\partial K_2)$. It is \emph{non-common} if there exists an open neighborhood $W\subset M$ of $\gamma$ with $W\cap K_1=\emptyset$. 
\end{df}
A component that is not strictly common may still match a component of $\partial K_1$ by a monotone isotopy of radial fronts. We introduce the relevant notion in the following definition. 

\begin{df}[Weakly common component]
\label{dfWeaklyCommonComponent}
Fix primitives $\theta_i$ near $\partial K_i$, and write $Z_1$ for the
Liouville vector field of $\theta_1$.
Let $\gamma\subset\partial K_2$ be a Reeb component which is not strictly common.
We say that $\gamma$ is \emph{weakly common} if there exist an open set
$W\subset M$ on which $\theta_1$ is defined, a local slice $Y\subset W$
transverse to $Z_1$, and a map
\[
\Psi\colon Y\times(-r,r)\longrightarrow W,
\qquad
\Psi(y,t)=\phi^t_{Z_1}(y),
\]
which is a diffeomorphism onto its image, together with a smooth family
of graph functions
\[
s_\tau\colon Y\longrightarrow(-r,r),
\qquad\tau\in[0,1],
\]
satisfying $s_{\tau'}\le s_\tau$ whenever $\tau\le\tau'$, and an
isolated family $\Gamma=\{\Gamma^\tau\}$ of closed characteristics (in
the sense of Definition~\ref{dfReebComponentIsotopy}) on the graphs
\[
\Sigma^\tau
:=
\Psi\bigl\{(y,s_\tau(y))\bigr\},
\]
all of which are contained in the image of $\Psi$, such that:
\begin{enumerate}
\item
$\Sigma^0=\partial K_2\cap W$ and $\Sigma^1=\partial K_1\cap W$;
\item $\Gamma^0=\gamma$ (up to $\theta_1|_{\Sigma^0}$-reparametrization);
\item $\Gamma^1$ is a Reeb component $\gamma^1\subset\partial K_1$;
\item the radial-front family is $\theta_1$-action preserving: the common
value
\[
T_{\theta_1}(\Gamma^\tau)
\]
is independent of~$\tau$.
\end{enumerate}
The datum
$(W,\Psi,\{s_\tau\},\Gamma)$ is understood to have been fixed. 

In the statements below, ``common components'' means strictly common or
weakly common.
\end{df}

The point of weak commonness is that the corresponding Reeb component may
persist under a moving radial front even though the relevant pieces of
$\partial K_2$ and $\partial K_1$ do not literally coincide.
The motivating example, used in a companion paper, is a long flat edge of
a polygon containing a proper sub-edge of a smaller polygon: after corner
smoothing, sliding the rounding profile inward along the edge produces the
weakly common isotopy.

\begin{df}\label{dfAdmissibleInclusion}
An inclusion $K_1\subset K_2$ of compact domains in $M$ is \emph{admissible} if every Reeb component of $\partial K_2$ is either strictly common, weakly common, or non-common.
\end{df}

\begin{df}[Local restriction realization]
\label{dfLocalRestrictionRealization}
Fix an inclusion $K_1\subset K_2$, primitives $\theta_i$ near $\partial K_i$,
finite-slope parameters $I_i=(a_i,\alpha_i,\epsilon_i)$ with
$\alpha_i=\theta_i|_{\partial K_i}$ and $H_{K_2,I_2}\le H_{K_1,I_1}$, and a
common source component $\gamma$. Let $\Gamma_\gamma$ denote the chosen
radial-front isotopy of $\gamma$ from
Definition~\ref{dfWeaklyCommonComponent} in the weakly common case, and the
stationary family in the strictly common case.

A \emph{local restriction realization for $\gamma$} is a monotone continuation

$$
H_{K_2,I_2}\longrightarrow H_{K_1,I_1}
$$
with the following local form near $\gamma$.  The associated isolated Reeb
family is obtained in two stages.  First, keeping the underlying closed
characteristic $\gamma$ fixed, we deform the local finite-slope data from the
source data to data using the target primitive and slope.  Second, with these
target data fixed, we transport the characteristic along the radial-front
isotopy $\Gamma_\gamma$ to the corresponding Reeb component of $\partial K_1$.
For a strictly common component the second stage is stationary.

Denote the resulting concatenated isolating family by
$\widetilde{\Gamma}_\gamma$.  We require the continuation, on a neighborhood
of this family, to be Reeb-adapted in the sense of
Definition~\ref{dfReebAdaptedLocalDatum}. Write $\Delta\cA_\gamma$ for its
source-minus-target action difference, computed using a fixed primitive on the
local exact worldtube.
\end{df}

For every common component set
\[
\iota_\gamma:=\iota_{\Gamma_\gamma}.
\]
If $\gamma$ is strictly common, then $\Gamma_\gamma$ is the stationary family,
so $\iota_\gamma=\mathrm{id}$ by Proposition~\ref{prpTransportLocalReeb}.
\begin{prp}[Restriction blocks from local realizations]
\label{prpAdmissibleInclusionBlocks}
Fix an inclusion $K_1\subset K_2$, primitives $\theta_i$ near $\partial K_i$,
parameters $I_i=(a_i,\alpha_i,\epsilon_i)$ with $\alpha_i=\theta_i|_{\partial K_i}$
and $H_{K_2,I_2}\le H_{K_1,I_1}$, and a period cutoff $a_{\mathrm{cut}}>0$. Assume that
every common source component $\gamma$ with $T_{\theta_2}(\gamma)\le a_{\mathrm{cut}}$
admits a local restriction realization. There exists $\delta_0>0$ such that for
every $0<\delta<\delta_0$, the truncated continuation
\[
\rho^{[0,\delta)}_{K_2,K_1}\colon
HF^*_{[0,\delta)}(H_{K_2,I_2})
\longrightarrow
HF^*_{[0,\delta)}(H_{K_1,I_1})
\]
has the following blocks on every common or non-common source Reeb component
$\gamma$ with $T_{\theta_2}(\gamma)\le a_{\mathrm{cut}}$:
\begin{enumerate}
\item if $\gamma$ is non-common, the restriction of
$\rho^{[0,\delta)}_{K_2,K_1}$ to the $\gamma$-summand is zero;
\item if $\gamma$ is common, writing $\Gamma_\gamma$ for the chosen
family associated to $\gamma$, its
summand maps into the summand of $\Gamma_\gamma^1$, and, after identifying
both blocks with the canonical groups $SH^*_{uw}(\gamma)$ and
$SH^*_{uw}(\Gamma_\gamma^1)$, the local block is
\[
T^{\Delta\cA_\gamma}
\bigl(
\iota_\gamma
\otimes
\id_{\Lambda_{[0,\delta)}}
\bigr);
\]
\end{enumerate}
\end{prp}
\begin{proof}
Truncated continuation maps between fixed endpoints are
independent, up to chain homotopy, of the choice of monotone homotopy
(Section~\ref{SecInfFloer}). It therefore suffices, for each source summand, to
compute the block under one convenient continuation with the same endpoints;
the result transfers to every other continuation, including
$\rho^{[0,\delta)}_{K_2,K_1}$.

\smallskip\noindent\textbf{Common components.}
Let $\gamma$ be common, and let $\rho_\gamma$ be a local restriction realization
for~$\gamma$. Let $\widetilde{\Gamma}_\gamma$ be the concatenated isolating
family of Definition~\ref{dfLocalRestrictionRealization}.
Corollary~\ref{cyTruncatedTransportSubdivision}, together with the identification
of unweighted transport with $\iota_{\widetilde{\Gamma}_\gamma}$ from
Proposition~\ref{prpTransportLocalReeb}, gives, for a sufficiently small
window, the truncated local block
\[
T^{\Delta\cA_\gamma}
\bigl(
\iota_{\widetilde{\Gamma}_\gamma}
\otimes
\id_{\Lambda_{[0,\delta)}}
\bigr).
\]
The preliminary segment is a change of finite-slope/contact-form presentation
of the same isolated Reeb component. By
Proposition~\ref{prpTransportLocalReeb} its unweighted transport is the
canonical parameter identification, and after identifying all finite-slope
presentations with $SH^*_{uw}(\gamma)$ this segment acts as the identity.
By concatenation,
$\iota_{\widetilde{\Gamma}_\gamma}=\iota_{\Gamma_\gamma}=\iota_\gamma$.
For a strictly common component the second segment is stationary, so
$\iota_\gamma$ is the identity. Therefore the local restriction block of
$\rho_\gamma$ is
\[
T^{\Delta\cA_\gamma}
\bigl(
\iota_\gamma
\otimes
\id_{\Lambda_{[0,\delta)}}
\bigr).
\]
By the independence recalled above, the same block holds for
$\rho^{[0,\delta)}_{K_2,K_1}$.

\smallskip\noindent\textbf{Non-common components.}
Let $\gamma$ be non-common. Factor a continuation as
$H_{K_2,I_2}\to H^\delta\to H_{K_1,I_1}$, where the first factor is a short
inward push of $\partial K_2$ into a gap neighborhood
$W_\gamma\Subset M\setminus K_1$: keep slope and modeling primitive fixed,
increase the local Hamiltonian along the orbits while remaining below
$H_{K_1,I_1}$. This first factor is Reeb-adapted for the
persisting isolated family $\gamma^\tau$, so the local barrier applies and the
action drops by a positive amount $d_\gamma$. In a sufficiently small window
supplied by that barrier, a trajectory with source asymptotic in the
$\gamma$-block cannot leave the moving local basin and still contribute in the
truncated window; hence there are no off-block outputs from that source
summand. Corollary~\ref{propLocalContinuationTruncated} then identifies the
restriction of the whole first continuation to the $\gamma$-summand as
$T^{d_\gamma}(\Phi_{uw}\otimes\id)$, and this vanishes over
$\Lambda_{[0,\delta)}$ whenever the window width is also $<d_\gamma$. The
second factor may be any monotone continuation. Since the first factor already
vanishes on the $\gamma$-summand in the window, the composite vanishes, and
hence so does every continuation with the same endpoints.

\smallskip\noindent\textbf{Common window.}
There are only finitely many common or non-common source components of
period at most~$a_{\mathrm{cut}}$. Take $\delta_0>0$ smaller than all of
the finitely many: source finite-model splitting and block windows of
Proposition~\ref{prpFiniteModelPackage}; the corresponding target finite-model windows for the
common target components; the local common-component transport windows of
Corollary~\ref{cyTruncatedTransportSubdivision}; and the non-common
local-barrier windows and action drops $d_\gamma$. Their minimum is
positive.
\end{proof}

We now turn to existence of local restriction realizations.

\begin{lm}[Existence of local restriction realizations]
\label{lmRestrictionAdmissiblePath}
Fix an inclusion $K_1\subset K_2$, local primitives $\theta_i$, a
period cutoff $a_{\mathrm{cut}}$, and source data $I_2$. Assume every
source component of $T_{\theta_2}$-period at most $a_{\mathrm{cut}}$ under
consideration is common (strictly or weakly). For all sufficiently
large admissible target slopes $a_1$, subject to
\[
H_{K_2,I_2}\le H_{K_1,I_1},
\]
every such common source component admits a local restriction realization from
$H_{K_2,I_2}$ to $H_{K_1,I_1}$. The realization may be chosen so that, with
respect to a fixed primitive on the local worldtube,
\[
\Delta\cA_\gamma=0.
\]
\end{lm}

\begin{proof}
  Fix a common source component $\gamma$, with $I_2$ $\gamma$-admissible.
  Near its evaluation image both primitives are of positive contact type
  on $\partial K_2$: in the weakly common case this follows from the
  $Z_1$-flowbox and the outward-side identification supplied by
  $K_1\subset K_2$.
  
  Put $\theta_s=(1-s)\theta_2+s\theta_1$, let $\rho_s$ be its Liouville
  coordinate relative to $\partial K_2$, and write $x_s=e^{\rho_s}-1$.
  On a common collar,
  \[
  x_s=b_sx_0,\qquad b_s>0,
  \]
  where $b_s$ is a smooth function of both position and parameter.
  For every sufficiently large $a_1$, the scalar slope
  $a_s=a_2e^{Cs}$, with $C=\log(a_1/a_2)$, satisfies
  \[
  \partial_s(a_sb_s)\ge0,
  \qquad
  a_s>2T_{\theta_s}(\gamma).
  \]
  Indeed, $\partial_s\log b_s$ is uniformly bounded, and source
  admissibility is strict. Thus $(a_sx_s)_+$ is a pointwise
  nondecreasing admissible family.
  
  For the second stage use the fixed $\theta_1$-flowbox of
  Definition~\ref{dfWeaklyCommonComponent}. Its radial coordinate is
  $\rho_\tau=t-s_\tau(y)$, so
  $a_1(e^{\rho_\tau}-1)_+$ is nondecreasing. Admissibility persists
  because $T_{\theta_1}(\Gamma^\tau)$ is constant. In the strictly
  common case this stage is stationary. The two stages agree at
  their common endpoint.
  
  These local models extend to an ordered continuous global family
  with the prescribed endpoints. To see this, use the target's
  nondecreasing outer profile in the moving-front stage, and choose
  a fixed neighborhood of the worldtube sufficiently thin outside
  $K_2$ that the first-stage profiles remain in their linear branches.
  The local endpoint functions then agree there with the prescribed
  global endpoints. Patch with a fixed spatial cutoff, keeping the
  exterior family equal to the source during the first stage and
  interpolating monotonically to the target during the second.
  
  We verify the smooth realization with one estimate. Choose common
  radial roundings $f_\delta$, normalized by $f_\delta(0)=0$, whose
  derivative is $\delta^2$ on $(-\infty,0]$, increases strictly to $1$
  on $(0,\delta)$ through a fixed rescaled transition, and equals $1$
  on $[\delta,\infty)$. Near the moving component set
  \[
  K_{p,\delta}
  =
  a_p f_\delta(e^{\rho_p}-1)-\delta^2.
  \]
  For small $\delta$ these are locally slope-adapted smoothings, with
  smoothing parameter $8\log(1+\delta)$. Continue them through the
  outer profiles and use the same fixed-cutoff construction globally.
  
  After parametrizing the concatenation by $p\in[0,1]$, there is a
  constant $B$, independent of $\delta$, such that
  \[
  \partial_pK_{p,\delta}\ge-B\delta,
  \qquad
  \|\partial_pK_{p,\delta}\|_{C^0}\le B.
  \]
  For the first stage, this follows from
  $\partial_sx_s=O(|x_s|)$: the negative-side slope is $O(\delta^2)$,
  the rounded corner has $|x_s|\le\delta$, and on the positive
  linear branch rounding changes the profile only by an
  $O(\delta)$ constant. The moving-front stage is already
  monotone. Fixed-cutoff patching with endpoint approximations
  having $O(\delta)$ errors preserves these bounds.
  
  Consequently
  \[
  H_{p,\delta}
  =
  K_{p,\delta}+(B+1)\delta(p-1)
  \]
  is globally monotone. Its spatial correction tends uniformly to
  zero, and subtracting that correction recovers the locally
  slope-adapted family $K_{p,\delta}$ with uniformly bounded
  parameter variation. Thus Definition~\ref{dfReebAdaptedLocalDatum}
  applies, with a subdivision independent of $\delta$.
  The endpoint approximations may be chosen cofinally from below;
  sufficiently small regular perturbations preserve monotonicity
  and Reeb-adaptedness. Reparametrizing to be constant near the ends
  therefore gives a local restriction realization of the prescribed
  finite-slope continuation.
  
  Finally compute actions using the fixed primitive $\theta_1$.
  The underlying characteristic is fixed during the first stage,
  and its $\theta_1$-period is constant during the second by weak
  commonness. The smoothing errors and spatial corrections tend
  uniformly to zero. Generator capture therefore gives
  \[
  A_{\mathrm{source}}
  =
  T_{\theta_1}(\gamma)
  =
  T_{\theta_1}(\Gamma^1)
  =
  A_{\mathrm{target}},
  \]
  hence $\Delta\cA_\gamma=0$.
  
  Taking the maximum of the target-slope thresholds for the finitely
  many source components proves the assertion simultaneously.
  \end{proof}

\subsection{Proof of Theorem~\ref{mainThmB} and Corollary~\ref{mainThmC}}
\label{SubsecProofThmB}

\begin{proof}[Proof of Theorem~\ref{mainThmB}]
Fix an inclusion $K_1\subset K_2$, positive contact-type primitives $\theta_i$ near $\partial K_i$, and
period cutoffs $a,b>0$ as in the statement. Apply Theorem~\ref{mainThmA} to
both domains, allowing each domain its own fixed auxiliary profile width
$\epsilon_2$ and~$\epsilon_1$ respectively, and obtaining windows
$\delta_2(a)$ and $\delta_1(b)$. Choose source finite-slope data
$I_2=(a_2,\alpha_2,\epsilon_2)$ sufficiently far out that every relevant source
component of $T_{\theta_2}$-period at most~$a$ is represented by its canonical
finite local block. Choose target data $I_1=(a_1,\alpha_1,\epsilon_1)$
sufficiently far out that
$H_{K_2,I_2}\le H_{K_1,I_1}$, all relevant common target components are
represented by their canonical blocks, and
Lemma~\ref{lmRestrictionAdmissiblePath} supplies the local restriction
realizations with $\Delta\cA_\gamma=0$ for every common source component under
consideration. Lemma~\ref{lmRestrictionAdmissiblePath} compares these finite
models even when their profile widths differ. Let $\epsilon_0$ be the window from
Proposition~\ref{prpAdmissibleInclusionBlocks}, and let
$\delta_{\mathrm{crit}}$ be the window of Lemma~\ref{lmCriticalRestriction}. Choose
\[
0<\delta
<
\min\bigl(\delta_2(a),\,\delta_1(b),\,\epsilon_0,\,\delta_{\mathrm{crit}}\bigr).
\]
Proposition~\ref{prpAdmissibleInclusionBlocks} applied to the truncated
continuation
\[
\rho_{I_2,I_1}\colon
HF^*_{[0,\delta)}(H_{K_2,I_2})
\longrightarrow
HF^*_{[0,\delta)}(H_{K_1,I_1})
\]
gives the blocks on the relevant finite local source summands: non-common
blocks map to zero; common blocks map by
$T^{\Delta\cA_\gamma}(\iota_\gamma\otimes\id)=\iota_\gamma\otimes\id$
(using $\Delta\cA_\gamma=0$ from
Lemma~\ref{lmRestrictionAdmissiblePath}). Naturality under the fixed-$\alpha$ structural
parameter continuations on each domain
(Proposition~\ref{prpFiniteModelPackage}\textup{(\ref{itFMPkgIdentify})}
and the construction of Theorem~\ref{mainThmA}) identifies these finite
blocks with the canonical Theorem~\ref{mainThmA} summands, so passing to
the colimits yields the asserted Reeb restriction blocks. The critical block
is ordinary singular restriction by Lemma~\ref{lmCriticalRestriction}.
\end{proof}

\begin{proof}[Proof of Corollary~\ref{mainThmC}]
  Translation of the action coordinate reduces the proof to $t=0$.
  
  Fix a positive contact-type primitive $\theta$ near $\partial K$. For every period cutoff $a$,
  Theorem~\ref{mainThmA} gives, for $0<\epsilon<\delta(a)$, compatible
  isomorphisms
  \[
  SH^{\le a}_{M,[0,\epsilon)}(K;\theta)
  \xrightarrow{\sim}
  RH^{\le a}(K;\theta)\otimes\Lambda_{[0,\epsilon)}.
  \]
  For any abelian group $A$,
  \[
  A\otimes\Lambda_{[0,\epsilon)}
  \cong\bigoplus_{-\epsilon<\lambda\le0}A\,T^\lambda.
  \]
  The maps for shrinking windows are the inclusions of these direct sums.
  Their inverse limit consists precisely of the coefficient at $T^0$, and
  hence is canonically $A$. Applying this with $A=RH^{\le a}(K;\theta)$,
  compatibility under shrinking the window gives
  \[
  \varprojlim_{\epsilon\to0}
  SH^{\le a}_{M,[0,\epsilon)}(K;\theta)
  \xrightarrow{\sim}
  RH^{\le a}(K;\theta).
  \]
  Taking the filtered colimit over $a$ yields
  \[
  SH^{*,\mathrm{bd}}_{M,0^+}(K;\theta)
  \xrightarrow{\sim}
  \varinjlim_a RH^{\le a}(K;\theta)
  =
  RH^*(K).
  \]
  
  We prove naturality. Let $K_1\subset K_2$ be admissible, with positive contact-type primitives
  $\theta_i$ near $\partial K_i$. For each cutoff $a$ on $K_2$, choose a
  cutoff $b(a)$ on $K_1$, increasing and cofinal, such that every common
  component of $\partial K_2$ of $T_{\theta_2}$-period at most $a$ has image
  of $T_{\theta_1}$-period at most $b(a)$. Theorems~\ref{mainThmA}
  and~\ref{mainThmB} give, for sufficiently small $\epsilon$, a commutative
  square
  \[
  \xymatrix{
  SH^{\le a}_{M,[0,\epsilon)}(K_2;\theta_2)
    \ar[r]\ar[d] &
  RH^{\le a}(K_2;\theta_2)\otimes\Lambda_{[0,\epsilon)}
    \ar[d]^{\iota\otimes\id}\\
  SH^{\le b(a)}_{M,[0,\epsilon)}(K_1;\theta_1)
    \ar[r] &
  RH^{\le b(a)}(K_1;\theta_1)\otimes\Lambda_{[0,\epsilon)}.
  }
  \]
  Passing first to the inverse limit in $\epsilon$ and then to the filtered
  colimit in $a$ gives
  \[
  \xymatrix{
  SH^{*,\mathrm{bd}}_{M,0^+}(K_2;\theta_2)
    \ar[r]^{\sim}\ar[d] &
  RH^*(K_2)\ar[d]^{\iota}\\
  SH^{*,\mathrm{bd}}_{M,0^+}(K_1;\theta_1)
    \ar[r]_{\sim} &
  RH^*(K_1).
  }
  \]
  This proves naturality. Since the right-hand side is independent of the
  primitive, it also proves the asserted independence of $\theta$.
\end{proof}

\begin{df}[Joint primitive and joint gap]
\label{dfJointGap}
A \emph{joint primitive} for an inclusion $K_1\subset K_2$ is a primitive
$\widetilde\theta$ of $\omega$ on a neighborhood of $\partial K_1\cup\partial K_2$
which is Liouville near $\partial K_2$. Writing
$\widetilde T(\gamma):=-\int_\gamma\widetilde\theta$, the inclusion is
\emph{jointly $\delta$-gapped at cutoff~$a$} if whenever $\gamma_2$ is a Reeb
component of $\partial K_2$ with $\widetilde T(\gamma_2)\le a$ and $\gamma_1$
is a Reeb component of $\partial K_1$ in the same class in
$H_1(\partial K_1\cup\partial K_2;\bZ)$, one has either
$\widetilde T(\gamma_2)=\widetilde T(\gamma_1)$ or
$\bigl|\widetilde T(\gamma_2)-\widetilde T(\gamma_1)\bigr|\ge\delta$.
These differences depend only on~$\omega$.
\end{df}

\begin{rem}[Window summary]
\label{thmUniformWindowInterface}
\label{prpBarrierWindows}
\label{lmBarrierWindowA}
\label{lmBarrierWindowRestriction}
\label{lmBarrierWindowIsotopy}
\label{remRestrictionWindowFormula}
\label{prpUniformSingleDomainBarrier}
\label{dfIndicatorWindows}
Write $\hbar$ for the constant of Proposition~\ref{prphbarfilt},
$\epsilon_{\mathrm{loc}}$ for a common $\epsilon_0$-barrier as in
Proposition~\ref{prpReebEpsilonBarriers} for the Reeb components of period
at most a cutoff~$c$ \textup{(}with $\gamma$-admissible finite-slope
data\textup{)}, $\epsilon_{\mathrm{rest}}$ for a window in which all
short-piece restriction blocks of
Proposition~\ref{prpAdmissibleInclusionBlocks} at that cutoff are valid
\textup{(}together with the non-common drops\textup{)}, and
$\delta_{\mathrm{crit}}$ for the window of Lemma~\ref{lmCriticalRestriction}.
Write
$\epsilon_{\mathrm{iso}}:=\min_j\epsilon_j$ for the minimum short-piece
continuation window of a finite local-barrier subdivision of an isolating
isotopy as in Theorem~\ref{thmIntroIsotopyRestriction}. Admissible windows may be taken as
follows.
\[
\begin{array}{c|c}
\text{result} & \text{admissible window}\\
\hline
\text{Theorem~\ref{mainThmA}}
  \ \textup{(\ref{itBarrierWinA})}
  & \min\{\delta(\gamma_0),\,\epsilon_{\mathrm{loc}}\}
    \ \text{or same-class period gaps}\\
\text{Theorem~\ref{mainThmB}}
  \ \textup{(\ref{itBarrierWinB})}
  & \min\{\delta_2(c),\,\delta_1(d),\,\epsilon_{\mathrm{rest}},\,\delta_{\mathrm{crit}}\}
    \ \text{or a joint gap as below}\\
\text{Theorem~\ref{thmIntroIsotopyRestriction}}
  \ \textup{(\ref{itBarrierWinIso})}
  & \text{any }0<\delta<\min\{\delta_{\mathrm{src}}(c),\,\delta_{\mathrm{tgt}}(c),\,\epsilon_{\mathrm{iso}}\}
\end{array}
\]
\begin{enumerate}[label=\textup{(\Alph*)},ref=\Alph*]
\item\label{itBarrierWinA}
Here $\delta(c)=\min\{\delta(\gamma_0),\,\epsilon_{\mathrm{loc}}\}$, via the
local-barrier alternative of Remark~\ref{remFMPkgBarrierWindow}; the
period-gap alternative is Remark~\ref{remMainAGapBarrier}.
\item\label{itBarrierWinB}
\label{prpMainBGapWindow}
\label{prpJointGapControl}
Here $\delta_2(c)$ and $\delta_1(d)$ are the windows of
Part~\ref{itBarrierWinA} on the source and target. Alternatively, if the
inclusion is jointly $\delta$-gapped at the relevant cutoff
(Definition~\ref{dfJointGap}) for some $\delta\in(0,\hbar]$ and
Theorem~\ref{mainThmA} holds on both sides in $[0,\delta)$, with
$\delta<\delta_{\mathrm{crit}}$ and $\delta$ below the relative-energy
threshold for the joint collar data, then the
conclusion of Theorem~\ref{mainThmB} holds in $[0,\delta)$: the blocks of
Proposition~\ref{prpAdmissibleInclusionBlocks} are known in the barrier
window $[0,\epsilon_0)$, while in $[0,\delta)$ Proposition~\ref{prpRelEn}
gives vanishing relative energy and preserved joint homology labels, so
unequal-center contributions are excluded by the gap and equal-center
contributions that vanish in $[0,\epsilon_0)$ remain zero. The same applies
fiberwise to a jointly $\delta$-gapped family.
\item\label{itBarrierWinIso}
Here $\delta_{\mathrm{src}}(c)$ and $\delta_{\mathrm{tgt}}(c)$ are the
endpoint windows of Part~\ref{itBarrierWinA}. The displayed window is taken
strictly smaller than the minimum, to match the open range
$0<\epsilon<\epsilon_\Gamma$ of
Corollary~\ref{cyTruncatedTransportSubdivision}. The same local isotopy
window $\epsilon_{\mathrm{iso}}$ works for every sub-isotopy; a common full
$\delta_0$ requires a uniform lower bound on those endpoint windows along
the family.
\end{enumerate}
If the named constants are uniform along a family \textup{(}e.g.\ by
Remark~\ref{corUniformIndicatorWindow} and
Corollary~\ref{corUniformReebBarriers}\textup{)}, so are the windows. In
particular, after fixing uniform source and target Theorem~\ref{mainThmA}
windows, a common restriction window for a family of inclusions follows if
the relevant common-component restriction realizations have uniform
Reeb-adapted bounds, the non-common action drops are uniformly positive,
and the critical windows have a uniform positive lower bound.
Uniform control of the target-slope enlargement in
Lemma~\ref{lmRestrictionAdmissiblePath} is one sufficient way to arrange the
first condition.
\end{rem}

\section{The regular locus of a Lagrangian torus fibration}
\label{subsec:local_floer_regular_fibration}

\begin{proof}[Proof of Theorem~\ref{thmRegularSymbol}]
We write $K_P:=\pi^{-1}(P)$ throughout.  The modules $M_P$, $E_P^*$, the
local affine functions $f_m$, and the displacements $d_{P,Q}(m)$ are as in
the introduction.

\smallskip\noindent\textbf{1. Uniform smallness scale and uniform $\hbar_C$.}
Choose finitely many nested precompact integral-affine coordinate domains
\[
W_i\Subset V_{0,i}\Subset V_{1,i}\Subset B_{\mathrm{reg}},
\qquad
C\subset\bigcup_i W_i,
\]
with each $V_{0,i}$ and $V_{1,i}$ a smooth convex coordinate domain, so that
\[
U_{0,i}:=\pi^{-1}(V_{0,i})
\Subset
U_{1,i}:=\pi^{-1}(V_{1,i})
\]
satisfies the exact-pair hypotheses of Proposition~\ref{prpRelEn}: $U_{0,i}$
has smooth boundary, $U_{1,i}$ deformation retracts onto $U_{0,i}$, and
$\omega$ admits a primitive on $U_{1,i}$.  For each $i$ fix once and for all a
primitive $\theta_i^0$ of $\omega$ on $U_{1,i}$.  Fix an auxiliary Riemannian
metric on a neighborhood of $C$ and a Lebesgue number $\lambda>0$ for the cover
$\{W_i\}$ of $C$.  Throughout the rest of the proof, ``sufficiently small''
means $\operatorname{diam}(P)<\lambda$.  In particular every sufficiently
small compact convex $P\subset C$ is contained in some $W_i$.  The metric and
Lebesgue number are used only to produce this uniform smallness threshold and
are not part of the data of the theorem.

Fix $i$.  Compactness of $\overline{W_i}\Subset V_{0,i}$ yields a positive
contact-collar width uniform over all smooth strongly convex $P\subset W_i$
and all centers $c\in\operatorname{int}P$.  Explicitly, in action--angle
coordinates on $V_{1,i}$ the radial primitive
\[
\theta_c=\sum_j(p_j-c_j)\,dq_j
\]
has Liouville flow $\Phi_c^\rho(q,p)=(q,c+e^\rho(p-c))$, and for
$|\rho|$ sufficiently small (independently of such $P$ and $c$) one has
$\Phi_c^\rho(K_P)\subset U_{0,i}$.

Confine all $P$-dependent contact-type almost complex structures and radial
Hamiltonian profiles to $U_{0,i}$.  On $M\setminus U_{0,i}$, choose the
Hamiltonian and almost-complex data from one fixed bounded family, with the
spatially nonconstant part of the Hamiltonians uniformly $C^1$-small on the
shell $U_{1,i}\setminus\overline{U}_{0,i}$.  Apply
Proposition~\ref{prpRelEn} using the fixed primitive $\theta_i^0$ and these
fixed exterior bounded-geometry data.  The conclusion needed later is
vanishing of the relative class: once a complement component has zero class
in $H_2(M,U_{1,i})$, its relative symplectic area vanishes with respect to any
primitive on $U_{1,i}$, in particular with respect to a varying radial
primitive $\theta_c$ used for a particular convex domain or inclusion.  Thus
no uniformity in the choice of $c$ or of the primitive is required for
Proposition~\ref{prpRelEn}.  The bounded-geometry packing estimate in the
proof of that proposition is used only on $M\setminus U_{0,i}$; the collar
confinement of Lemma~\ref{lmBorgOanc} is the exact Stokes/no-escape argument
and introduces no shape-dependent curvature constant associated with
$\partial P$.  The same arrangement applies to monotone continuation data
between nested convex domains in $W_i$.

Invoking the proofs of Proposition~\ref{prpRelEn}, Lemma~\ref{lmBorgOanc}, and
Remark~\ref{corUniformIndicatorWindow}, one obtains a positive constant
$\hbar_i$ which works uniformly for all sufficiently small domains
$P\subset W_i$ and for the finite-slope indicator data and structural
continuations used below.  Set
\[
\hbar_C:=\min_i\hbar_i>0.
\]
This is the only place where the finite affine cover, the auxiliary metric,
and the Lebesgue number enter.

\smallskip\noindent\textbf{2. Smooth strongly convex domains.}
Let $P\subset C$ be smooth, strongly convex, and sufficiently small, and fix
$i$ with $P\subset W_i$.  Work in the corresponding action--angle chart.
For each $m\in M_P\setminus\{0\}$ the function $f_m$ has a unique minimizing
point $p_m\in\partial P$.  The corresponding Morse--Bott Reeb family is
\[
\gamma_x(t)=(x+tm,p_m),
\qquad x\in T^n.
\]
It represents the fiber homology class $m$, so distinct values of $m$ give
distinct classes in $H_1(\partial K_P;\bZ)$.

Here strongly convex means that the second fundamental form of the smooth
boundary is positive definite. If $r$ is its radial gauge about an interior
point, choose the inner smoothing $H=h(r)$ with $h'>0$ and $h''>0$.
Then
\[
D_p^2H=h''\,dr\otimes dr+h'D_p^2r>0.
\]
In our sign convention $X_H=-\nabla_pH\,\partial_q$. At the inner orbit in
class $m$, $\nabla_pH=-m$, and the linearized return map satisfies
\[
(d\phi_H^1-\id)(\delta q,\delta p)=(-D_p^2H\,\delta p,0).
\]
Its kernel is exactly $\{(\delta q,0)\}$, the tangent space to the orbit
torus. This proves the Morse--Bott assertion.

The moving frame $q\mapsto q-tm$ replaces $H$ locally by
$H+\langle m,p\rangle$, which has a nondegenerate minimum in the $p$-directions.
Shrink the isolating neighborhood and deform this positive quadratic
normal form to a sufficiently small one, preserving isolation. Perturb by
a small Morse function of $q$. The confined Floer--Morse comparison
\cite[\S7]{HoferSalamon} identifies the local complex with $CM^*(T^n)$:
the positive $p$-Hessian contributes Morse index zero, and the moving-frame
linearization has no Maslov rotation. This gives the action--angle grading
shift zero and fixes the orientation-line identification. Positive definite
normal forms form a contractible space, so this identification introduces
no extra orientation local system and is natural along their deformations.
Consequently
\[
SH^*_{uw}(\gamma_m(P))\cong H^*(T^n)\cong E_P^*.
\]
For a smooth family of strongly convex domains the same construction
identifies local continuation with Morse continuation on the fiber, hence
with Gauss--Manin transport. Changes of action--angle coordinates induce
the corresponding ordinary cohomology map on the torus. The critical summand is
\[
H^*(K_P)=E_P^*,
\]
corresponding to $m=0$.

Since distinct nonzero $m$ lie in distinct $H_1$-classes, the same-class
period-gap hypothesis of Theorem~\ref{mainThmA} is vacuous.  Apply that
theorem with the uniform window $\hbar_C$ at every period cutoff.  At each
fixed cutoff only finitely many lattice classes $m$ occur, so all finite-stage
choices may be made simultaneously before passing to the cutoff colimit.
Exhaustiveness of the Reeb-period filtration, as established in the proof of
Theorem~\ref{mainThmA}, then yields the identification~\eqref{eqIntroRegularWindow}.

\smallskip\noindent\textbf{3. Restriction for smooth $Q\subset P$.}
Let $Q\subset P$ be smooth, strongly convex, and sufficiently small.  Since
$P\subset W_i$, one also has $Q\subset W_i$.  Choose $c\in\operatorname{int}Q$
and use the common primitive $\theta_c$ for both domains.  For $m\neq0$ the
source-minus-target action-center displacement is
\[
T_{P,c}(m)-T_{Q,c}(m)
=
\bigl(f_m(c)-\min_P f_m\bigr)
-
\bigl(f_m(c)-\min_Q f_m\bigr)
=
d_{P,Q}(m).
\]
By the same relative-homology argument underlying the uniform $\hbar_C$-window,
every continuation trajectory contributing in this window has relatively
trivial exterior pieces.  Consequently its source and target loops determine
the same class in the common action--angle neighborhood
$H_1(T^n\times V_{1,i};\bZ)$, and therefore have the same fiber class~$m$.
Hence the restriction map is diagonal in~$m$.

Fix $m\neq0$.  Join $P$ to $Q$ by a nested smooth strongly convex family in
$W_i$ carrying the common primitive $\theta_c$, and subdivide the resulting
isolating Reeb isotopy into finitely many short pieces as in
Proposition~\ref{prpBarrierShortIsotopy}, with local barriers $\epsilon_j$,
chosen so that
\[
\bigl|\Delta A_{m,j}\bigr|<\epsilon_j/2,
\]
where $\Delta A_{m,j}$ is the action-center displacement of the $m$-family
on that piece.  For a continuation trajectory in the $m$-block contributing
in the ambient window $[t,t+\hbar_C)$, the uniform relative-energy threshold
arranged in Step~1 gives
\[
E^{\mathrm{top}}(u)
=
\cA_{\mathrm{source}}(x_-)
-
\cA_{\mathrm{target}}(x_+).
\]
Far enough out in the finite models the endpoint actions concentrate at the
corresponding $m$-centers, so $E^{\mathrm{top}}(u)<\epsilon_j$ on a short
piece.  The local barrier therefore confines the trajectory even if
$\hbar_C>\epsilon_j$.  Preservation of the local span and this confinement
having been established independently of the ambient window width, the proof
of Corollary~\ref{propLocalContinuationTruncated}\textup{(ii)} applies, or
equivalently Lemma~\ref{lmAlgTruncUwMaps} together with the identification of
the confined unweighted continuation with Gauss--Manin transport~$\iota_j$.
Thus, in the full $\hbar_C$-window, the short-piece $m$-block is
\[
T^{\Delta A_{m,j}}
\bigl(\iota_j\otimes\id_{\Lambda_{[t,t+\hbar_C)}}\bigr).
\]
Composing over the subdivision,
\[
\sum_j\Delta A_{m,j}=d_{P,Q}(m),
\]
so the total $m$-block is multiplication by $T^{d_{P,Q}(m)}$ on the formal
generator together with ordinary restriction
$E_P^*\to E_Q^*$ on coefficients.  At a fixed period cutoff only finitely many
$m$ occur, so the required subdivisions may be chosen simultaneously before
the cutoff colimit.

For $m=0$ the summand is the critical block. The proof of
Lemma~\ref{lmCriticalRestriction} works here throughout the full
$\hbar_C$-window: choose exterior plateau heights greater than $4\hbar_C$.
The common action--angle neighborhood and Step~1 supply the relative-energy
vanishing needed at the inward cut, independently of the source slope.
There are no nonconstant Reeb generators in fiber class zero, and the
preserved fiber grading separates all other classes. The critical cocycles
therefore have source representatives without imposing the subperiod-slope
condition. The same confined Morse projection gives ordinary restriction
$H^*(K_P)\to H^*(K_Q)$, with $d_{P,Q}(0)=0$. Thus no additional
critical-window bound is imposed on $\hbar_C$.

\smallskip\noindent\textbf{4. General compact convex domains.}
Let $P\subset C$ be sufficiently small and fix $i$ with $P\subset W_i$.
Because $P\Subset W_i$, choose nested smooth strongly convex outward
approximations (obtained by smoothing support functions and adding shrinking
balls, with decreasing outer buffers)
\[
P\subset\operatorname{int}P_j\subset W_i,
\qquad
P_{j+1}\subset P_j,
\qquad
\bigcap_j P_j=P.
\]
The corresponding indicator Hamiltonians satisfy
$H_{K_{P_j}}\nearrow H_{K_P}$ pointwise.  By the truncation-colimit
principle of Lemma~\ref{lm:truncation_exactness},
\[
SH^*_{M,[t,t+\hbar_C)}(K_P)
\cong
\varinjlim_j
SH^*_{M,[t,t+\hbar_C)}(K_{P_j}).
\]
Since all sets lie in one affine chart, each $m\in M_P$ extends uniquely as
the same constant integral-affine covector over every $P_j$.  Then
$\min_{P_j}f_m\nearrow\min_P f_m$, and the smooth weighted continuation
formulas are compatible with the limit. For each fixed $m$, their remaining
action shifts tend to zero; the coefficient colimit is therefore
$\Lambda_{[t,t+\hbar_C)}$ by the open-lattice calculation in
Section~\ref{sec:floer_general_hamiltonian}. Filtered colimits commute with
the direct sum over $m$, yielding~\eqref{eqIntroRegularWindow}.

For $Q\subset P$, choose compatible smooth strongly convex outward
approximations $Q_j\subset P_j$ with both $Q_j$ and $P_j$ contained in the
same $W_i$, and pass the smooth restriction formula to the limit.  This
completes the proof of Theorem~\ref{thmRegularSymbol}.
\end{proof}

%% file: sections/07-appendices.tex

\appendix

\section{Finite-time estimates for approximate flow lines}
\label{SecFiniteTimeODE}
\label{SecEstimates}

This appendix records a finite-time shadowing estimate for paths with small
Hamiltonian defect. It is used in the proof of positivity of the Reeb isolation modulus.

\subsection{The energy distance inequality}
The following proposition is a slight refinement of \cite[Proposition 3.5]{Hein} and the proof is taken from there. 
\begin{prp}[Finite-time shadowing estimate]\label{PrpDistEn0}
Let $V$ be a possibly time dependent vector field whose flow is complete on a Riemannian manifold $M$. Denoting the time $t$ flow by $\psi_t$ let $\lambda>0$ be a constant such that $\|d\psi_tv\|\ge\lambda^{-1}|v|$ for all tangent vectors $v$. Let $\gamma:[0,T]\to M$ be a smooth path. Then
\begin{equation}
    d(\gamma(0),\psi_T^{-1}(\gamma(T)))\le\lambda\int_0^T\|\gamma'(t)-V_t\circ\gamma(t)\|dt.
\end{equation}
\end{prp}
\begin{proof}
    Let $\eta(t):=\psi_t^{-1}(\gamma(t))$. Then $\dot{\gamma}(t)=d\psi_t\dot{\eta}(t)+V_t\circ\gamma(t)$. So,
    \begin{align*}
       \int_0^T\|\gamma'(t)-V_t\circ\gamma(t)\|dt&= \int_0^T\|d\psi_t\dot{\eta}(t)\|dt\\
       &\ge\lambda^{-1}\int_0^T\|\dot{\eta}(t)\|dt\\
       &\geq \lambda^{-1}d(\eta(0),\eta(T))\\
       &=\lambda^{-1}d(\gamma(0),\psi_T^{-1}(\gamma(T))).
    \end{align*}
\end{proof}

%% file: ref.bib
@article{AbouzaidGromanVarolgunes2024,
  title={Framed {$E_2$} structures in {F}loer theory},
  author={Abouzaid, Mohammed and Groman, Yoel and Varolgunes, Umut},
  journal={Advances in Mathematics},
  volume={450},
  pages={109755},
  year={2024},
  doi={10.1016/j.aim.2024.109755},
  eprint={2210.11027},
  archivePrefix={arXiv},
  primaryClass={math.SG}
}

@article{GinzburgGurel,
  title={Action and index spectra and periodic orbits in {H}amiltonian dynamics},
  author={Ginzburg, Viktor L and G{\"u}rel, Ba{\c{s}}ak Z},
  journal={Geometry \& Topology},
  volume={13},
  number={5},
  pages={2745--2805},
  year={2009},
  publisher={Mathematical Sciences Publishers}
}

@article{Hein,
  title={The {C}onley conjecture for irrational symplectic manifolds},
  author={Hein, Doris},
  journal={Journal of Symplectic Geometry},
  volume={10},
  number={2},
  pages={183--202},
  year={2012},
  publisher={International Press of Boston}
}

@article{BourgeoisOancea2009,
  title={Symplectic homology, autonomous Hamiltonians, and {Morse--Bott} moduli spaces},
  author={Bourgeois, Fr{\'e}d{\'e}ric and Oancea, Alexandru},
  journal={Duke Math. J.},
  volume={146},
  number={1},
  pages={71--174},
  year={2009}
}

@article{GromanVarolgunes2021,
  title={Locality of relative symplectic cohomology for complete embeddings},
  author={Groman, Yoel and Varolgunes, Umut},
  journal={Compositio Mathematica},
  volume={159},
  number={12},
  pages={2551--2637},
  year={2023},
  publisher={London Mathematical Society}
}

@article{groman,
  title={Floer theory and reduced cohomology on open manifolds},
  author={Groman, Yoel},
  journal={Geometry \& Topology},
  volume={27},
  number={4},
  pages={1273--1390},
  year={2023},
  doi={10.2140/gt.2023.27.1273}
}

@article{GromanVarolgunes2022,
  title={Closed string mirrors of symplectic cluster manifolds},
  author={Groman, Yoel and Varolgunes, Umut},
  journal={arXiv preprint arXiv:2211.07523},
  year={2022}
}

@article{Groman2025,
  title={Boundary Depth and Deformations of Symplectic Cohomology},
  author={Groman, Yoel},
  journal={arXiv preprint arXiv:2510.17607},
  year={2025}
}

@incollection{HoferSalamon,
  title={{F}loer homology and {N}ovikov rings},
  author={Hofer, Helmut and Salamon, Dietmar A},
  booktitle={The Floer memorial volume},
  pages={483--524},
  year={1995},
  publisher={Springer}
}

@book{HoferZhender,
  title={Symplectic invariants and Hamiltonian dynamics},
  author={Hofer, Helmut and Zehnder, Eduard},
  year={2012},
  publisher={Birkh{\"a}user}
}

@book{MS2,
  title={J-holomorphic curves and symplectic topology},
  author={McDuff, Dusa and Salamon, Dietmar},
  volume={52},
  year={2012},
  publisher={American Mathematical Soc.}
}

@article{Salamon1999,
  title={Lectures on {F}loer homology},
  author={Salamon, Dietmar},
  journal={Symplectic geometry and topology (Park City, UT, 1997)},
  volume={7},
  pages={143--229},
  year={1999},
  publisher={Citeseer}
}

@incollection {Sikorav94,
    AUTHOR = {Sikorav, Jean-Claude},
     TITLE = {Some properties of holomorphic curves in almost complex
              manifolds},
 BOOKTITLE = {Holomorphic curves in symplectic geometry},
    SERIES = {Progr. Math.},
    VOLUME = {117},
     PAGES = {165--189},
 PUBLISHER = {Birkh\"auser, Basel},
      YEAR = {1994},
   MRCLASS = {58E15 (53C15)},
  MRNUMBER = {1274929},
}

@article{Mclean2012,
  title={Local {F}loer homology and infinitely many simple {R}eeb orbits},
  author={McLean, Mark},
  journal={Algebraic \& Geometric Topology},
  volume={12},
  number={4},
  pages={1901--1923},
  year={2012},
  publisher={Mathematical Sciences Publishers}
}

@article{varolgunes,
  title={Mayer--{V}ietoris property for relative symplectic cohomology},
  author={Varolgunes, Umut},
  journal={Geometry \& Topology},
  volume={25},
  number={2},
  pages={547--642},
  year={2021},
  publisher={Mathematical Sciences Publishers}
}

@phdthesis{SchmaeschkeThesis,
  author = {Schm{\"a}schke, Felix},
  title  = {Abelianization and Floer Homology of Lagrangians in Clean Intersection},
  school = {Universit{\"a}t Leipzig},
  year   = {2016}
}

@book{SeidelBook,
  title={Fukaya categories and Picard-Lefschetz theory},
  author={Seidel, Paul},
  volume={10},
  year={2008},
  publisher={European Mathematical Society}
}

@inproceedings{Pozniak,
  title={Floer homology, Novikov rings and clean intersections},
  author={Po{\'z}niak, Marcin},
  booktitle={Northern California Symplectic Geometry Seminar},
  pages={119--181},
  year={1999},
  organization={American Mathematical Society}
}
